\documentclass[11pt]{amsart}

\usepackage{
    amsmath,
    amsfonts,
    amssymb,
    amsthm,
    amscd,
    comment,
    enumitem,
    etoolbox,
    textcomp,   % To avoid warning in gensymb
    gensymb,    % For \degree
    mathtools,
    mathdots,
    stmaryrd    % For \llbracket and \rrbracket
}
\usepackage[usenames,dvipsnames]{xcolor}
\usepackage[all]{xy}

\makeatletter
\@namedef{subjclassname@2020}{%
  \textup{2020} Mathematics Subject Classification}
\makeatother

\usepackage[T1]{fontenc}
\usepackage{bbm}                     % For \mathbbm{1} (unit in monoidal category)
\usepackage[colorlinks=true, linkcolor=blue, citecolor=blue, urlcolor=blue, breaklinks=true]{hyperref}

\DeclareFontFamily{OT1}{pzc}{}
\DeclareFontShape{OT1}{pzc}{m}{it}{<-> s * [1.10] pzcmi7t}{}
\DeclareMathAlphabet{\mathpzc}{OT1}{pzc}{m}{it}

\usepackage[noabbrev,capitalize]{cleveref}

\crefname{cor}{Corollary}{Corollaries}
\crefname{defin}{Definition}{Definitions}
\crefname{eg}{Example}{Examples}
\crefname{lem}{Lemma}{Lemmas}
\crefname{prop}{Proposition}{Propositions}
\crefname{rem}{Remark}{Remarks}
\crefname{theo}{Theorem}{Theorems}
\crefname{equation}{}{}
\crefname{enumi}{}{}
\renewcommand{\descriptionlabel}[1]{\hspace{\labelsep}(#1)}

\makeatletter
\let\orgdescriptionlabel\descriptionlabel
\renewcommand*{\descriptionlabel}[1]{%
  \let\orglabel\label
  \let\label\@gobble
  \phantomsection
  \edef\@currentlabel{#1}%
  \let\label\orglabel
  \orgdescriptionlabel{#1}%
}
\makeatother

\newcommand\C{\mathbb{C}}
\newcommand\N{\mathbb{N}}

\newcommand\Q{\mathbb{Q}}

\newcommand\Z{\mathbb{Z}}
\newcommand\kk{\Bbbk}
\newcommand\one{\mathbbm{1}}

\newcommand\bB{\mathbf{B}}

\newcommand\bx{\mathbf{x}}
\newcommand\by{\mathbf{y}}
\newcommand\bz{\mathbf{z}}
\newcommand\bZ{\mathbf{Z}}

\newcommand\cC{\mathcal{C}}
\newcommand\cD{\mathcal{D}}

\newcommand\cI{\mathcal{I}}
\newcommand\cM{\mathcal{M}}
\newcommand\cP{\mathcal{P}}

\newcommand\fg{\mathfrak{g}}
\newcommand\fS{\mathfrak{S}}            % Symmetric group

\newcommand\op{\mathrm{op}}

\newcommand{\pmd}{\textup{-pmod}}

\newcommand\Laurent[1]{%    For Laurent series
    (\!(#1)\!)
}

\newcommand{\xrightarrowdbl}[2][]{%
  \xrightarrow[#1]{#2}\mathrel{\mkern-14mu}\rightarrow
}

\newcommand\cEnd{\mathpzc{End}}         % Category of endofunctors
\newcommand\cH{\mathpzc{H}}             % Cyclotomic quantum Heisenberg category
\newcommand\Heis{\mathpzc{Heis}}        % Heisenberg category
\newcommand\FOT{\mathpzc{FOT}}          % Category of framed oriented tangles
\newcommand\OS{\mathpzc{OS}}            % Oriented skein category
\newcommand\cU{\mathpzc{U}}             % Kac--Moody 2-category

\newcommand\AH{\mathrm{AH}}             % Affine Hecke algebra
\newcommand\BS{\mathrm{BS}}             % Elliptic Hall algebra of Burban-Schiffmann
\newcommand\BSc{\widetilde{\BS}}        % Central extension of \BS
\newcommand\EH{\mathrm{EH}}             % Elliptic Hall associative algebra
\newcommand\EHc{\widetilde{\EH}}        % Central extension of Elliptic Hall associated algebra
\newcommand\fEH{\mathfrak{EH}}          % Elliptic Hall Lie algebra
\newcommand\fEHc{\widetilde{\fEH}}      % Central extension of Elliptic Hall Lie algebra
\newcommand\rH{\mathrm{H}}              % Hecke algebra
\newcommand\rHeis{\mathrm{Heis}}        % Heisenberg algebra
\newcommand\Sk{\mathrm{Sk}}             % Skein algebra

\newcommand\qint[1]{\{#1\}}

\DeclareMathOperator{\Add}{Add}

\DeclareMathOperator{\End}{End}

\DeclareMathOperator{\GL}{GL}
\DeclareMathOperator{\Hom}{Hom}

\DeclareMathOperator{\id}{id}
\DeclareMathOperator{\Kar}{Kar}

\DeclareMathOperator{\Ob}{Ob}

\DeclareMathOperator{\Span}{span}
\DeclareMathOperator{\Sym}{Sym}
\DeclareMathOperator{\Tr}{Tr}           % (Vertical) trace of a category

\usepackage{tikz}
\usetikzlibrary{arrows,patterns}
\usepackage{tikz-cd}

\tikzset{anchorbase/.style={>=To,baseline={([yshift=-0.5ex]current bounding box.center)}}}
\tikzset{ % Syntax: \begin{tikzpicture}[centerzero={0,0.2}]
    centerzero/.style={>=To,baseline={([yshift=-0.5ex](#1))}},
    centerzero/.default={0,0}
}
\tikzset{wipe/.style={white,line width=4pt}}

\newcommand{\dotlabel}[1]{$\scriptstyle{#1}$}
\newcommand{\braidup}{to[out=up,in=down]}
\newcommand{\braiddown}{to[out=down,in=up]}
\newcommand{\singdot}[1]{\filldraw[fill=white,draw=black] (#1) circle (0.06)}
\newcommand{\multdot}[3]{% \multdot{position}{anchor}{label}
    \filldraw[fill=white, draw=black] (#1) circle (0.06) node[anchor=#2] {\dotlabel{#3}}
}
\newcommand{\bubrightblank}[1]{% \bubrightblank{position}
    \draw[->] (#1)+(0.2,0) arc(360:0:0.2)
}
\newcommand{\bubleftblank}[1]{% \bubleftblank{position}
    \draw[->] (#1)+(-0.2,0) arc(-180:180:0.2)
}
\newcommand{\bubright}[2]{% \bubright{position}{dot}
    \bubrightblank{#1};
    \filldraw[fill=white, draw=black] (#1)+(0,-0.2) circle (0.06) node[anchor=north] {\dotlabel{#2}}
}
\newcommand{\bubrightside}[2]{% \bubrightside{position}{dot}
    \bubrightblank{#1};
    \filldraw[fill=white, draw=black] (#1)+(-0.2,0) circle (0.06) node[anchor=east] {\dotlabel{#2}}
}
\newcommand{\bubleft}[2]{% \bubleft{position}{dot}
    \bubleftblank{#1};
    \filldraw[fill=white, draw=black] (#1)+(0,-0.2) circle (0.06) node[anchor=north] {\dotlabel{#2}}
}
\newcommand{\bubleftside}[2]{% \bubleftside{position}{dot}
    \bubleftblank{#1};
    \filldraw[fill=white, draw=black] (#1)+(0.2,0) circle (0.06) node[anchor=west] {\dotlabel{#2}}
}
\newcommand{\plusright}[2]{% \plusright{position}{dot}
    \bubright{#1}{#2};
    \node at (#1) {\dotlabel{+}}
}
\newcommand{\plusrightside}[2]{% \plusrightside{position}{dot}
    \bubrightside{#1}{#2};
    \node at (#1) {\dotlabel{+}}
}
\newcommand{\plusrightgen}[1]{% \plusrightgen{position}
    \bubrightblank{#1};
    \node at (#1) {\dotlabel{+}};
    \draw (#1)+(0.2,0) node[anchor=west] {$(u)$}
}
\newcommand{\plusleft}[2]{% \plusleft{position}{dot}
    \bubleft{#1}{#2};
    \node at (#1) {\dotlabel{+}}
}
\newcommand{\plusleftside}[2]{% \plusleftsideposition}{dot}
    \bubleftside{#1}{#2};
    \node at (#1) {\dotlabel{+}}
}
\newcommand{\plusleftgen}[1]{% \plusleftgen{position}
    \bubleftblank{#1};
    \node at (#1) {\dotlabel{+}};
    \draw (#1)+(0.2,0) node[anchor=west] {$(u)$}
}
\newcommand{\minusright}[2]{% \minusright{position}{dot}
    \bubright{#1}{#2};
    \node at (#1) {\dotlabel{-}}
}
\newcommand{\minusrightside}[2]{% \minusrightside{position}{dot}
    \bubrightside{#1}{#2};
    \node at (#1) {\dotlabel{-}}
}
\newcommand{\minusrightgen}[1]{% \minusrightblank{position}
    \bubrightblank{#1};
    \node at (#1) {\dotlabel{-}};
    \draw (#1)+(0.2,0) node[anchor=west] {$(u)$}
}
\newcommand{\minusleft}[2]{% \minusleft{position}{dot}
    \bubleft{#1}{#2};
    \node at (#1) {\dotlabel{-}}
}
\newcommand{\minusleftside}[2]{% \minusleftside{position}{dot}
    \bubleftside{#1}{#2};
    \node at (#1) {\dotlabel{-}}
}
\newcommand{\minusleftgen}[1]{% \minusleftgen{position}
    \bubleftblank{#1};
    \node at (#1) {\dotlabel{-}};
    \draw (#1)+(0.2,0) node[anchor=west] {$(u)$}
}
\newcommand{\pmright}[2]{% \pmright{position}
    \bubright{#1}{#2};
    \node at (#1) {\dotlabel{\pm}};
}
\newcommand{\pmleft}[2]{% \pmright{position}
    \bubleft{#1}{#2};
    \node at (#1) {\dotlabel{\pm}};
}
\newcommand{\pmrightgen}[1]{% \pmrightgen{position}
    \bubrightblank{#1};
    \node at (#1) {\dotlabel{\pm}};
    \draw (#1)+(0.2,0) node[anchor=west] {$(u)$}
}
\newcommand{\pmleftgen}[1]{% \pmrightgen{position}
    \bubleftblank{#1};
    \node at (#1) {\dotlabel{\pm}};
    \draw (#1)+(0.2,0) node[anchor=west] {$(u)$}
}
\newcommand{\bubsym}[2]{% \bubsym{position}{function}
    \draw (#1)+(0.2,0) arc(360:0:0.2);
    \node at (#1) {\dotlabel{#2}}
}
\newcommand{\identify}[4]{% \identify{left}{bottom}{right}{top}
    \draw[red,dashed] (#1,#2) -- (#1,#4);
    \draw[red,dashed] (#3,#2) -- (#3,#4)
}

\newcommand{\posupcross}{
    \begin{tikzpicture}[centerzero]
        \draw[->] (0.2,-0.2) -- (-0.2,0.2);
        \draw[wipe] (-0.2,-0.2) -- (0.2,0.2);
        \draw[->] (-0.2,-0.2) -- (0.2,0.2);
    \end{tikzpicture}
}
\newcommand{\negupcross}{
    \begin{tikzpicture}[centerzero]
        \draw[->] (-0.2,-0.2) -- (0.2,0.2);
        \draw[wipe] (0.2,-0.2) -- (-0.2,0.2);
        \draw[->] (0.2,-0.2) -- (-0.2,0.2);
    \end{tikzpicture}
}
\newcommand{\posrightcross}{
    \begin{tikzpicture}[centerzero]
        \draw[->] (-0.2,-0.2) -- (0.2,0.2);
        \draw[wipe] (0.2,-0.2) -- (-0.2,0.2);
        \draw[<-] (0.2,-0.2) -- (-0.2,0.2);
    \end{tikzpicture}
}
\newcommand{\negrightcross}{
    \begin{tikzpicture}[centerzero]
        \draw[<-] (0.2,-0.2) -- (-0.2,0.2);
        \draw[wipe] (-0.2,-0.2) -- (0.2,0.2);
        \draw[->] (-0.2,-0.2) -- (0.2,0.2);
    \end{tikzpicture}
}
\newcommand{\posdowncross}{
    \begin{tikzpicture}[centerzero]
        \draw[<-] (0.2,-0.2) -- (-0.2,0.2);
        \draw[wipe] (-0.2,-0.2) -- (0.2,0.2);
        \draw[<-] (-0.2,-0.2) -- (0.2,0.2);
    \end{tikzpicture}
}
\newcommand{\negdowncross}{
    \begin{tikzpicture}[centerzero]
        \draw[<-] (-0.2,-0.2) -- (0.2,0.2);
        \draw[wipe] (0.2,-0.2) -- (-0.2,0.2);
        \draw[<-] (0.2,-0.2) -- (-0.2,0.2);
    \end{tikzpicture}
}
\newcommand{\posleftcross}{
    \begin{tikzpicture}[centerzero]
        \draw[<-] (-0.2,-0.2) -- (0.2,0.2);
        \draw[wipe] (0.2,-0.2) -- (-0.2,0.2);
        \draw[->] (0.2,-0.2) -- (-0.2,0.2);
    \end{tikzpicture}
}
\newcommand{\negleftcross}{
    \begin{tikzpicture}[centerzero]
        \draw[->] (0.2,-0.2) -- (-0.2,0.2);
        \draw[wipe] (-0.2,-0.2) -- (0.2,0.2);
        \draw[<-] (-0.2,-0.2) -- (0.2,0.2);
    \end{tikzpicture}
}

\newcommand{\upstrand}{
    \begin{tikzpicture}[centerzero]
        \draw[->] (0,-0.2) -- (0,0.2);
    \end{tikzpicture}
}
\newcommand{\downstrand}{
    \begin{tikzpicture}[centerzero]
        \draw[<-] (0,-0.2) -- (0,0.2);
    \end{tikzpicture}
}
\newcommand{\updot}{
    \begin{tikzpicture}[centerzero]
        \draw[->] (0,-0.2) -- (0,0.2);
        \singdot{0,0};
    \end{tikzpicture}
}
\newcommand{\downdot}{
    \begin{tikzpicture}[centerzero]
        \draw[->] (0,0.2) -- (0,-0.2);
        \singdot{0,0};
    \end{tikzpicture}
}
\newcommand{\multupdot}[2][west]{ % \multupdot[anchor]{dot}
    \begin{tikzpicture}[centerzero]
        \draw[->] (0,-0.2) -- (0,0.2);
        \multdot{0,0}{#1}{#2};
    \end{tikzpicture}
}
\newcommand{\multdowndot}[2][west]{ % \multdowndot[anchor]{dot}
    \begin{tikzpicture}[centerzero]
        \draw[<-] (0,-0.2) -- (0,0.2);
        \multdot{0,0}{#1}{#2};
    \end{tikzpicture}
}

\newcommand{\rightcup}{
    \begin{tikzpicture}[anchorbase]
        \draw[->] (-0.15,0.15) -- (-0.15,0) arc(180:360:0.15) -- (0.15,0.15);
    \end{tikzpicture}
}
\newcommand{\leftcup}{
    \begin{tikzpicture}[anchorbase]
        \draw[<-] (-0.15,0.15) -- (-0.15,0) arc(180:360:0.15) -- (0.15,0.15);
    \end{tikzpicture}
}
\newcommand{\rightcap}{
    \begin{tikzpicture}[anchorbase]
        \draw[->] (-0.15,-0.15) -- (-0.15,0) arc(180:0:0.15) -- (0.15,-0.15);
    \end{tikzpicture}
}
\newcommand{\leftcap}{
    \begin{tikzpicture}[anchorbase]
        \draw[<-] (-0.15,-0.15) -- (-0.15,0) arc(180:0:0.15) -- (0.15,-0.15);
    \end{tikzpicture}
}

\newcommand{\rightbubside}[1]{
    \begin{tikzpicture}[centerzero]
        \bubrightside{0,0}{#1};
    \end{tikzpicture}
}
\newcommand{\rightplus}[1]{
    \begin{tikzpicture}[centerzero]
        \plusright{0,0}{#1};
    \end{tikzpicture}
}
\newcommand{\rightplusside}[1]{
    \begin{tikzpicture}[centerzero]
        \plusrightside{0,0}{#1};
    \end{tikzpicture}
}
\newcommand{\rightplusgen}{
    \begin{tikzpicture}[centerzero]
        \plusrightgen{0,0};
    \end{tikzpicture}
}
\newcommand{\rightplusblank}{
    \begin{tikzpicture}[centerzero]
        \draw[->] (0.2,0) arc(360:0:0.2);
        \node at (0,0) {\dotlabel{+}};
    \end{tikzpicture}
}
\newcommand{\rightminus}[1]{
    \begin{tikzpicture}[centerzero]
        \minusright{0,0}{#1};
    \end{tikzpicture}
}
\newcommand{\rightminusside}[1]{
    \begin{tikzpicture}[centerzero]
        \minusrightside{0,0}{#1};
    \end{tikzpicture}
}
\newcommand{\rightminusgen}{
    \begin{tikzpicture}[centerzero]
        \minusrightgen{0,0};
    \end{tikzpicture}
}
\newcommand{\rightminusblank}{
    \begin{tikzpicture}[centerzero]
        \draw[->] (0.2,0) arc(360:0:0.2);
        \node at (0,0) {\dotlabel{-}};
    \end{tikzpicture}
}
\newcommand{\leftbub}[1]{
    \begin{tikzpicture}[centerzero]
        \bubleft{0,0}{#1};
    \end{tikzpicture}
}
\newcommand{\leftbubside}[1]{
    \begin{tikzpicture}[centerzero]
        \bubleftside{0,0}{#1};
    \end{tikzpicture}
}
\newcommand{\leftplus}[1]{
    \begin{tikzpicture}[centerzero]
        \plusleft{0,0}{#1};
    \end{tikzpicture}
}
\newcommand{\leftplusside}[1]{
    \begin{tikzpicture}[centerzero]
        \plusleftside{0,0}{#1};
    \end{tikzpicture}
}
\newcommand{\leftplusgen}{
    \begin{tikzpicture}[centerzero]
        \plusleftgen{0,0};
    \end{tikzpicture}
}
\newcommand{\leftplusblank}{
    \begin{tikzpicture}[centerzero]
        \bubleftblank{0,0};
        \node at (0,0) {\dotlabel{+}};
    \end{tikzpicture}
}
\newcommand{\leftminus}[1]{
    \begin{tikzpicture}[centerzero]
        \minusleft{0,0}{#1};
    \end{tikzpicture}
}
\newcommand{\leftminusside}[1]{
    \begin{tikzpicture}[centerzero]
        \minusleftside{0,0}{#1};
    \end{tikzpicture}
}
\newcommand{\leftminusgen}{
    \begin{tikzpicture}[centerzero]
        \minusleftgen{0,0};
    \end{tikzpicture}
}
\newcommand{\leftminusblank}{
    \begin{tikzpicture}[centerzero]
        \bubleftblank{0,0};
        \node at (0,0) {\dotlabel{-}};
    \end{tikzpicture}
}
\newcommand{\rightpm}[1]{
    \begin{tikzpicture}[centerzero]
        \pmright{0,0}{#1};
    \end{tikzpicture}
}
\newcommand{\leftpm}[1]{
    \begin{tikzpicture}[centerzero]
        \pmleft{0,0}{#1};
    \end{tikzpicture}
}

\newcommand{\symbub}[1]{
    \begin{tikzpicture}[centerzero]
        \bubsym{0,0}{#1};
    \end{tikzpicture}
}

\newtheorem{theo}{Theorem}[section]

\newtheorem{prop}[theo]{Proposition}
\newtheorem{lem}[theo]{Lemma}
\newtheorem{cor}[theo]{Corollary}

\theoremstyle{definition}
\newtheorem{defin}[theo]{Definition}
\newtheorem{rem}[theo]{Remark}

\numberwithin{equation}{section}
\allowdisplaybreaks

\setenumerate[1]{label=(\alph*)}          % Only need to reference enumerate items with \ref{}

\newtoggle{comments}
\newtoggle{details}
\newtoggle{detailsnote}

\iftoggle{comments}{%
  \newcommand{\acomments}[1]{
    \ \\
    {\color{red}
      \textbf{AS:} #1
    }
    \ \\
    }
  \newcommand{\ycomments}[1]{
    \ \\
    {\color{red}
      \textbf{YM:} #1
    }
    \ \\
    }
}{%
  \newcommand{\acomments}[1]{}
  \newcommand{\ycomments}[1]{}
}

\iftoggle{details}{%
  \newcommand{\details}[1]{
      \ \\
      {\color{OliveGreen}
        \textbf{Details:} #1
      }
      \\
  }
}{%
  \newcommand{\details}[1]{}
}

\begin{document}
%===============

\title{Categorification of the elliptic Hall algebra}

\author{Youssef Mousaaid}
\address[Y.M.]{
  Department of Mathematics and Statistics \\
  University of Ottawa \\
  Ottawa, ON K1N 6N5, Canada
}
\email{ymous016@uottawa.ca}

\author{Alistair Savage}
\address[A.S.]{
  Department of Mathematics and Statistics \\
  University of Ottawa \\
  Ottawa, ON K1N 6N5, Canada
}
\urladdr{\href{https://alistairsavage.ca}{alistairsavage.ca}, \textrm{\textit{ORCiD}:} \href{https://orcid.org/0000-0002-2859-0239}{orcid.org/0000-0002-2859-0239}}
\email{alistair.savage@uottawa.ca}

\begin{abstract}
    We show that the central charge $k$ reduction of the universal central extension of the elliptic Hall algebra is isomorphic to the trace, or zeroth Hochschild homology, of the quantum Heisenberg category of central charge $k$.  As an application, we construct large families of representations of the universal extension of the elliptic Hall algebra.
\end{abstract}



\ifboolexpr{togl{comments} or togl{details}}{%
  {\color{magenta}DETAILS OR COMMENTS ON}
}{%
}

\maketitle
\thispagestyle{empty}

\tableofcontents

%=====================
\section{Introduction}
%=====================

The \emph{elliptic Hall algebra} associated to a smooth elliptic curve $X$ over a finite field is the Drinfeld double of the Hall algebra of the category of coherent sheaves over $X$.  In \cite{BS12}, Burban and Schiffmann gave an explicit realization of a \emph{generic} elliptic Hall algebra $\EH$, depending on two formal parameters $\sigma, \bar{\sigma}$, which specializes to the elliptic Hall algebra for \emph{any} $X$.  The importance of the algebra $\EH$ is underlined by the fact that versions of it (more precisely, its ``positive half'' or central extensions) have appeared in many different contexts under different names: a \emph{generalized quantum affine algebra} \cite{DI97}, a \emph{$(q,\gamma)$-analogue of the $W_{1+\infty}$ algebra} \cite{Mik07}, the \emph{shuffle algebra} \cite{FT11,Neg14}, the \emph{spherical $\mathfrak{gl}_\infty$ double affine Hecke algebra} \cite{SV11,FFJMM11}, and the \emph{quantum continuous $\mathfrak{gl}_\infty$} \cite{FFJMM11}.  It is also intimately related to the equivariant $K$-theory of the Hilbert scheme of points on $\mathbb{A}^2$ \cite{SV13,FFJMM11,FT11,Neg15}.  In this paper we show that the elliptic Hall algebra is categorified by the \emph{quantum Heisenberg category} defined in \cite{BSW-qheis}.  We then use this categorification to construct large families of representations of central extensions of the elliptic Hall algebra.

Let us explain our results in more detail.  We first show that the elliptic Hall algebra has a universal central extension $\EHc$ by a two-dimensional center (\cref{uce}).  Then, to any \emph{central charge} $k \in \Z$, one can define a natural central reduction $\EH_k$.  In fact, \emph{every} central reduction of $\EHc$ is isomorphic to $\EH_k$ for some $k$ (\cref{violet}).

To this same central charge, one can associate a quantum Heisenberg category $\Heis_k$ as in \cite{BSW-qheis}.  This is a strict $\kk$-linear pivotal monoidal category modelled on the affine Hecke algebras of type $A$.  When $k \ne 0$, it acts naturally on the category of modules for cyclotomic Hecke algebras of level $|k|$.  When $k = -1$, it extends an earlier $q$-deformed Heisenberg category introduced in \cite{LS13}.  On the other hand, when $k = 0$, it is the framed HOMFLYPT skein category over the annulus and it acts naturally on the category of modules for $U_q(\mathfrak{gl}_n)$.

The \emph{trace}, or \emph{zeroth Hochschild homology}, of a small $\kk$-linear category is the $\kk$-module
\[
    \Tr(\cC) := \left( \bigoplus_{X \in \cC} \End_\cC(X) \right) / \Span_\kk\{f \circ g - g \circ f\},
\]
where $f$ and $g$ run through all pairs of morphisms $f \colon X \to Y$ and $g \colon Y \to X$ in $\cC$.  The trace can be thought of as a categorical analogue of the cocenter of an algebra.  If $\cC$ is monoidal, then $\Tr(\cC)$ is naturally an associative $\kk$-algebra.  The main result of the current paper (\cref{mainthm}) is that there is an isomorphism of algebras
\begin{equation} \label{gem}
    \EH_k \xrightarrow{\cong} \Tr(\Heis_k).
\end{equation}
This isomorphism is given explicitly, by specifying the images of the elements of a natural basis for $\EH_k$.

When $k=0$, $\Tr(\Heis_0)$ is isomorphic to the skein algebra of the torus.  This skein algebra was identified with $\EH = \EH_0$ by Morton and Samuelson \cite{MS17}.  On the other hand, when $k=-1$, the $q$-deformed Heisenberg category of \cite{LS13} was identified with the positive half of $\EH_{-1}$ by Cautis, Lauda, Licata, Samuelson, and Sussan \cite{CLLSS18}.  This corresponds to the fact that the $q$-deformed Heisenberg category can be viewed as ``half'' of the quantum Heisenberg category $\Heis_{-1}$.  (See \cref{chess}.)  In some sense, $\Tr(\Heis_k)$ can be thought of as a deformation of the skein algebra of the torus, depending on the central charge $k$, that breaks the symmetry between the two directions.  This central charge deformation allows us to categorify arbitrary central reductions $\EH_k$.

The split Grothendieck ring $K_0(\Heis_k)$ of the quantum Heisenberg category is conjecturally isomorphic to the central charge $k$ reduction $\rHeis_k$ of the universal enveloping algebra of the infinite rank Heisenberg Lie algebra.  (The corresponding statement for the \emph{degenerate} Heisenberg category has been proved; see \cite[Th.~1.1]{BSW-K0}.)  The Chern character map gives a homomorphism $K_0(\Heis_k) \to \Tr(\Heis_k)$.  Assuming the aforementioned conjecture, this corresponds to a natural inclusion $\rHeis_k \hookrightarrow \EH_k$.  (See \cref{Heisenberg}.)

One immediate application of our categorification of $\EH_k$ is that we obtain a large number of representations of this algebra.  The first family of representations arises from the fact that the trace of a linear pivotal category acts naturally on its center, which is the endomorphism algebra of the unit object.  For the quantum Heisenberg category, the center is isomorphic to $\Sym \otimes \Sym$, where $\Sym$ is the algebra of symmetric functions.  Thus, we obtain a natural family of actions of $\EH_k$ on $\Sym \otimes \Sym$ depending on a parameter $t$ in the ground ring.  This generalizes an action of $\EH = \EH_0$ on $\Sym \otimes \Sym$ described in \cite[\S4]{MS17}, corresponding to the action of the skein algebra of the torus acting on the skein of the annulus.

The second family of representations emerges from the natural action of $\Heis_k$ on the category of modules for cyclotomic Hecke algebras.  Passing to traces, this yields an action of $\EH_k$ on the cocenters of cyclotomic Hecke algebras.  We expect these actions to be related to the geometry of moduli spaces of framed torsion-free sheaves on $\mathbb{P}^2$, extending work of Schiffmann and Vasserot \cite{SV13}.  (See \cref{glow}.)

We also expect that the results of the current paper can be generalized by incorporating a Frobenius superalgebra.  More precisely, to every Frobenius superalgebra $A$ and central charge $k \in \Z$, there is a \emph{quantum Frobenius Heisenberg category}, introduced in \cite{BSW-qFrobHeis}.  The trace of this category should be isomorphic to a Frobenius algebra generalization of $\EH_k$.  In the degenerate setting, the trace of the Frobenius Heisenberg category was related to a Frobenius algebra generalization of the $W$-algebra $W_{1+\infty}$ in \cite{RS20}.

%-----------------------------
\subsection*{Acknowledgements}
%-----------------------------

This research was supported by Discovery Grant RGPIN-2017-03854 from the Natural Sciences and Engineering Research Council of Canada.  The authors would like to thank J.~Brundan, D.~Ciubotaru, A.~Licata, P.~Samuelson, and E.~Vasserot for helpful conversations.

%=======================================================================
\section{Central extensions of the elliptic Hall algebra\label{sec:EHA}}
%=======================================================================

In this section we introduce our main algebra of interest, which is a specialization of a central extension of the elliptic Hall algebra of Burban and Schiffmann \cite{BS12}.  We give here a direct description of this algebra, and explain the connection to the algebra of Burban and Schiffmann (which is not needed for the results of the current paper) in \cref{sec:BS}.

In this section, unless otherwise specified, we work over an arbitrary integral domain $\kk$, and we fix $q \in \kk^\times$ such that
\begin{equation} \label{trevor}
    \qint{d} := q^d - q^{-d} \in \kk^\times \text{ for all } d \ne 0.
\end{equation}
Thus the most generic choice is $\kk = \Z[q^{\pm 1}, \{d\}^{-1} : d \ge 1]$.  All algebras and tensor products are over $\kk$ unless otherwise indicated.

%---------------------------------------
\subsection{Universal central extension}
%---------------------------------------

Let
\begin{gather*}
    \bZ := \Z^2, \quad
    \bZ^* := \bZ \setminus \{(0,0)\},
    \\
    \bZ^+ := \{(r,n) \in \bZ : n > 0 \text{ or } n=0,\, r>0\},\quad
    \bZ^- := - \bZ^+.
\end{gather*}
Let $\fEH$ be the Lie algebra over $\kk$ with basis $w_\bx$, $\bx \in \bZ^*$, and Lie bracket given by
\begin{equation} \label{sloth}
    [w_\bx, w_\by] = \qint{d} w_{\bx + \by},
    \quad \text{where } d = \det \begin{pmatrix} \bx & \by \end{pmatrix}.
\end{equation}
Here $\begin{pmatrix} \bx & \by \end{pmatrix}$ denotes the $2 \times 2$ matrix with columns $\bx$ and $\by$.  We will write $w_{r,n}$ for $w_{(r,n)}$, and we adopt the convention that $w_{0,0}=0$.  It is a straightforward computation to verify that \cref{sloth} satisfies the axioms of a Lie bracket.  It is also not hard to see that $\fEH$ is perfect, that is, $[\fEH,\fEH] = \fEH$.  Thus, $\fEH$ has a universal central extension, which we now describe.

Let $\bZ_\kk := \kk \otimes_\Z \bZ \cong \kk^2$.  It is straightforward to check that the $\kk$-bilinear map
\[
    \fEH \times \fEH \to \bZ_\kk,\quad
    (w_\bx,w_\by) \mapsto \delta_{\bx,-\by} \bx,
\]
is a 2-cocycle, where we view $\bZ_\kk$ as a trivial $\fEH$-module.  Let $\fEHc$ be the corresponding central extension.  Thus, $\fEHc = \fEH \oplus \bZ_\kk$ as $\kk$-modules, with Lie bracket given by the fact that the elements of $\bZ_\kk$ are central and
\begin{equation} \label{phoenix}
    [w_\bx, w_\by] = \qint{d} w_{\bx + \by} + \delta_{\bx,-\by} \bx,
    \quad \text{where } d = \det \begin{pmatrix} \bx & \by \end{pmatrix}.
\end{equation}

\begin{prop} \label{uce}
    If $q$ is not integral over the canonical image of $\Z$ in $\kk$, then the Lie algebra $\fEHc$ is the universal central extension of $\fEH$.
\end{prop}

Since \cref{uce} is not directly used elsewhere in the paper, and its verification is somewhat lengthy, we have relegated the proof to \cref{sec:uce}.  We do not know if the genericity assumption on $q$ is necessary.

\begin{cor}
    The second cohomology module $\mathrm{H}^2(\fEH;\kk)$ has rank two, with basis given by the classes of the two cocycles $\fEH \times \fEH \to \kk$ defined by
    \[
        (w_\bx,w_\by) \mapsto \delta_{\bx,-\by}r,
        \qquad \text{and} \qquad
        (w_\bx,w_\by) \mapsto \delta_{\bx,-\by}n,
    \]
    for $\bx = (r,n) \in \bZ^*$, $\by \in \bZ^*$.
\end{cor}

\begin{rem} \label{Heisenberg}
    The Lie algebra $\fEHc$ contains a copy of the infinite rank Heisenberg algebra for every rank one sublattice of $\bZ$.  More precisely, for $\bx \in \bZ^*$, we have
    \begin{equation}
        [w_{i\bx}, w_{j\bx}] = i \delta_{i,-j} \bx,\quad i,j \in \Z \setminus \{0\},
    \end{equation}
    and so $\Span_\kk \{w_{i\bx}, \bx : i \in \Z \setminus \{0\}\}$ is an infinite rank Heisenberg algebra with central element $\bx$.
\end{rem}

%---------------------------------------------------
\subsection{Central reductions\label{subsec:reduct}}
%---------------------------------------------------

Let $\EHc$ be the universal enveloping algebra of $\fEHc$.  For a $\Z$-linear map $\lambda \colon \bZ \to \Z$, define the corresponding central reduction
\begin{equation} \label{pendulum}
    \EH_{\lambda} = \EHc/ \langle \bx - \lambda(\bx) : \bx \in \bZ \rangle.
\end{equation}
For $k \in \Z$, let
\[
    \lambda_k \colon \bZ \to \Z,\quad (r,n) \mapsto kn,
\]
and define $\EH_k := \EH_{\lambda_k}$.  Thus, $\EH_k$ is the associative $\kk$-algebra generated by $w_\bx$, $\bx \in \bZ^*$, and relations
\begin{equation}
    [w_\bx, w_\by] = \{d\} w_{\bx+\by} + kn \delta_{\bx,-\by},\quad
    \text{where } d = \det \begin{pmatrix} \bx & \by \end{pmatrix},\ \bx = (r,n).
\end{equation}
We will denote the image of $w_\bx$ in $\EH_k$ again by $w_\bx$.

The integral general linear group $\GL_2(\Z)$ acts on $\EHc$ by $\kk$-algebra automorphisms via
\begin{equation} \label{rotate}
    \EHc \xrightarrow{\cong} \EHc,\qquad
    w_\bx \mapsto \det(\gamma) w_{\gamma \bx},\quad
    \bx \mapsto \gamma \bx,\qquad
    \bx \in \bZ,\ \gamma \in \GL_2(\Z).
\end{equation}
For a $\Z$-linear map $\lambda \colon \bZ \to \Z$, this induces isomorphisms
\begin{equation} \label{icy}
    \EH_\lambda \xrightarrow{\cong} \EH_{\lambda \gamma^{-1}},\qquad
    \gamma \in \GL_2(\Z).
\end{equation}
The importance of the $\EH_k$ is given by the following result, which says that every central reduction is isomorphic to some $\EH_k$.

\begin{prop} \label{violet}
    For every $\Z$-linear map $\lambda \colon \bZ \to \Z$, there exists $k \in \Z$ such that $\EH_\lambda \cong \EH_k$ as algebras.
\end{prop}

\begin{proof}
    Let $\lambda \colon \bZ \to \Z$ be a $\Z$-linear map.  Thus, there exist $a,b \in \Z$ such that $\lambda(r,n) = ar + bn$.  Let $k = \gcd(a,b)$, and choose $c,d \in \Z$ such that $ac+bd=k$.  Define
    \[
        \gamma =
        \begin{pmatrix}
            d & -c \\
            a/k & b/k
        \end{pmatrix}
        \in \GL_2(\Z),\qquad
        \text{so that} \quad
        \gamma^{-1} =
        \begin{pmatrix}
            b/k & c \\
            -a/k & d
        \end{pmatrix}
        .
    \]
    Then we have $\lambda \gamma^{-1} = \lambda_k$, and the result follows from \cref{icy}.
\end{proof}

\begin{rem}
    Note that $\EH_0 \cong U(\fEH)$.  Furthermore, by \cite[Th.~2, Th.~3]{MS17}, $\EH_0$ is isomorphic to the elliptic Hall algebra of \cite{BS12}, specialized at $\bar{\sigma} = q^2 = \sigma^{-1}$.
\end{rem}

\begin{rem} \label{tensor}
    Being the universal enveloping algebra of a Lie algebra, $\EHc$ has a natural Hopf algebra structure.  For $\Z$-linear maps $\lambda_1, \lambda_2 \colon \bZ \to \Z$, the coproduct on $\EHc$ induces an algebra homomorphism $E_{\lambda_1 + \lambda_2} \to E_{\lambda_1} \otimes E_{\lambda_2}$.  In particular, if $M$ is an $\EH_k$-module and $N$ is an $\EH_l$-module, then $M \otimes N$ is naturally an $\EH_{k+l}$-module.
\end{rem}

Let $\EH^\pm$ be the subalgebra of $\EHc$ generated by $w_\bx$, $\bx \in \bZ^\pm$.  Note that, for any $k \in \Z$, $\EH^\pm$ is also isomorphic to the subalgebra of $\EH_k$ generated by the $w_\bx$, $\bx \in \bZ^\pm$.  It follows from the PBW theorem that multiplication induces a linear isomorphism
\begin{equation} \label{lantern}
    \EH^+ \otimes \EH^- \xrightarrow{\cong} \EH_k.
\end{equation}
We have automorphisms
\begin{align} \label{psi}
    \psi &\colon \EHc \xrightarrow{\cong} \EHc,&
    w_\bx &\mapsto w_{-\bx},&
    \bx &\mapsto -\bx,
    \\ \label{omega}
    \omega &\colon \EHc \xrightarrow{\cong} \EHc,&
    w_{r,n} &\mapsto (-1)^{n+1} w_{r,-n},&
    (r,n) &\mapsto (r,-n),
\end{align}
and $\psi(\EH^\pm) = \EH^\mp$.  (Note that $\psi$ is the automorphism \cref{rotate} for $\gamma$ equal to negative the identity matrix.)  For $k \in \Z$, $\psi$ and $\omega$ induce algebra isomorphisms
\begin{align} \label{psik}
    \psi_k &\colon \EH_k \xrightarrow{\cong} \EH_{-k},&
    w_\bx &\mapsto w_{-\bx},
    \\ \label{omegak}
    \omega_k &\colon \EH_k \xrightarrow{\cong} \EH_{-k},&
    w_{r,n} &\mapsto (-1)^{n+1} w_{r,-n}.
\end{align}

%-----------------------------------
\subsection{Biangular presentations}
%-----------------------------------

\begin{lem} \label{bulls}
    The subalgebra $\EH^\pm$ is generated by the elements $w_{r, \pm 1}$, $r \in \Z$, and $w_{\pm r,0}$, $r \ge 1$.
\end{lem}

\begin{proof}
    It suffices to consider $\EH^+$, since the result for $\EH^-$ then follows by applying the involution $\psi$ from \cref{psi}.  Let $\EH'$ be the subalgebra of $\EHc$ generated by the elements $w_{r,1}$, $r \in \Z$, and $w_{r,0}$, $r \ge 1$.  We show by induction on $n \ge 1$ that $w_{r,n} \in \EH'$ for all $r \in \Z$, from which the lemma follows.  The base case $n=1$ holds by definition.  Let $n \ge 1$, and assume that $w_{r,n} \in \EH'$ for all $r \in \Z$.  If $r \ne 0$, then $w_{r,n+1} = \{r\}^{-1} [w_{r,n}, w_{0,1}] \in \EH'$.  Otherwise, if $r=0$, we have $w_{0,n+1} = \{n+1\}^{-1}[w_{1,n},w_{-1,1}] \in \EH'$.
\end{proof}

\begin{lem} \label{twizzler}
    Suppose $\fg$ is a Lie algebra with $\kk$-module decomposition $\fg = \fg_1 \oplus \fg_2$, where $\fg_1,\fg_2$ are Lie subalgebras of $\fg$.  Furthermore, suppose that $S_i$ is a set of generators of $\fg_i$ for $i=1,2$, and that $[x,y] \in S_1 \cup S_2$ for all $x \in S_1$, $y \in S_2$.  Then
    \[
        U(\fg) \cong \big( U(\fg_1) \star U(\fg_2) \big)/ \langle xy-yx-[x,y] : x \in S_1,\ y \in S_2 \rangle,
    \]
    where $\star$ denotes the free product of associative algebras.
\end{lem}

\begin{proof}
    Let $I$ be the ideal $\langle xy-yx-[x,y] : x \in S_1,\ y \in S_2 \rangle$.  Consider the sequence
    \[
        U(\fg_1) \otimes U(\fg_2)
        \xrightarrow{f} \big( U(\fg_1) \star U(\fg_2) \big)/ I
        \xrightarrow{g} U(\fg),
    \]
    where $f$ is the $\kk$-linear map given by multiplication, and $g$ is the algebra homomorphism arising from the universal property of the free product and the fact that the generators of $I$ are zero in $U(\fg)$.  The assumption $\fg = \fg_1 \oplus \fg_2$ implies that $g$ is surjective.  Now, elements of $U(\fg_1) \star U(\fg_2)$ can be written as linear combinations of words in $S_1 \cup S_2$.  In the quotient $\big( U(\fg_1) \star U(\fg_2) \big)/ I$, one can use the generators of $I$ to move elements of $U(\fg_1)$ to the left of elements of $U(\fg_2)$ modulo shorter words.  This implies that the map $f$ is surjective.  Finally, the composition $gf$ is a linear isomorphism by the PBW theorem.  It follows that $g$ is surjective, and hence an isomorphism.
\end{proof}

\begin{prop}
    The algebra $\EHc$ is isomorphic to the free product of the algebras $\EH^+ \otimes U(\bZ_\kk)$ and $\EH^-$ modulo the following relations:
    \begin{align}
        [w_{s,-1}, w_{1,1}] &= \{s+1\} w_{s+1,0} - \delta_{s,-1} (1,1),& s \in \Z, \label{cross1} \\
        [w_{s, \pm 1}, w_{\mp r,0}] &= \{r\} w_{s \mp r, \pm 1},& r \ge 1,\ s \in \Z, \label{cross2} \\
        [w_{r,0}, w_{-s,0}] &= \delta_{r,s} (r,0),& r,s \ge 1, \label{cross3} \\
        [\bx, w_{r,-1}] = [\bx,w_{-s,0}] &= 0,& r \in \Z,\ s \ge 1,\ \bx \in \bZ. \label{cross4}
    \end{align}
\end{prop}

\begin{proof}
    Let $\fg = \fEH$, $\fg_1 = \Span_\kk \{w_\bx, \by : \bx \in \bZ^+,\ \by \in \bZ\}$, $\fg_2 = \Span_\kk \{w_\bx : \bx \in \bZ^-\}$, $S_1 = \{\bx, w_{r,1}, w_{s,0} : \bx \in \bZ,\ r \in \Z,\ s \ge 1\}$, $S_2 = \{w_{r,-1}, w_{-s,0} : r \in \Z,\ s \ge 1\}$.  Then it follows from \cref{bulls,twizzler} that $\EHc$ is isomorphic to the free product of the algebras $\EH^+ \otimes U(\bZ_\kk)$ and $\EH^-$ modulo the relations \cref{cross2,cross3,cross4} and
    \begin{equation} \label{cross0}
        [w_{s,-1}, w_{r,1}] = \{r+s\} w_{r+s,0} + \delta_{r,-s} (s,-1),\quad r,s \in \Z,
    \end{equation}
    which specializes to \cref{cross1} when $r=1$.

    It remains to show that the relations \cref{cross1,cross2,cross3,cross4} imply \cref{cross0}.  We first prove that they imply the $r \ge 1$ cases of \cref{cross0} by induction on $r$.  Fix $r \ge 1$ and suppose that \cref{cross0} holds for all $s \in \Z$.  Then, for $s \in \Z$, we have
    \begin{align*}
        \{1\} [w_{s,-1}, w_{r+1,1}]
        \ \ &\overset{\mathclap{\cref{cross2}}}{=}\ \ [w_{s,-1}, [ w_{1,0}, w_{r,1}]] \\
        &= [[w_{s,-1}, w_{1,0}],w_{r,1}] + [w_{1,0}, [w_{s,-1}, w_{r,1}]] \\
        &\overset{\mathclap{\cref{cross2}}}{=}\ \ \{1\} [w_{s+1,-1}, w_{r,1}] + \{r+s\}[w_{1,0}, w_{r+s,0}] +\delta_{r,-s} [w_{1,0}, (s,-1)] \\
        &\overset{\mathclap{\cref{cross3}}}{\underset{\mathclap{\cref{cross4}}}{=}}\ \ \{1\} \{r+s+1\} w_{r+s+1,0} + \{1\} \delta_{r+1,-s} (s+1,-1) - \{1\} \delta_{r+1,-s}(1,0) \\
        &= \{1\} \{r+s+1\} w_{r+s+1,0} + \{1\} \delta_{r+1,-s} (s,-1),
    \end{align*}
    where we used the Jacobi identity in the second equality and the induction hypothesis in the third and fourth equalities.  Dividing both sides by $\{1\}$, this completes the proof of the induction step.

    Finally, we prove that \cref{cross1,cross2,cross3,cross4} imply \cref{cross0} for $r \le 1$ by induction on $r$.  Fix $r \le 1$ and suppose that \cref{cross0} holds for all $s \in \Z$.  Then, for $s \in \Z$, we have
    \begin{align*}
        \{1\} [w_{s,-1}, w_{r-1,1}]
        \ \ &\overset{\mathclap{\cref{cross2}}}{=}\ \ [w_{s,-1}, [w_{r,1}, w_{-1,0}]] \\
        &= [[w_{s,-1}, w_{r,1}], w_{-1,0}] + [w_{r,1}, [w_{s,-1}, w_{-1,0}]] \\
        &\overset{\mathclap{\cref{cross2}}}{=}\ \ \{r+s\} [w_{r+s,0}, w_{-1,0}] + \delta_{r,-s}[(s,-1),w_{-1,0}] - \{1\} [w_{r,1}, w_{s-1,-1}] \\
        &\overset{\mathclap{\cref{cross3}}}{\underset{\mathclap{\cref{cross4}}}{=}}\ \ \{1\} \delta_{r-1,-s} (1,0) + \{1\}\{r+s-1\} w_{r+s-1,0} - \{1\} \delta_{r-1,-s} (r,1)\\
        &= \{1\} \{r+s-1\} w_{r+s-1,0} + \{1\} \delta_{r-1,-s} (s,-1),
    \end{align*}
    where we used the Jacobi identity in the second equality, the induction hypothesis in the third and fourth equalities, and the relation $[w_{s,-1},w_{-1,0}] = -\{1\} w_{s-1,-1}$ in $\EH^-$ in the third equality.
\end{proof}

\begin{cor} \label{crossbow}
    For $k \in \Z$, the algebra $\EH_k$ is isomorphic to the free product of the algebras $\EH^+$ and $\EH^-$ modulo the relations
    \begin{align} \label{bolt1}
        [w_{s,-1},w_{1,1}] &= \{s+1\} w_{s+1,0} - \delta_{s,-1} k,& s \in \Z,
        \\ \label{bolt2}
        [w_{s,\pm 1}, w_{\mp r,0}] &= \{r\} w_{s \mp r, \pm 1},& r \ge 1,\ s \in \Z,
        \\ \label{bolt3}
        [w_{r,0}, w_{-s,0}] &= 0,& r,s \ge 1.
    \end{align}
\end{cor}

\begin{lem} \label{slime}
    For $k \in \Z$, $\EH_k$ is generated, as an algebra, by $w_{r,\pm 1}$, $r \in \Z$.
\end{lem}

\begin{proof}
    Fix $k \in \Z$.  Let $A$ denote the subalgebra of $\EH_k$ generated by $w_{r, \pm 1}$, $r \in \Z$.  By \cref{bulls}, it suffices to show that $w_{r,0} \in A$ for all nonzero $r \in \Z$.  But this follows easily from the fact that, for $r \in \Z$, $r \ne 0$, we have $[w_{0,-1},w_{r,1}] = \{r\} w_{r,0}$.
\end{proof}

\begin{cor}
    The isomorphism $\omega_k$ from \cref{omegak} is the unique isomorphism $\EH_k \xrightarrow{\cong} \EH_{-k}$ such that $w_{r,\pm 1} \mapsto w_{r,\mp 1}$ for $r \in \Z$.
\end{cor}

%=============================================
\section{Trace of a category\label{sec:trace}}
%=============================================

In this section we collect some important facts about traces of categories.  We refer the reader to \cite{BGHL14} for a more thorough treatment.  Throughout this section $\kk$ denotes an arbitrary commutative ring.

Recall that the \emph{trace} or \emph{zeroth Hochschild homology} of a small $\kk$-linear category $\cC$ is the $\kk$-module
\begin{equation}
    \Tr(\cC) := \left( \bigoplus_{X \in \cC} \End_\cC(X) \right)/\Span_\kk \{f \circ g - g \circ f\},
\end{equation}
where $f$ and $g$ run through all pairs of morphisms $f \colon X \to Y$ and $g \colon Y \to X$ in $\cC$.  We let $[f] \in \Tr(\cC)$ denote the class of an endomorphism $f \in \End_\cC(X)$.  For $f,g \in \bigoplus_{X \in \cC} \End_\cC(X)$, we define
\begin{equation} \label{equivdef}
    f \equiv g \iff [f] = [g].
\end{equation}
Thus, for example, we have
\begin{equation} \label{dizzy}
    fg \equiv gf \quad \text{for all } f \colon X \to Y,\ g \colon Y \to X.
\end{equation}

If $\cC$ is a $\kk$-linear \emph{monoidal} category, then $\Tr(\cC)$ is an associative $\kk$-algebra with multiplication given by
\begin{equation}
    [f] [g] := [f \otimes g].
\end{equation}
A $\kk$-linear functor $F \colon \cC \to \cD$ induces a linear map on traces:
\begin{equation} \label{crayon}
    \Tr(F) \colon \Tr(\cC) \to \Tr(\cD),\quad [f] \mapsto [F(f)],\quad f \text{ an endomorphism in } \cC.
\end{equation}
If $\cC,\cD$ are monoidal categories and $F$ is a monoidal functor, then \cref{crayon} is a homomorphism of associative $\kk$-algebras.

If $\cC_i$, $i \in I$, are $\kk$-linear categories, then we have a canonical isomorphism
\[
    \Tr \left( \bigsqcup_{i \in I} \cC_i \right) \cong \bigoplus_{i \in I} \Tr(\cC_i).
\]
If $\cC_1, \cC_2$ are $\kk$-linear subcategories of a $\kk$-linear category $\cC$, then the tensor product $\otimes \colon \cC_1 \times \cC_2 \to \cC$ induces a linear map
\[
    \Tr(\cC_1) \otimes \Tr(\cC_2) \to \Tr(\cC).
\]

For a $\kk$-linear category $\cC$, let $\Add(\cC)$ denote its additive envelope.  If $\cC$ is monoidal, then $\Add(\cC)$ inherits a natural monoidal structure.

\begin{lem}[{\cite[Exercise~9]{BGHL14}}] \label{GME}
    If $\cC$ is a $\kk$-linear category, then the inclusion functor $\cC \to \Add(\cC)$ induces a linear isomorphism $\Tr(\cC) \cong \Tr(\Add(\cC))$.  If $\cC$ is monoidal, then this is an isomorphism of associative $\kk$-algebras.
\end{lem}

\begin{proof}
    Objects of $\Add(\cC)$ are formal direct sums $\bigoplus_{i=1}^n X_i$, $X_i \in \cC$.  An endomorphism of such an object is a matrix $(f_{ij})_{i,j=1}^n$ with $f_{ij} \colon X_i \to X_j$.  For $j=1,\dotsc,n$, we have canonical inclusion and projection maps
    \[
        X_j \xhookrightarrow{\imath_j} \bigoplus_{i=1}^n X_i \xrightarrowdbl{\pi_j} X_j.
    \]
    Then we have
    \[
        (f_{ij})_{i,j=1}^n
        = \sum_{i,j=1}^n \imath_j f_{ij} \pi_i
        \equiv \sum_{i,j=1}^n f_{ij} \pi_i \imath_j
        \equiv \sum_{i=1}^n f_{ii}.
    \]
    Thus the map $\Tr(\cC) \to \Tr(\Add(\cC))$ induced by the inclusion functor is surjective.

    Similarly, for morphisms
    \[
        \bigoplus_{i=1}^n X_i \xrightarrow{f=(f_{ij})} \bigoplus_{j=1}^m Y_j \xrightarrow{g=(g_{ji})} \bigoplus_{i=1}^n X_i
    \]
    in $\Add(\cC)$, we have
    \[
        fg - gf
        = \left( \sum_{i=1}^n f_{il} g_{ji} \right)_{j,l=1}^m - \left( \sum_{j=1}^m g_{jl} f_{ij} \right)_{i,l=1}^n
        \equiv \sum_{i=1}^n \sum_{j=1}^m (f_{ij} g_{ji} - g_{ji} f_{ij}).
    \]
    Hence the map $\Tr(\cC) \to \Tr(\Add(\cC))$ induced by the inclusion functor is also injective.  As noted above, this map is a homomorphism of associative $\kk$-algebras when $\cC$ is monoidal.
\end{proof}

If $S \subseteq \Ob(\cC)$ is a subset of the set of objects of a small category $\cC$, let $\cC|_S$ denote the full subcategory of $\cC$ with $\Ob(\cC|_S) = S$.

\begin{lem}[{\cite[Lem.~2.1]{BHLW17}}] \label{AMC}
    Suppose $\cC$ is a $\kk$-linear additive category.  Let $S \subseteq \Ob(\cC)$ be a subset such that every object of $\cC$ is isomorphic to a direct sum of finitely many copies of objects in $S$.  Then the inclusion functor $\cC|_S \to \cC$ induces a linear isomorphism $\Tr(\cC|_S) \cong \Tr(\cC)$.
\end{lem}

\begin{cor} \label{WSB}
    Suppose $\cC$ is a $\kk$-linear (not necessarily additive) category.  Let $S \subseteq \Ob(\cC)$ be a subset such that every object of $\cC$ is isomorphic in $\Add(\cC)$ to a direct sum of finitely many copies of objects in $S$.  Then the inclusion functor $\cC|_S \to \cC$ induces a linear isomorphism $\Tr(\cC|_S) \cong \Tr(\cC)$.
\end{cor}

\begin{proof}
    The inclusion functors $\cC|_S \xrightarrow{F} \cC \xrightarrow{G} \Add(\cC)$ induce linear maps
    \[
        \Tr(\cC|_S) \xrightarrow{\Tr(F)} \Tr(\cC) \xrightarrow{\Tr(G)} \Tr(\Add(\cC)).
    \]
    The maps $\Tr(G) \circ \Tr(F) = \Tr(G \circ F)$ and $\Tr(G)$ are linear isomorphisms by \cref{AMC} and \cref{GME}, respectively.  It follows that $\Tr(F)$ is also a linear isomorphism.
\end{proof}

For the remainder of this section, we assume that $\cC$ is a small $\kk$-linear additive category with
\[
    \Ob(\cC) = \{X_n : n \in \N\},
\]
where $\N$ denotes the set of nonnegative integers, and we assume that $X_n \ne X_m$ for $n \ne m$.  Furthermore, suppose that we have subsets $D_{m,n} \subseteq \Hom_\cC(X_n,X_m)$, $m,n \in \N$, and $D_n \subseteq \Hom_\cC(X_n,X_n)$, $n \in \N$, with the following properties:
\begin{description}
    \item[B1\label{B1}] We have $D_{n,n} = \{1_{X_n}\}$ for all $n \in \N$, and $\bB_{m,n} := \bigsqcup_{l=0}^{\min(m,n)} D_{m,l} D_l D_{l,n}$ is a basis of $\Hom_\cC(X_n,X_m)$ for each $m,n \in \N$.  (Part of our assumption here is that the sets $D_{m,l} D_l D_{l,n}$ are disjoint.)
    \item[B2\label{B2}] For all $n \in \N$, $R_n := \Span_\kk D_n$ is a subalgebra of $\End_\cC(X_n)$.
\end{description}
For $n \in \N$, let $\cC_n$ denote $\kk$-linear subcategory of $\cC$ with one object $X_n$ and $\End_{\cC_n}(X_n) = R_n$.

\begin{prop} \label{saber}
    Under the above assumptions on $\cC$, the inclusion $\bigoplus_{n \in \N} \cC_n \to \cC$ induces a linear isomorphism
    \[
        \bigoplus_{n \in \N} \Tr(\cC_n) \xrightarrow{\cong} \Tr(\cC).
    \]
\end{prop}

\begin{proof}
    This is proved in \cite[Prop.~2.11]{RS20} using the equivalent language of locally unital algebras.
\end{proof}

%=====================================================
\section{Quantum Heisenberg category\label{sec:qheis}}
%=====================================================

In this section, we recall the definition of the quantum Heisenberg category introduced in \cite{BSW-qheis} and state some important relations that will be used in our computations to follow.  Throughout this section we fix an integral domain $\kk$ containing $\Q$, and $q \in \kk^\times$ satisfying \cref{trevor}.  Let $z=q-q^{-1}=\{1\}$ and choose $t \in \kk^\times$.  The most generic choice of ground ring is thus $\kk = \Q[q^{\pm 1}, t^{\pm 1}, \{d\}^{-1} : d \ge 1]$.  Another valid choice is a field $\kk$ of characteristic zero, with $q,t \in \kk^\times$ such that $q$ is not a root of unity.  We also fix a \emph{central charge} $k \in \Z$.

%----------------------
\subsection{Definition}
%----------------------

\begin{defin}[{\cite[Def.~4.1]{BSW-qheis}}]
    The \emph{quantum Heisenberg category} $\Heis_k$ is the strict $\kk$-linear monoidal category generated by objects $\uparrow$, $\downarrow$ and morphisms
    \begin{gather*}
        \posupcross,\ \negupcross \colon \uparrow \otimes \uparrow\ \to\ \uparrow \otimes \uparrow,\qquad
        \updot \colon \uparrow\ \to\ \uparrow,
        \\
        \rightcup \colon \one \to\ \downarrow \otimes \uparrow,\qquad
        \rightcap \colon \uparrow \otimes \downarrow\ \to \one,\qquad
        \leftcup \colon \one \to\ \uparrow \otimes \downarrow,\qquad
        \leftcap \colon \downarrow \otimes \uparrow\ \to \one,
    \end{gather*}
    subject to relations that we now describe.  First, we require $\updot$, which we call a \emph{dot}, to be invertible.  For $r \in \Z$, we let $\multupdot{r}$ denote the composition of $r$ dots if $r \ge 0$ and the composition of $|r|$ inverse dots if $r < 0$.  We then impose the following additional relations:
    \begin{gather} \label{braid}
        \begin{tikzpicture}[centerzero]
            \draw[->] (-0.2,-0.5) \braidup (0.2,0) \braidup (-0.2,0.5);
            \draw[wipe] (0.2,-0.5) \braidup (-0.2,0) \braidup (0.2,0.5);
            \draw[->] (0.2,-0.5) \braidup (-0.2,0) \braidup (0.2,0.5);
        \end{tikzpicture}
        =
        \begin{tikzpicture}[centerzero]
            \draw[->] (-0.2,-0.5) -- (-0.2,0.5);
            \draw[->] (0.2,-0.5) -- (0.2,0.5);
        \end{tikzpicture}
        =
        \begin{tikzpicture}[centerzero]
            \draw[->] (0.2,-0.5) \braidup (-0.2,0) \braidup (0.2,0.5);
            \draw[wipe] (-0.2,-0.5) \braidup (0.2,0) \braidup (-0.2,0.5);
            \draw[->] (-0.2,-0.5) \braidup (0.2,0) \braidup (-0.2,0.5);
        \end{tikzpicture}
        \ ,\qquad
        \begin{tikzpicture}[centerzero]
            \draw[->] (0.3,-0.5) -- (-0.3,0.5);
            \draw[wipe] (0,-0.5) \braidup (-0.3,0) \braidup (0,0.5);
            \draw[->] (0,-0.5) \braidup (-0.3,0) \braidup (0,0.5);
            \draw[wipe] (-0.3,-0.5) -- (0.3,0.5);
            \draw[->] (-0.3,-0.5) -- (0.3,0.5);
        \end{tikzpicture}
        =
        \begin{tikzpicture}[centerzero]
            \draw[->] (0.3,-0.5) -- (-0.3,0.5);
            \draw[wipe] (0,-0.5) \braidup (0.3,0)\braidup (0,0.5);
            \draw[->] (0,-0.5) \braidup (0.3,0)\braidup (0,0.5);
            \draw[wipe] (-0.3,-0.5) -- (0.3,0.5);
            \draw[->] (-0.3,-0.5) -- (0.3,0.5);
        \end{tikzpicture}
        \ ,
        \\ \label{skein}
        \posupcross - \negupcross = z\
        \begin{tikzpicture}[centerzero]
            \draw[->] (-0.15,-0.2) -- (-0.15,0.2);
            \draw[->] (0.15,-0.2) -- (0.15,0.2);
        \end{tikzpicture}
        \ ,
        \\ \label{dotslide}
        \begin{tikzpicture}[centerzero]
            \draw[->] (0.3,-0.3) -- (-0.3,0.3);
            \draw[wipe] (-0.3,-0.3) -- (0.3,0.3);
            \draw[->] (-0.3,-0.3) -- (0.3,0.3);
            \singdot{0.17,-0.17};
        \end{tikzpicture}
        =
        \begin{tikzpicture}[centerzero]
            \draw[->] (-0.3,-0.3) -- (0.3,0.3);
            \draw[wipe] (0.3,-0.3) -- (-0.3,0.3);
            \draw[->] (0.3,-0.3) -- (-0.3,0.3);
            \singdot{-0.15,0.15};
        \end{tikzpicture}
        \ ,\qquad
        \begin{tikzpicture}[centerzero]
            \draw[->] (-0.3,-0.3) -- (0.3,0.3);
            \draw[wipe] (0.3,-0.3) -- (-0.3,0.3);
            \draw[->] (0.3,-0.3) -- (-0.3,0.3);
            \singdot{-0.17,-0.17};
        \end{tikzpicture}
        =
        \begin{tikzpicture}[centerzero]
            \draw[->] (0.3,-0.3) -- (-0.3,0.3);
            \draw[wipe] (-0.3,-0.3) -- (0.3,0.3);
            \draw[->] (-0.3,-0.3) -- (0.3,0.3);
            \singdot{0.15,0.15};
        \end{tikzpicture}
        \ ,
        \\ \label{rightadj}
        \begin{tikzpicture}[centerzero]
            \draw[->] (-0.3,-0.4) -- (-0.3,0) arc(180:0:0.15) arc(180:360:0.15) -- (0.3,0.4);
        \end{tikzpicture}
        =
        \begin{tikzpicture}[centerzero]
            \draw[->] (0,-0.4) -- (0,0.4);
        \end{tikzpicture}
        \ ,\qquad
        \begin{tikzpicture}[centerzero]
            \draw[->] (-0.3,0.4) -- (-0.3,0) arc(180:360:0.15) arc(180:0:0.15) -- (0.3,-0.4);
        \end{tikzpicture}
        =
        \begin{tikzpicture}[centerzero]
            \draw[<-] (0,-0.4) -- (0,0.4);
        \end{tikzpicture}
        \ ,
        \\ \label{pos}
        \begin{tikzpicture}[centerzero]
            \draw[->] (0.2,0) \braidup (-0.2,0.5);
            \draw[<-] (0.2,-0.5) \braidup (-0.2,0);
            \draw[wipe] (-0.2,-0.5) \braidup (0.2,0);
            \draw (-0.2,-0.5) \braidup (0.2,0);
            \draw[wipe] (-0.2,0) \braidup (0.2,0.5);
            \draw (-0.2,0) \braidup (0.2,0.5);
        \end{tikzpicture}
        \ =\
        \begin{tikzpicture}[centerzero]
            \draw[->] (-0.2,-0.5) to (-0.2,0.5);
            \draw[<-] (0.2,-0.5) to (0.2,0.5);
        \end{tikzpicture}
        \ - t^{-1}z\
        \begin{tikzpicture}[centerzero]
            \draw[<-] (-0.2,0.5) -- (-0.2,0.3) arc(180:360:0.2) -- (0.2,0.5);
            \draw[->] (-0.2,-0.5) -- (-0.2,-0.3) arc(180:0:0.2) -- (0.2,-0.5);
        \end{tikzpicture}
        \ + z^2 \sum_{r,s > 0}\
        \begin{tikzpicture}[anchorbase]
            \draw[<-] (-0.2,0.6) to (-0.2,0.35) arc(180:360:0.2) to (0.2,0.6);
            \draw[->] (-0.2,-0.6) to (-0.2,-0.35) arc(180:0:0.2) to (0.2,-0.6);
            \plusleftside{-0.6,0}{-r-s};
            \multdot{-0.2,0.42}{east}{r};
            \multdot{-0.2,-0.42}{east}{s};
        \end{tikzpicture}
        \ ,
        \\ \label{neg}
        \begin{tikzpicture}[centerzero]
            \draw[->] (-0.2,0) \braidup (0.2,0.5);
            \draw[wipe] (-0.2,-0.5) \braidup (0.2,0) \braidup (-0.2,0.5);
            \draw[<-] (-0.2,-0.5) \braidup (0.2,0) \braidup (-0.2,0.5);
            \draw[wipe] (0.2,-0.5) \braidup (-0.2,0);
            \draw (0.2,-0.5) \braidup (-0.2,0);
        \end{tikzpicture}
        \ =\
        \begin{tikzpicture}[centerzero]
            \draw[<-] (-0.2,-0.5) to (-0.2,0.5);
            \draw[->] (0.2,-0.5) to (0.2,0.5);
        \end{tikzpicture}
        \ + tz\
        \begin{tikzpicture}[centerzero]
            \draw[->] (-0.2,0.5) -- (-0.2,0.3) arc(180:360:0.2) -- (0.2,0.5);
            \draw[<-] (-0.2,-0.5) -- (-0.2,-0.3) arc(180:0:0.2) -- (0.2,-0.5);
        \end{tikzpicture}
        \ + z^2 \sum_{r,s > 0}\
        \begin{tikzpicture}[anchorbase]
            \draw[->] (-0.2,0.6) to (-0.2,0.35) arc(180:360:0.2) to (0.2,0.6);
            \draw[<-] (-0.2,-0.6) to (-0.2,-0.35) arc(180:0:0.2) to (0.2,-0.6);
            \plusrightside{0.6,0}{-r-s};
            \multdot{0.2,0.42}{west}{r};
            \multdot{0.2,-0.42}{west}{s};
        \end{tikzpicture}
        \ ,
        \\ \label{curls1}
        \begin{tikzpicture}[centerzero]
            \draw (-0.2,-0.5) -- (-0.2,-0.35) to[out=up,in=west] (0.05,0.2) to[out=right,in=up] (0.2,0);
            \draw[wipe] (0.2,0) to[out=down,in=east] (0.05,-0.2) to[out=left,in=down] (-0.2,0.35) -- (-0.2,0.5);
            \draw[->] (0.2,0) to[out=down,in=east] (0.05,-0.2) to[out=left,in=down] (-0.2,0.35) -- (-0.2,0.5);
        \end{tikzpicture}
        \ = \delta_{k,0} t^{-1}\
        \begin{tikzpicture}[centerzero]
            \draw[->] (0,-0.5) to (0,0.5);
        \end{tikzpicture}
        \quad \text{if } k \ge 0,
        \qquad \qquad
        \rightbubside{r}
        = \frac{\delta_{r,0} t - \delta_{r,k} t^{-1}}{z} 1_\one \quad \text{if } 0 \le r \le k,
        \\ \label{curls2}
        \begin{tikzpicture}[centerzero]
            \draw (0.2,-0.5) -- (0.2,-0.35) to[out=up,in=east] (-0.05,0.2) to[out=left,in=up] (-0.2,0);
            \draw[wipe] (-0.2,0) to[out=down,in=west] (-0.05,-0.2) to[out=right,in=down] (0.2,0.35) -- (0.2,0.5);
            \draw[->] (-0.2,0) to[out=down,in=west] (-0.05,-0.2) to[out=right,in=down] (0.2,0.35) -- (0.2,0.5);
        \end{tikzpicture}
        \ = \delta_{k,0} t\
        \begin{tikzpicture}[centerzero]
            \draw[->] (0,-0.5) to (0,0.5);
        \end{tikzpicture}
        \quad \text{if } k \le 0,
        \qquad \qquad
        \leftbubside{r}
        = \frac{\delta_{r,-k} t - \delta_{r,0} t^{-1}}{z} 1_\one \quad \text{if } 0 \le r \le -k.
    \end{gather}
    (In fact, the second relation in \cref{dotslide} is redundant, since it follows from the first relation in \cref{dotslide} and the first two equalities in \cref{braid}.)  In the above relations we have used right and left crossings defined by
    \begin{equation} \label{windmill}
        \posrightcross
        :=
        \begin{tikzpicture}[centerzero]
            \draw[->] (0.2,-0.3) \braidup (-0.2,0.3);
            \draw[wipe] (-0.4,0.3) -- (-0.4,0.1) to[out=down,in=left] (-0.2,-0.2) to[out=right,in=left] (0.2,0.2) to[out=right,in=up] (0.4,-0.1) -- (0.4,-0.3);
            \draw[->] (-0.4,0.3) -- (-0.4,0.1) to[out=down,in=left] (-0.2,-0.2) to[out=right,in=left] (0.2,0.2) to[out=right,in=up] (0.4,-0.1) -- (0.4,-0.3);
        \end{tikzpicture}
        \ ,\qquad
        \negrightcross
        :=
        \begin{tikzpicture}[centerzero]
            \draw[->] (-0.4,0.3) -- (-0.4,0.1) to[out=down,in=left] (-0.2,-0.2) to[out=right,in=left] (0.2,0.2) to[out=right,in=up] (0.4,-0.1) -- (0.4,-0.3);
            \draw[wipe] (0.2,-0.3) \braidup (-0.2,0.3);
            \draw[->] (0.2,-0.3) \braidup (-0.2,0.3);
        \end{tikzpicture}
        \ ,\qquad
        \posleftcross
        \ :=\
        \begin{tikzpicture}[centerzero]
            \draw[->] (0.4,0.3) -- (0.4,0.1) to[out=down,in=right] (0.2,-0.2) to[out=left,in=right] (-0.2,0.2) to[out=left,in=up] (-0.4,-0.1) -- (-0.4,-0.3);
            \draw[wipe] (-0.2,-0.3) \braidup (0.2,0.3);
            \draw[->] (-0.2,-0.3) \braidup (0.2,0.3);
        \end{tikzpicture}
        \ ,\qquad
        \negleftcross
        \ :=\
        \begin{tikzpicture}[centerzero]
            \draw[->] (-0.2,-0.3) \braidup (0.2,0.3);
            \draw[wipe] (0.4,0.3) -- (0.4,0.1) to[out=down,in=right] (0.2,-0.2) to[out=left,in=right] (-0.2,0.2) to[out=left,in=up] (-0.4,-0.1) -- (-0.4,-0.3);
            \draw[->] (0.4,0.3) -- (0.4,0.1) to[out=down,in=right] (0.2,-0.2) to[out=left,in=right] (-0.2,0.2) to[out=left,in=up] (-0.4,-0.1) -- (-0.4,-0.3);
        \end{tikzpicture}
        \ ,
    \end{equation}
    and $(+)$-bubbles defined by
    \begin{equation} \label{clear}
        \leftplusside{r} := \leftbubside{r},\qquad
        \rightplusside{r} := \rightbubside{r},\qquad
        r > 0,
    \end{equation}
    and
    \begin{align}
        \leftplusside{r-k} &:=
        t^{r+1} z^{r-1} \det \left( \rightbubside{k+i-j+1} \right)_{i,j=1,\dotsc,r},& r \le k,
        \\
        \rightplusside{r+k} &:=
        -t^{-r-1} z^{r-1} \det \left( - \leftbubside{-k+i-j+1} \right)_{i,j=1,\dotsc,r},& r \le -k,
    \end{align}
    where we interpret the determinants as $\delta_{r,0}$ when $r \le 0$.  In particular, note that the sums appearing in \cref{pos,neg} are finite.  When we wish to make the parameters $z$ and $t$ explicit, we will write $\Heis_k(z,t)$ for $\Heis_k$.  This completes the definition of $\Heis_k$.
\end{defin}

As explained in the proof of \cite[Th.~4.2]{BSW-qheis}, the defining relations of $\Heis_k$ imply that we have the isomorphisms
\begin{equation} \label{invrel}
    \begin{aligned}
        \begin{pmatrix}
            \posrightcross
            &
            \begin{tikzpicture}[centerzero]
                \draw[->] (-0.15,-0.2) -- (-0.15,0.05) arc(180:0:0.15) -- (0.15,-0.2);
            \end{tikzpicture}
            &
            \begin{tikzpicture}[centerzero]
                \draw[->] (-0.15,-0.2) -- (-0.15,0.05) arc(180:0:0.15) -- (0.15,-0.2);
                \singdot{-0.15,0};
            \end{tikzpicture}
            &
            \cdots
            &
            \begin{tikzpicture}[centerzero]
                \draw[->] (-0.15,-0.2) -- (-0.15,0.05) arc(180:0:0.15) -- (0.15,-0.2);
                \multdot{-0.15,0}{east}{k-1};
            \end{tikzpicture}
        \end{pmatrix}^T
        &\colon \uparrow \otimes \downarrow\ \to\ \downarrow \otimes \uparrow \otimes \one^{\oplus k} & \text{if } k \ge 0,
        \\
        \begin{pmatrix}
            \posrightcross
            &
            \begin{tikzpicture}[centerzero]
                \draw[->] (-0.15,0.2) -- (-0.15,-0.05) arc(180:360:0.15) -- (0.15,0.2);
            \end{tikzpicture}
            &
            \begin{tikzpicture}[centerzero]
                \draw[->] (-0.15,0.2) -- (-0.15,-0.05) arc(180:360:0.15) -- (0.15,0.2);
                \singdot{0.15,0};
            \end{tikzpicture}
            &
            \cdots
            &
            \begin{tikzpicture}[centerzero]
                \draw[->] (-0.15,0.2) -- (-0.15,-0.05) arc(180:360:0.15) -- (0.15,0.2);
                \multdot{0.15,0}{west}{-k-1};
            \end{tikzpicture}
        \end{pmatrix}
        &\colon \uparrow \otimes \downarrow \oplus \one^{\oplus (-k)} \to\ \downarrow \otimes \uparrow & \text{if } k \le 0,
    \end{aligned}
\end{equation}
in $\Add(\Heis_k)$, where $\Add$ denotes the additive envelope.

%--------------------------------
\subsection{Additional relations}
%--------------------------------

We now recall some additional relations that hold in $\Heis_k$.  It follows from the defining relations that
\begin{equation} \label{silence}
    \begin{tikzpicture}[centerzero]
        \draw (-0.2,-0.5) \braidup (0.2,0);
        \draw (-0.2,0) \braidup (0.2,0.5);
        \draw[wipe] (0.2,0) \braidup (-0.2,0.5);
        \draw[->] (0.2,0) \braidup (-0.2,0.5);
        \draw[wipe] (0.2,-0.5) \braidup (-0.2,0);
        \draw[<-] (0.2,-0.5) \braidup (-0.2,0);
    \end{tikzpicture}
    \ =\
    \begin{tikzpicture}[centerzero]
        \draw[->] (-0.2,-0.5) to (-0.2,0.5);
        \draw[<-] (0.2,-0.5) to (0.2,0.5);
    \end{tikzpicture}
    \text{ if } k < 0,
    \quad
    \begin{tikzpicture}[centerzero]
        \draw (0.2,0) \braidup (-0.2,0.5);
        \draw (0.2,-0.5) \braidup (-0.2,0);
        \draw[wipe] (-0.2,-0.5) \braidup (0.2,0);
        \draw[<-] (-0.2,-0.5) \braidup (0.2,0);
        \draw[wipe] (-0.2,0) \braidup (0.2,0.5);
        \draw[->] (-0.2,0) \braidup (0.2,0.5);
    \end{tikzpicture}
    \ =\
    \begin{tikzpicture}[centerzero]
        \draw[<-] (-0.2,-0.5) to (-0.2,0.5);
        \draw[->] (0.2,-0.5) to (0.2,0.5);
    \end{tikzpicture}
    \text{ if } k > 0,
    \quad
    \begin{tikzpicture}[centerzero]
        \draw[->] (0.2,-0.5) \braidup (-0.2,0) \braidup (0.2,0.5);
        \draw[wipe] (-0.2,-0.5) \braidup (0.2,0) \braidup (-0.2,0.5);
        \draw[<-] (-0.2,-0.5) \braidup (0.2,0) \braidup (-0.2,0.5);
    \end{tikzpicture}
    =
    \begin{tikzpicture}[centerzero]
        \draw[<-] (-0.2,-0.5) -- (-0.2,0.5);
        \draw[->] (0.2,-0.5) -- (0.2,0.5);
    \end{tikzpicture}
    \text{ if } k=0,
    \quad
    \begin{tikzpicture}[centerzero]
        \draw[->] (-0.2,-0.5) \braidup (0.2,0) \braidup (-0.2,0.5);
        \draw[wipe] (0.2,-0.5) \braidup (-0.2,0) \braidup (0.2,0.5);
        \draw[<-] (0.2,-0.5) \braidup (-0.2,0) \braidup (0.2,0.5);
    \end{tikzpicture}
    =
    \begin{tikzpicture}[centerzero]
        \draw[->] (-0.2,-0.5) -- (-0.2,0.5);
        \draw[<-] (0.2,-0.5) -- (0.2,0.5);
    \end{tikzpicture}
    \text{ if } k=0.
\end{equation}
(See \cite[(4.14)--(4.16)]{BSW-qheis}.)

Note that \cref{rightadj} implies that $\uparrow$ is right dual to $\downarrow$.  We also have (see \cite[Lem.~3.7]{BSW-qheis})
\begin{equation} \label{leftadj}
    \begin{tikzpicture}[centerzero]
        \draw[<-] (-0.3,-0.4) -- (-0.3,0) arc(180:0:0.15) arc(180:360:0.15) -- (0.3,0.4);
    \end{tikzpicture}
    =
    \begin{tikzpicture}[centerzero]
        \draw[<-] (0,-0.4) -- (0,0.4);
    \end{tikzpicture}
    \ ,\qquad
    \begin{tikzpicture}[centerzero]
        \draw[<-] (-0.3,0.4) -- (-0.3,0) arc(180:360:0.15) arc(180:0:0.15) -- (0.3,-0.4);
    \end{tikzpicture}
    =
    \begin{tikzpicture}[centerzero]
        \draw[->] (0,-0.4) -- (0,0.4);
    \end{tikzpicture}
    \ ,
\end{equation}
so that $\uparrow$ is also left dual to $\downarrow$.  In fact $\Heis_k$ is strictly pivotal, with duality functor defined on morphisms by rotating diagrams through $180\degree$; see \cite[(3.2.1)]{BSW-qheis}.  Thus, for example, we can define downward crossings and dots by
\[
    \posdowncross
    :=
    \begin{tikzpicture}[centerzero]
        \draw[->] (-0.4,0.3) -- (-0.4,0.1) to[out=down,in=left] (-0.2,-0.2) to[out=right,in=left] (0.2,0.2) to[out=right,in=up] (0.4,-0.1) -- (0.4,-0.3);
        \draw[wipe] (0.2,-0.3) \braidup (-0.2,0.3);
        \draw[<-] (0.2,-0.3) \braidup (-0.2,0.3);
    \end{tikzpicture}
    =
    \begin{tikzpicture}[centerzero]
        \draw[<-] (-0.2,-0.3) \braidup (0.2,0.3);
        \draw[wipe] (0.4,0.3) -- (0.4,0.1) to[out=down,in=right] (0.2,-0.2) to[out=left,in=right] (-0.2,0.2) to[out=left,in=up] (-0.4,-0.1) -- (-0.4,-0.3);
        \draw[->] (0.4,0.3) -- (0.4,0.1) to[out=down,in=right] (0.2,-0.2) to[out=left,in=right] (-0.2,0.2) to[out=left,in=up] (-0.4,-0.1) -- (-0.4,-0.3);
    \end{tikzpicture}
    \, ,\qquad
    \negdowncross
    :=
    \begin{tikzpicture}[centerzero]
        \draw[<-] (0.2,-0.3) \braidup (-0.2,0.3);
        \draw[wipe] (-0.4,0.3) -- (-0.4,0.1) to[out=down,in=left] (-0.2,-0.2) to[out=right,in=left] (0.2,0.2) to[out=right,in=up] (0.4,-0.1) -- (0.4,-0.3);
        \draw[->] (-0.4,0.3) -- (-0.4,0.1) to[out=down,in=left] (-0.2,-0.2) to[out=right,in=left] (0.2,0.2) to[out=right,in=up] (0.4,-0.1) -- (0.4,-0.3);
    \end{tikzpicture}
    =
    \begin{tikzpicture}[centerzero]
        \draw[->] (0.4,0.3) -- (0.4,0.1) to[out=down,in=right] (0.2,-0.2) to[out=left,in=right] (-0.2,0.2) to[out=left,in=up] (-0.4,-0.1) -- (-0.4,-0.3);
        \draw[wipe] (-0.2,-0.3) \braidup (0.2,0.3);
        \draw[<-] (-0.2,-0.3) \braidup (0.2,0.3);
    \end{tikzpicture}
    \, ,\qquad
    \downdot
    :=
    \begin{tikzpicture}[centerzero]
        \draw[<-] (-0.3,-0.4) -- (-0.3,0) arc(180:0:0.15) arc(180:360:0.15) -- (0.3,0.4);
        \singdot{0,0};
    \end{tikzpicture}
    =
    \begin{tikzpicture}[centerzero]
        \draw[->] (-0.3,0.4) -- (-0.3,0) arc(180:360:0.15) arc(180:0:0.15) -- (0.3,-0.4);
        \singdot{0,0};
    \end{tikzpicture}
    \, ,
\]
and we have right, left, and downwards skein relations,
\begin{equation} \label{skeinrot}
    \posdowncross - \negdowncross
    = z\
    \begin{tikzpicture}[centerzero]
        \draw[<-] (-0.15,-0.2) -- (-0.15,0.2);
        \draw[<-] (0.15,-0.2) -- (0.15,0.2);
    \end{tikzpicture}
    \ ,\quad
    \posrightcross - \negrightcross
    = z\
    \begin{tikzpicture}[centerzero]
        \draw[->] (-0.15,0.3) -- (-0.15,0.25) arc(180:360:0.15) -- (0.15,0.3);
        \draw[->] (-0.15,-0.3) -- (-0.15,-0.25) arc(180:0:0.15) -- (0.15,-0.3);
    \end{tikzpicture}
    \ ,\quad
    \posleftcross - \negleftcross
    = z\
    \begin{tikzpicture}[centerzero]
        \draw[<-] (-0.15,0.3) -- (-0.15,0.25) arc(180:360:0.15) -- (0.15,0.3);
        \draw[<-] (-0.15,-0.3) -- (-0.15,-0.25) arc(180:0:0.15) -- (0.15,-0.3);
    \end{tikzpicture}
    \ .
\end{equation}
as well as right, left, and downward versions of \cref{dotslide,dotslide}.  In what follows, we will freely use the pivotal structure, referring to a relation by equation number even when we use a rotated version of it.  In addition, it follows from the pivotal structure on $\Heis_k$ that dots slide over cups and caps.  Therefore, we will sometimes draw dots at the critical points of cups or caps, since this causes no ambiguity.

It follows from repeated use of \cref{dotslide,skein} that the following relations hold for $r \in \Z$:
\begin{align} \label{teapos}
    \begin{tikzpicture}[centerzero]
        \draw[->] (0.3,-0.4) -- (-0.3,0.4);
        \draw[wipe] (-0.3,-0.4) -- (0.3,0.4);
        \draw[->] (-0.3,-0.4) -- (0.3,0.4);
        \multdot{0.15,-0.2}{west}{r};
    \end{tikzpicture}
    &=
    \begin{cases}
        \begin{tikzpicture}[centerzero]
            \draw[->] (-0.3,-0.4) -- (0.3,0.4);
            \draw[wipe] (0.3,-0.4) -- (-0.3,0.4);
            \draw[->] (0.3,-0.4) -- (-0.3,0.4);
            \multdot{-0.15,0.2}{east}{r};
        \end{tikzpicture}
        - z\displaystyle\sum_{\substack{a+b=r \\ a,b > 0}}
        \begin{tikzpicture}[centerzero]
            \draw[->] (-0.2,-0.4) -- (-0.2,0.4);
            \draw[->] (0.2,-0.4) -- (0.2,0.4);
            \multdot{-0.2,0}{east}{a};
            \multdot{0.2,0}{west}{b};
        \end{tikzpicture}
        & \text{if } r > 0,
        \\
        \begin{tikzpicture}[anchorbase]
            \draw[->] (-0.3,-0.4) -- (0.3,0.4);
            \draw[wipe] (0.3,-0.4) -- (-0.3,0.4);
            \draw[->] (0.3,-0.4) -- (-0.3,0.4);
            \multdot{-0.15,0.2}{east}{r};
        \end{tikzpicture}
        + z\displaystyle\sum_{\substack{a+b=r \\ a,b \le 0}}
        \begin{tikzpicture}[centerzero]
            \draw[->] (-0.2,-0.4) -- (-0.2,0.4);
            \draw[->] (0.2,-0.4) -- (0.2,0.4);
            \multdot{-0.2,0}{east}{a};
            \multdot{0.2,0}{west}{b};
        \end{tikzpicture}
        & \text{if } r \le 0;
    \end{cases}
    &
    \begin{tikzpicture}[anchorbase]
        \draw[->] (0.3,-0.4) -- (-0.3,0.4);
        \draw[wipe] (-0.3,-0.4) -- (0.3,0.4);
        \draw[->] (-0.3,-0.4) -- (0.3,0.4);
        \multdot{-0.15,-0.2}{east}{r};
    \end{tikzpicture}
    &=
    \begin{cases}
        \begin{tikzpicture}[anchorbase]
            \draw[->] (-0.3,-0.4) -- (0.3,0.4);
            \draw[wipe] (0.3,-0.4) -- (-0.3,0.4);
            \draw[->] (0.3,-0.4) -- (-0.3,0.4);
            \multdot{0.15,0.2}{west}{r};
        \end{tikzpicture}
        + z\displaystyle\sum_{\substack{a+b=r \\ a,b \ge 0}}
        \begin{tikzpicture}[centerzero]
            \draw[->] (-0.2,-0.4) -- (-0.2,0.4);
            \draw[->] (0.2,-0.4) -- (0.2,0.4);
            \multdot{-0.2,0}{east}{a};
            \multdot{0.2,0}{west}{b};
        \end{tikzpicture}
        & \text{if } r \ge 0,
        \\
        \begin{tikzpicture}[anchorbase]
            \draw[->] (-0.3,-0.4) -- (0.3,0.4);
            \draw[wipe] (0.3,-0.4) -- (-0.3,0.4);
            \draw[->] (0.3,-0.4) -- (-0.3,0.4);
            \multdot{0.15,0.2}{west}{r};
        \end{tikzpicture}
        - z\displaystyle\sum_{\substack{a+b=r \\ a,b < 0}}
        \begin{tikzpicture}[centerzero]
            \draw[->] (-0.2,-0.4) -- (-0.2,0.4);
            \draw[->] (0.2,-0.4) -- (0.2,0.4);
            \multdot{-0.2,0}{east}{a};
            \multdot{0.2,0}{west}{b};
        \end{tikzpicture}
        & \text{if } r < 0;
    \end{cases}
    \\ \label{teaneg}
    \begin{tikzpicture}[anchorbase]
        \draw[->] (-0.3,-0.4) -- (0.3,0.4);
        \draw[wipe] (0.3,-0.4) -- (-0.3,0.4);
        \draw[->] (0.3,-0.4) -- (-0.3,0.4);
        \multdot{-0.15,-0.2}{east}{r};
    \end{tikzpicture}
    &=
    \begin{cases}
        \begin{tikzpicture}[anchorbase]
            \draw[->] (0.3,-0.4) -- (-0.3,0.4);
            \draw[wipe] (-0.3,-0.4) -- (0.3,0.4);
            \draw[->] (-0.3,-0.4) -- (0.3,0.4);
            \multdot{0.15,0.2}{west}{r};
        \end{tikzpicture}
        + z\displaystyle\sum_{\substack{a+b=r \\ a,b > 0}}
        \begin{tikzpicture}[centerzero]
            \draw[->] (-0.2,-0.4) -- (-0.2,0.4);
            \draw[->] (0.2,-0.4) -- (0.2,0.4);
            \multdot{-0.2,0}{east}{a};
            \multdot{0.2,0}{west}{b};
        \end{tikzpicture}
        & \text{if } r > 0,
        \\
        \begin{tikzpicture}[anchorbase]
            \draw[->] (0.3,-0.4) -- (-0.3,0.4);
            \draw[wipe] (-0.3,-0.4) -- (0.3,0.4);
            \draw[->] (-0.3,-0.4) -- (0.3,0.4);
            \multdot{0.15,0.2}{west}{r};
        \end{tikzpicture}
        - z\displaystyle\sum_{\substack{a+b=r \\ a,b \le 0}}
        \begin{tikzpicture}[centerzero]
            \draw[->] (-0.2,-0.4) -- (-0.2,0.4);
            \draw[->] (0.2,-0.4) -- (0.2,0.4);
            \multdot{-0.2,0}{east}{a};
            \multdot{0.2,0}{west}{b};
        \end{tikzpicture}
        & \text{if } r \le 0;
    \end{cases}
    &
    \begin{tikzpicture}[anchorbase]
        \draw[->] (-0.3,-0.4) -- (0.3,0.4);
        \draw[wipe] (0.3,-0.4) -- (-0.3,0.4);
        \draw[->] (0.3,-0.4) -- (-0.3,0.4);
        \multdot{0.15,-0.2}{west}{r};
    \end{tikzpicture}
    &=
    \begin{cases}
        \begin{tikzpicture}[anchorbase]
            \draw[->] (0.3,-0.4) -- (-0.3,0.4);
            \draw[wipe] (-0.3,-0.4) -- (0.3,0.4);
            \draw[->] (-0.3,-0.4) -- (0.3,0.4);
            \multdot{-0.15,0.2}{east}{r};
        \end{tikzpicture}
        - z\displaystyle\sum_{\substack{a+b=r \\ a,b \ge 0}}
        \begin{tikzpicture}[centerzero]
            \draw[->] (-0.2,-0.4) -- (-0.2,0.4);
            \draw[->] (0.2,-0.4) -- (0.2,0.4);
            \multdot{-0.2,0}{east}{a};
            \multdot{0.2,0}{west}{b};
        \end{tikzpicture}
        & \text{if } r \ge 0,
        \\
        \begin{tikzpicture}[anchorbase]
            \draw[->] (0.3,-0.4) -- (-0.3,0.4);
            \draw[wipe] (-0.3,-0.4) -- (0.3,0.4);
            \draw[->] (-0.3,-0.4) -- (0.3,0.4);
            \multdot{-0.15,0.2}{east}{r};
        \end{tikzpicture}
        + z\displaystyle\sum_{\substack{a+b=r \\ a,b < 0}}
        \begin{tikzpicture}[centerzero]
            \draw[->] (-0.2,-0.4) -- (-0.2,0.4);
            \draw[->] (0.2,-0.4) -- (0.2,0.4);
            \multdot{-0.2,0}{east}{a};
            \multdot{0.2,0}{west}{b};
        \end{tikzpicture}
        & \text{if } r < 0.
    \end{cases}
\end{align}

We define $(-)$-bubbles (see \cite[(2.18)]{BSW-qheis}) by
\begin{equation} \label{poppy}
    \leftminusside{r} := \leftbubside{r} - \leftplusside{r},\quad
    \rightminusside{r} := \rightbubside{r} - \rightplusside{r},\quad
    r \in \Z.
\end{equation}
We then have the \emph{infinite grassmannian relations} (\cite[Lem.~3.4]{BSW-qheis})
\begin{equation} \label{infgrass}
    \sum_{r+s=n} \leftplus{r} \rightplus{s}
    = \sum_{r+s=n} \leftminus{r} \rightminus{s}
    = - \delta_{n,0} z^{-2} 1_\one,\quad n \in \Z,
\end{equation}
and the relations
\begin{align} \label{weeds+}
    \leftplus{r} &= \delta_{r,-k} tz^{-1} 1_\one,\quad r \le -k,&
    \rightplus{r} &= - \delta_{r,k} t^{-1}z^{-1} 1_\one,\quad r \le k,
    \\ \label{weeds-}
    \rightminus{r} &= \delta_{r,0} tz^{-1} 1_\one,\quad r \ge 0,&
    \leftminus{r} &= -\delta_{r,0} t^{-1}z^{-1} 1_\one.\quad r \ge 0.
\end{align}

It will useful to express some our relations in terms of generating functions in an indeterminate $u$.  Define
\begin{align} \label{bg1}
    \leftplusgen &:= t^{-1}z \sum_{r \in \Z} \leftplus{r} u^{-r} \in u^k 1_\one + u^{k-1} \End_{\Heis_k}(\one) \llbracket u^{-1} \rrbracket,
    \\ \label{bg2}
    \rightplusgen &:= -tz \sum_{r \in \Z} \rightplus{r} u^{-r} \in u^{-k} 1_\one + u^{-k-1} \End_{\Heis_k}(\one) \llbracket u^{-1} \rrbracket,
    \\ \label{bg3}
    \leftminusgen &:= -tz \sum_{r \in \Z} \leftminus{r} u^{-r} \in 1_\one + u \End_{\Heis_k}(\one) \llbracket u \rrbracket,
    \\ \label{bg4}
    \rightminusgen &:= t^{-1}z \sum_{r \in \Z} \rightminus{r} u^{-r} \in 1_\one + u \End_{\Heis_k}(\one) \llbracket u \rrbracket.
\end{align}
Then \cref{infgrass} can be restated as
\begin{equation} \label{eyes}
    \leftplusgen \rightplusgen
    = \leftminusgen \rightminusgen
    = 1_\one.
\end{equation}

The following \emph{curl relations} hold for all $r \in \Z$ (\cite[Lem.~4.4]{BSW-qheis}):
\begin{align} \label{leftcurl}
    \begin{tikzpicture}[centerzero]
        \draw (0.2,-0.5) -- (0.2,-0.35) to[out=up,in=east] (-0.05,0.2) to[out=left,in=up] (-0.2,0);
        \draw[wipe] (-0.2,0) to[out=down,in=west] (-0.05,-0.2) to[out=right,in=down] (0.2,0.35) -- (0.2,0.5);
        \draw[->] (-0.2,0) to[out=down,in=west] (-0.05,-0.2) to[out=right,in=down] (0.2,0.35) -- (0.2,0.5);
        \multdot{-0.2,0}{east}{r};
    \end{tikzpicture}
    &= z\sum_{a \ge 0}
    \begin{tikzpicture}[centerzero]
        \draw[->] (0,-0.5) -- (0,0.5);
        \multdot{0,0}{west}{a};
        \plusleft{-0.5,0}{r-a};
    \end{tikzpicture}
    - z\sum_{a > 0}
    \begin{tikzpicture}[centerzero]
        \draw[->] (0,-0.5) -- (0,0.5);
        \multdot{0,0}{west}{-a};
        \minusleft{-0.5,0}{r+a};
    \end{tikzpicture}
    \ ,&
    \begin{tikzpicture}[centerzero]
        \draw[->] (-0.2,0) to[out=down,in=west] (-0.05,-0.2) to[out=right,in=down] (0.2,0.35) -- (0.2,0.5);
        \draw[wipe] (0.2,-0.5) -- (0.2,-0.35) to[out=up,in=east] (-0.05,0.2) to[out=left,in=up] (-0.2,0);
        \draw (0.2,-0.5) -- (0.2,-0.35) to[out=up,in=east] (-0.05,0.2) to[out=left,in=up] (-0.2,0);
        \multdot{-0.2,0}{east}{r};
    \end{tikzpicture}
    &= z\sum_{a > 0}
    \begin{tikzpicture}[centerzero]
        \draw[->] (0,-0.5) -- (0,0.5);
        \multdot{0,0}{west}{a};
        \plusleft{-0.5,0}{r-a};
    \end{tikzpicture}
    - z\sum_{a \ge 0}
    \begin{tikzpicture}[centerzero]
        \draw[->] (0,-0.5) -- (0,0.5);
        \multdot{0,0}{west}{-a};
        \minusleft{-0.5,0}{r+a};
    \end{tikzpicture}
    \ ,
    \\ \label{rightcurl}
    \begin{tikzpicture}[centerzero]
        \draw[->] (0.2,0) to[out=down,in=east] (0.05,-0.2) to[out=left,in=down] (-0.2,0.35) -- (-0.2,0.5);
        \draw[wipe] (-0.2,-0.5) -- (-0.2,-0.35) to[out=up,in=west] (0.05,0.2) to[out=right,in=up] (0.2,0);
        \draw (-0.2,-0.5) -- (-0.2,-0.35) to[out=up,in=west] (0.05,0.2) to[out=right,in=up] (0.2,0);
        \multdot{0.2,0}{west}{r};
    \end{tikzpicture}
    &= z\sum_{a \ge 0}
    \begin{tikzpicture}[centerzero]
        \draw[->] (0,-0.5) -- (0,0.5);
        \multdot{0,0}{east}{-a};
        \minusright{0.5,0}{r+a};
    \end{tikzpicture}
    - z\sum_{a > 0}
    \begin{tikzpicture}[centerzero]
        \draw[->] (0,-0.5) -- (0,0.5);
        \multdot{0,0}{east}{a};
        \plusright{0.5,0}{r-a};
    \end{tikzpicture}
    \ ,&
    \begin{tikzpicture}[centerzero]
        \draw (-0.2,-0.5) -- (-0.2,-0.35) to[out=up,in=west] (0.05,0.2) to[out=right,in=up] (0.2,0);
        \draw[wipe] (0.2,0) to[out=down,in=east] (0.05,-0.2) to[out=left,in=down] (-0.2,0.35) -- (-0.2,0.5);
        \draw[->] (0.2,0) to[out=down,in=east] (0.05,-0.2) to[out=left,in=down] (-0.2,0.35) -- (-0.2,0.5);
        \multdot{0.2,0}{west}{r};
    \end{tikzpicture}
    &= z\sum_{a > 0}
    \begin{tikzpicture}[centerzero]
        \draw[->] (0,-0.5) -- (0,0.5);
        \multdot{0,0}{east}{-a};
        \minusright{0.5,0}{r+a};
    \end{tikzpicture}
    - z\sum_{a \ge 0}
    \begin{tikzpicture}[centerzero]
        \draw[->] (0,-0.5) -- (0,0.5);
        \multdot{0,0}{east}{a};
        \plusright{0.5,0}{r-a};
    \end{tikzpicture}
    \ .
\end{align}

By \cite[Lem.~4.5]{BSW-qheis} (see \cite[Lem.~5.7]{BSW-qFrobHeis} for a proof, taking $A=\kk$ there) we have the following braid relation for alternating crossings:
\begin{align} \label{altbraid1}
    \begin{tikzpicture}[anchorbase]
        \draw[->] (-0.4,-0.5) -- (0.4,0.5);
        \draw[wipe] (0,-0.5) to[out=up, in=down] (-0.4,0) to[out=up,in=down] (0,0.5);
        \draw[<-] (0,-0.5) to[out=up, in=down] (-0.4,0) to[out=up,in=down] (0,0.5);
        \draw[wipe] (0.4,-0.5) -- (-0.4,0.5);
        \draw[->] (0.4,-0.5) -- (-0.4,0.5);
    \end{tikzpicture}
    \ -\
    \begin{tikzpicture}[anchorbase]
        \draw[->] (-0.4,-0.5) -- (0.4,0.5);
        \draw[wipe] (0,-0.5) to[out=up, in=down] (0.4,0) to[out=up,in=down] (0,0.5);
        \draw[<-] (0,-0.5) to[out=up, in=down] (0.4,0) to[out=up,in=down] (0,0.5);
        \draw[wipe] (0.4,-0.5) -- (-0.4,0.5);
        \draw[->] (0.4,-0.5) -- (-0.4,0.5);
    \end{tikzpicture}
    \ &= \sum_{\substack{r,s \ge 0 \\ a > 0}}\
    \begin{tikzpicture}[anchorbase]
        \draw[<-] (-0.2,0.5) to (-0.2,0.3) arc(180:360:0.2) to (0.2,0.5);
        \draw[->] (-0.2,-0.5) to (-0.2,-0.3) arc(180:0:0.2) to (0.2,-0.5);
        \draw[->] (-0.6,-0.2) arc(-90:270:0.2);
        \node at (-0.6,0) {\dotlabel{+}};
        \draw[->] (0.7,-0.5) to (0.7,0.5);
        \multdot{-0.8,0}{east}{-r-s-a};
        \multdot{0.2,0.3}{west}{r};
        \multdot{0.2,-0.3}{west}{s};
        \multdot{0.7,0}{west}{a};
    \end{tikzpicture}
    \qquad \text{if } k \ge 0,
    \\ \label{altbraid2}
    \begin{tikzpicture}[anchorbase]
        \draw[->] (0.4,-0.5) -- (-0.4,0.5);
        \draw[wipe] (0,-0.5) to[out=up, in=down] (-0.4,0) to[out=up,in=down] (0,0.5);
        \draw[<-] (0,-0.5) to[out=up, in=down] (-0.4,0) to[out=up,in=down] (0,0.5);
        \draw[wipe] (-0.4,-0.5) -- (0.4,0.5);
        \draw[->] (-0.4,-0.5) -- (0.4,0.5);
    \end{tikzpicture}
    \ -\
    \begin{tikzpicture}[anchorbase]
        \draw[->] (0.4,-0.5) -- (-0.4,0.5);
        \draw[wipe] (0,-0.5) to[out=up, in=down] (0.4,0) to[out=up,in=down] (0,0.5);
        \draw[<-] (0,-0.5) to[out=up, in=down] (0.4,0) to[out=up,in=down] (0,0.5);
        \draw[wipe] (-0.4,-0.5) -- (0.4,0.5);
        \draw[->] (-0.4,-0.5) -- (0.4,0.5);
    \end{tikzpicture}
    \ &= \sum_{\substack{r,s \ge 0 \\ a > 0}}\
    \begin{tikzpicture}[anchorbase]
        \draw[<-] (0.2,0.5) to (0.2,0.3) arc(360:180:0.2) to (-0.2,0.5);
        \draw[->] (0.2,-0.5) to (0.2,-0.3) arc(0:180:0.2) to (-0.2,-0.5);
        \draw[->] (0.6,-0.2) arc(-90:-450:0.2);
        \node at (0.6,0) {\dotlabel{+}};
        \draw[->] (-0.7,-0.5) to (-0.7,0.5);
        \multdot{0.8,0}{west}{-r-s-a};
        \multdot{-0.2,0.3}{east}{r};
        \multdot{-0.2,-0.3}{east}{s};
        \multdot{-0.7,0}{east}{a};
    \end{tikzpicture}
    \qquad \text{if } k \le 0.
\end{align}
For all other orientations of the strands, the usual braid relation holds.

%----------------------
\subsection{The center}
%----------------------

Recall that the \emph{center} of a monoidal category $\cC$ is the endomorphism algebra $\End_\cC(\one)$ of the unit object.  In this subsection, we describe the center of the quantum Heisenberg category and how elements of the center slide past strands.

Let $\Sym$ denote the ring of symmetric functions with coefficients in $\kk$.  For $r \in \N$, let $h_r$, $e_r$, and $p_r$ denote the $r$-th complete homogeneous, elementary, and power sum symmetric functions, respectively.  For $f \in \Sym$, define the following elements of $\Sym \otimes \Sym$:
\begin{equation} \label{electric}
    f^+ := f \otimes 1,\quad
    f^- := 1 \otimes f.
\end{equation}

\begin{prop} \label{cannon}
    We have an isomorphism
    \begin{equation}
        \beta \colon \Sym \otimes \Sym \to \End_{\Heis_k}(\one),
    \end{equation}
    given by, for $r \ge 1$,
    \begin{align} \label{hcannon}
        h_r^+ &\mapsto -tz \rightplus{r+k},&
        h_r^- &\mapsto t^{-1}z \rightminus{-r},
        \\ \label{ecannon}
        e_r^+ &\mapsto (-1)^r t^{-1}z\ \leftplus{r-k},&
        e_r^- &\mapsto (-1)^{r-1} tz\ \leftminus{-r},
        \\ \label{pcannon}
        p_r^+ &\mapsto z^2 \sum_{s \in \Z} s\ \leftplus{s} \rightplus{r-s},&
        p_r^- &\mapsto z^2 \sum_{s \in \Z} s\ \leftminus{-s} \rightminus{s-r}.
    \end{align}
\end{prop}

\begin{proof}
    The fact that we have an isomorphism $\beta$ given by \cref{hcannon,ecannon} was first shown in \cite[Cor.~10.2]{BSW-qheis}, although we use here the sign conventions of \cite[Cor.~9.3]{BSW-qFrobHeis} (where the Frobenius algebra $A$ there is $\kk$).  For the power sums, recall that $p_r = \sum_{s=0}^r (-1)^{s-1} s e_s h_{r-s}$; see \cite[p.~33]{Mac95}.  Also note that the maps \cref{hcannon,ecannon,pcannon} are valid for $r=0$ when we adopt the usual conventions that $h_0=e_0=1$ and $p_0=0$.  The image of $p_r^-$ given in \cref{pcannon} follows immediately.  For the image of $p_r^+$, we have
    \[
        p_r^+
        \mapsto z^2 \sum_{s \in \Z} s\ \leftplus{s-k} \rightplus{r-s+k}
        = z^2 \sum_{s \in \Z} (s+k)\ \leftplus{s} \rightplus{r-s}
        \overset{\cref{infgrass}}{=} z^2 \sum_{s \in \Z} s\ \leftplus{s} \rightplus{r-s}.
        \qedhere
    \]
\end{proof}

Next we recall how bubbles slide past strings.  The precise relation is easiest to state using the generating functions \cref{bg1,bg2,bg3,bg4} and dots labelled by formal power series.  We define
\[
    \begin{tikzpicture}[anchorbase]
        \draw[->] (0,-0.3) to (0,0.3);
        \multdot{0,0}{east}{x^r};
    \end{tikzpicture}
    :=
    \begin{tikzpicture}[anchorbase]
        \draw[->] (0,-0.3) to (0,0.3);
        \multdot{0,0}{east}{r};
    \end{tikzpicture}
    \ ,\quad r \in \Z.
\]
Then, expanding linearly, we can also label dots by polynomials $a_n x^n + \dotsb + a_1 x + a_0 \in \kk[x]$, or even by Laurent series in $\kk[x]\Laurent{u^{-1}}$ or $\kk[x]\Laurent{u}$.  For example, expanding in $\kk[x] \Laurent{u^{-1}}$, we have
\[
    \begin{tikzpicture}[anchorbase]
        \draw[->] (0,-0.3) to (0,0.3);
        \multdot{0,0}{east}{xu(u-x)^{-2}};
    \end{tikzpicture}
    = u^{-1}\,
    \begin{tikzpicture}[anchorbase]
        \draw[->] (0,-0.3) to (0,0.3);
        \singdot{0,0};
    \end{tikzpicture}
    + 2u^{-2}\,
    \begin{tikzpicture}[anchorbase]
        \draw[->] (0,-0.3) to (0,0.3);
        \multdot{0,0}{west}{2};
    \end{tikzpicture}
    + 3u^{-3}\,
    \begin{tikzpicture}[anchorbase]
        \draw[->] (0,-0.3) to (0,0.3);
        \multdot{0,0}{west}{3};
    \end{tikzpicture}
    + 4u^{-4}\,
    \begin{tikzpicture}[anchorbase]
        \draw[->] (0,-0.3) to (0,0.3);
        \multdot{0,0}{west}{4};
    \end{tikzpicture}
    + \dotsb.
\]
We adopt the convention that, in any equation involving the generating functions \cref{bg3,bg4}, we expand all rational functions as Laurent series in $\kk[x]\Laurent{u}$.  In all other equations, we expand rational functions as Laurent series in $\kk[x]\Laurent{u^{-1}}$.  With these conventions, we have the following \emph{bubble slides}:
\begin{equation} \label{bs}
    \begin{tikzpicture}[anchorbase]
        \draw[->] (0,-0.5) to (0,0.5);
        \pmrightgen{-1.1,0};
    \end{tikzpicture}
    \ =
    \begin{tikzpicture}[centerzero]
        \draw[->] (0,-0.5) to (0,0.5);
        \pmrightgen{0.5,0};
        \multdot{0,0}{east}{1-z^2xu(u-x)^{-2}};
    \end{tikzpicture}
    \ ,\qquad
    \begin{tikzpicture}[anchorbase]
        \draw[->] (0,-0.5) to (0,0.5);
        \pmleftgen{0.5,0};
    \end{tikzpicture}
    =\
    \begin{tikzpicture}[centerzero]
        \draw[->] (0,-0.5) to (0,0.5);
        \pmleftgen{-1.1,0};
        \multdot{0,0}{west}{1-z^2xu(u-x)^2};
    \end{tikzpicture}
    \ .
\end{equation}
(See \cite[Lem.~4.6]{BSW-qheis} for the statement and \cite[Lem.~5.6]{BSW-qFrobHeis} for a proof, taking $A = \kk$ there.)  In fact, as we see in the next result, the bubble slides are simpler when using the images under $\beta$ of the power sums.  For $r \in \Z$, we define the following element of $\End_{\Heis_k}(\one)$:
\begin{equation} \label{fort}
    \symbub{r} :=
    \begin{cases}
        -\{r\}^{-1} \beta(p_r^+) & \text{if } r > 0, \\
        0 & \text{if } r=0, \\
        -\{r\}^{-1} \beta(p_{-r}^-) & \text{if } r < 0.
    \end{cases}
\end{equation}
It follows that the center of $\Heis_k$ is a polynomial algebra in the $\symbub{r}$, $r \ne 0$:
\begin{equation} \label{renegade1}
    \End_{\Heis_k}(\one) = \kk \left[ \symbub{r} : r \in \Z,\ r \ne 0 \right],
\end{equation}
and we have
\begin{equation} \label{renegade2}
    \beta(\Sym \otimes 1) = \kk \left[ \symbub{r} : r>0 \right],\qquad
    \beta(1 \otimes \Sym) = \kk \left[ \symbub{r} : r<0 \right].
\end{equation}

\begin{prop} \label{goatee}
    For $r \in \Z$, we have
    \begin{equation} \label{goat}
        \symbub{r}\
        \begin{tikzpicture}[centerzero]
            \draw[->] (0,-0.5) -- (0,0.5);
        \end{tikzpicture}
        =
        \begin{tikzpicture}[centerzero]
            \draw[->] (0,-0.5) to (0,0.5);
        \end{tikzpicture}
        \ \symbub{r}
        + \{r\}\
        \begin{tikzpicture}[centerzero]
            \draw[->] (0,-0.5) to (0,0.5);
            \multdot{0,0}{west}{r};
        \end{tikzpicture}
        \ .
    \end{equation}
\end{prop}

\begin{proof}
    The statement is trivial for $r=0$.  Now suppose $r>0$ and consider the generating functions
    \begin{equation} \label{mars}
        H_+(u) := \sum_{r \ge 0} h_r^+ u^{-r},\qquad
        E_+(u) := \sum_{r \ge 0} e_r^+ u^{-r},\qquad
        P_+(u) := \sum_{r \ge 1} p_r^+ u^{1-r}.
    \end{equation}
    Then we have
    \begin{equation} \label{blob}
        H_+(u) E_+(-u) = 1,\qquad
        P_+(u) = -u^2 H_+'(u)/H_+(u) = -u^2H_+'(u) E_+(-u).
    \end{equation}
    (See, for example, \cite[(I.2.6) and (I.2.10)]{Mac95} setting the $t$ there equal to $u^{-1}$.)  Furthermore,
    \begin{equation} \label{road}
        \beta(H_+(u)) = u^k \rightplusblank(u),\qquad
        \beta(E_+(-u)) = u^{-k} \leftplusblank(u),
    \end{equation}
    Let $f(u)= 1-z^2xu(u-x)^{-2}$ be the rational function appearing in \cref{bs}.  Then we have
    \[
        \beta(H_+(u))\
        \begin{tikzpicture}[anchorbase]
            \draw[->] (0,-0.5) to (0,0.5);
        \end{tikzpicture}
        \ \overset{\cref{bs}}{=}
        \begin{tikzpicture}[centerzero]
            \draw[->] (0,-0.5) to (0,0.5);
            \multdot{0,0}{east}{f(u)};
        \end{tikzpicture}
        \ \beta(H_+(u)).
    \]
    Differentiating with respect to $u$ gives
    \[
        \beta(H_+'(u))\
        \begin{tikzpicture}[centerzero]
            \draw[->] (0,-0.5) -- (0,0.5);
        \end{tikzpicture}
        =
        \begin{tikzpicture}[centerzero]
            \draw[->] (0,-0.5) to (0,0.5);
            \multdot{0,0}{east}{f(u)};
        \end{tikzpicture}
        \ \beta(H_+'(u))
        +
        \begin{tikzpicture}[centerzero]
            \draw[->] (0,-0.5) to (0,0.5);
            \multdot{0,0}{east}{f'(u)};
        \end{tikzpicture}
        \ \beta(H_+(u)).
    \]
    Multiplying on the left by $-u^2\beta(E_+(-u) \otimes 1)$ and using \cref{bs} again, we have
    \[
        \beta(P_+(u))\
        \begin{tikzpicture}[centerzero]
            \draw[->] (0,-0.5) -- (0,0.5);
        \end{tikzpicture}
        =
        \begin{tikzpicture}[centerzero]
            \draw[->] (0,-0.5) to (0,0.5);
        \end{tikzpicture}
        \ \beta(P_+(u))
        +
        \begin{tikzpicture}[centerzero]
            \draw[->] (0,-0.5) to (0,0.5);
            \multdot{0,0}{east}{-u^2f'(u)/f(u)};
        \end{tikzpicture}
        \ .
    \]
    Now, noting that $f(u) = \frac{(u-q^2x)(u-q^{-2}x)}{(u-x)^2}$, we have
    \begin{multline*}
        \frac{f'(u)}{f(u)}
        = \frac{\partial}{\partial u} \ln (f(u))
        = \left( \frac{1}{u-q^2x} + \frac{1}{u-q^{-2}x} - \frac{2}{u-x} \right)
        \\
        = u^{-1} \sum_{r \ge 0} \left( \left( \frac{q^2x}{u} \right)^r + \left( \frac{q^{-2}x}{u} \right)^r - 2 \left( \frac{x}{u} \right)^r \right)
        = \sum_{r \ge 1} \{r\}^2 x^r u^{-r-1}.
    \end{multline*}
    Thus we have
    \[
        \beta(P_+(u))\
        \begin{tikzpicture}[centerzero]
            \draw[->] (0,-0.5) -- (0,0.5);
        \end{tikzpicture}
        =
        \begin{tikzpicture}[centerzero]
            \draw[->] (0,-0.5) to (0,0.5);
        \end{tikzpicture}
        \ \beta(P_+(u))
        - \sum_{r \ge 1} \{r\}^2 u^{1-r}\
        \begin{tikzpicture}[centerzero]
            \draw[->] (0,-0.5) to (0,0.5);
            \multdot{0,0}{west}{r};
        \end{tikzpicture}
        \ ,
    \]
    and \cref{goat} follows after dividing both sides by $-\{r\}$ and equating coefficients of $u$.  The case $r<0$ is similar, except that we work with power series in $u$, as opposed to $u^{-1}$.
    \details{
        We use the generating functions
        \[
            H_-(u) := \sum_{r \ge 0} h_r^- u^r,\qquad
            E_-(u) := \sum_{r \ge 0} e_r^- u^r,\qquad
            P_-(u) := \sum_{r \ge 1} p_r^- u^{r-1}.
        \]
        Then we have $H_-(u) E_-(-u) = 1$ and $P_-(u) = H_-'(u)/H_-(u) = H_-'(u) E_-(-u)$. (See, for example, \cite[(I.2.6) and (I.2.10)]{Mac95} setting the $t$ there equal to $u$.)  Furthermore,
        \[
            \beta(H_-(u)) = \rightminusblank(u),\qquad
            \beta(E_-(-u)) = \leftminusblank(u).
        \]
        As above, we have
        \[
            \beta(H_-'(u))\
            \begin{tikzpicture}[centerzero]
                \draw[->] (0,-0.5) -- (0,0.5);
            \end{tikzpicture}
            =
            \begin{tikzpicture}[centerzero]
                \draw[->] (0,-0.5) to (0,0.5);
                \multdot{0,0}{east}{f(u)};
            \end{tikzpicture}
            \ \beta(H_-'(u))
            +
            \begin{tikzpicture}[centerzero]
                \draw[->] (0,-0.5) to (0,0.5);
                \multdot{0,0}{east}{f'(u)};
            \end{tikzpicture}
            \ \beta(H_-(u)).
        \]
        Multiplying on the left by $\beta(E_-(-u))$ and using \cref{bs} again, we have
        \[
            \beta(P_-(u))\
            \begin{tikzpicture}[centerzero]
                \draw[->] (0,-0.5) -- (0,0.5);
            \end{tikzpicture}
            =
            \begin{tikzpicture}[centerzero]
                \draw[->] (0,-0.5) to (0,0.5);
            \end{tikzpicture}
            \ \beta(P_-(u))
            +
            \begin{tikzpicture}[centerzero]
                \draw[->] (0,-0.5) to (0,0.5);
                \multdot{0,0}{east}{f'(u)/f(u)};
            \end{tikzpicture}
            \ .
        \]
        Now, expanding as a Laurent series in $u$, we have
        \begin{multline*}
            \frac{f'(u)}{f(u)}
            = \frac{\partial}{\partial u} \ln (f(u))
            = \left( \frac{1}{u-q^2x} + \frac{1}{u-q^{-2}x} - \frac{2}{u-x} \right)
            \\
            = -\sum_{r \ge 0} \left( q^{-2r-2} x^{-r-1} u^r + q^{2r+2} x^{-r-1} u^r - 2x^{-r-1} u^r \right)
            = -\sum_{r \ge 1} \{r\}^2 x^{-r} u^{r-1}.
        \end{multline*}
        Dividing by $-\{r\}$ and equating coefficients of $u^r$ then yields \cref{goat}.
    }
\end{proof}

%-------------------------
\subsection{Basis theorem}
%-------------------------

We now recall the important basis theorem for the morphism spaces of $\Heis_k$.  Let $X = X_r \otimes \dotsb \otimes X_1$ and $Y = Y_s \otimes \dotsb \otimes Y_1$ be objects of $\Heis_k$ for $X_i, Y_j \in \{\uparrow,\downarrow\}$.  An \emph{$(X,Y)$-matching} is a bijection between $\{i : X_i = \uparrow\} \sqcup \{j : Y_j = \downarrow\}$ and $\{i : X_i = \downarrow\} \sqcup \{j : Y_j = \uparrow\}$.  A \emph{reduced lift} of an $(X,Y)$-matching is a diagram representing a morphism $X \rightarrow Y$ such that
\begin{itemize}
    \item the endpoints of each string are points corresponding under the given matching;
    \item there are no floating bubbles and no dots on any string;
    \item there are no self-intersections of strings and no two strings cross each other more than once.
\end{itemize}
Fix a set $B(X,Y)$ consisting of a choice of reduced lift for each of the $(X,Y)$-matchings.  Let $B_{\circ}(X,Y)$ be the set of all morphisms that can be obtained from the elements of $B(X,Y)$ by adding dots labelled with integer multiplicities near to the terminus of each string.  Using the morphism $\beta$ of \cref{cannon}, we can make the morphism space $\Hom_{\Heis_k}(X,Y)$ into a right $\Sym \otimes \Sym$-module:
\[
    \phi \theta := \phi \otimes \beta(\theta),\quad
    \phi \in \Hom_{\Heis_k}(X,Y),\quad \theta \in \Sym \otimes \Sym.
\]

\begin{theo}[{\cite[Th.~10.1]{BSW-qheis}}] \label{basisthm}
    For any objects $X, Y \in \Heis_k$, the morphism space\newline $\Hom_{\Heis_k}(X,Y)$ is a free right $\Sym \otimes \Sym$-module with basis $B_{\circ}(X,Y)$.
\end{theo}

By \cite[Th.~3.2]{BSW-qheis}, there is a unique isomorphism of $\kk$-linear monoidal categories
\begin{equation} \label{earth}
    \Omega_k \colon \Heis_k(z,t) \xrightarrow{\cong} \Heis_{-k}(z,t^{-1})^\op
\end{equation}
given on the generating morphisms by
\begin{equation} \label{solaris}
    \posupcross \mapsto -\negdowncross,\quad
    \negupcross \mapsto -\posdowncross,\quad
    \updot \mapsto \downdot,\quad
    \rightcup \mapsto \rightcap,\quad
    \rightcap \mapsto \rightcup,\quad
    \leftcup \mapsto -\leftcap,\quad
    \leftcap \mapsto -\leftcup.
\end{equation}
The isomorphism $\Omega_k$ acts on bubbles as
\begin{equation} \label{aurora}
    \leftpm{r} \mapsto - \rightpm{r}
    \ ,\qquad
    \rightpm{r} \mapsto - \leftpm{r}
    \ ,\qquad
    \symbub{s} \mapsto - \symbub{s},
\end{equation}
for $r,s \in \Z$, $s \ne 0$.
\details{
    Since the image of the power sums is not stated in \cite[Th.~3.2]{BSW-qheis}, we include the details here.  For $s \ge 1$, we have
    \begin{multline*}
        \symbub{s}
        \overset{\cref{pcannon}}{=} z^2 \sum_{a \in \Z} a\ \leftplus{a} \rightplus{s-a}
        \mapsto z^2 \sum_{a \in \Z} a \leftplus{s-a} \rightplus{a}
        = z^2 \sum_{b \in \Z} (s-b)\ \leftplus{b} \rightplus{s-b}
        \\
        \overset{\cref{infgrass}}{=} - z^2 \sum_{b \in \Z} b\ \leftplus{b} \rightplus{s-b}
        = - \symbub{s}
    \end{multline*}
    and for $s \le -1$, we have
    \begin{multline*}
        \symbub{s}
        \overset{\cref{pcannon}}{=} z^2 \sum_{a \in \Z} a\ \leftminus{-a} \rightminus{a+s}
        \mapsto z^2 \sum_{a \in \Z} a \leftminus{a+s} \rightminus{-a}
        = z^2 \sum_{b \in \Z} (-s-b)\ \leftminus{-b} \rightminus{b+s}
        \\
        \overset{\cref{infgrass}}{=} - z^2 \sum_{b \in \Z} b \leftminus{-b} \rightminus{b+s}
        = - \symbub{s}.
    \end{multline*}
}

%=================================================================
\section{Partial quantum Heisenberg categories\label{sec:partial}}
%=================================================================

In this short section we define certain subcategories of $\Heis_k$, which we call \emph{partial quantum Heisenberg categories}.  Roughly speaking, our aim is to split $\Heis_k$ in half in such a way that, when we later identify its trace with the elliptic Hall algebra $\EH_k$, the two halves correspond to $\EH^+$ and $\EH^-$.  Throughout this section we continue with the assumptions on $\kk$, $q$, $z$, $t$, and $k$ made at the beginning of \cref{sec:qheis}.

\begin{defin}
    Define $\Heis^+$ to be the strict $\kk$-linear monoidal category generated by the object $\uparrow$ and morphisms
    \[
        \posupcross,\ \negupcross \colon \uparrow \otimes \uparrow\ \to\ \uparrow \otimes \uparrow,\qquad
        \updot \colon \uparrow\ \to\ \uparrow,\qquad
        \symbub{r} \colon \one \to \one,\ r > 0,
    \]
    subject to the relations \cref{braid,skein,dotslide}, relation \cref{goat} for $r>0$, and the relation that the dot is invertible.  Define $\Heis^-$ to be the strict $\kk$-linear monoidal category generated by the object $\downarrow$ and morphisms
    \[
        \posdowncross,\ \negdowncross \colon \downarrow \otimes \downarrow\ \to\ \downarrow \otimes \downarrow,\qquad
        \downdot \colon \downarrow\ \to\ \downarrow,\qquad
        \symbub{r} \colon \one \to \one,\ r < 0,
    \]
    subject to the $180\degree$ rotation of relations \cref{braid,skein,dotslide} and \cref{goat} for $r<0$, and the relation that the dot is invertible.
\end{defin}
Note that the definition of $\Heis^\pm$ does not involve $k$.

\begin{prop} \label{yurt}
    For $k \in \Z$, we have faithful $\kk$-linear monoidal functors
    \[
        \Psi_k^\pm \colon \Heis^\pm \to \Heis_k,
    \]
    mapping the generating objects and morphisms of $\Heis^\pm$ to the objects and morphisms in $\Heis_k$ denoted by the same symbols.
\end{prop}

\begin{proof}
    We give the proof for $\Psi_k^+$, since the proof for $\Psi_k^-$ is analogous.  Because all generating morphisms of $\Heis^+$ are endomorphisms, we have $\Hom_{\Heis^+}(\uparrow^{\otimes n}, \uparrow^{\otimes m}) = 0$ for $n \ne m$.  Since the defining relations of $\Heis^+$ hold in $\Heis_k$, the functor is well defined.

    Fix $n \in \N$ define $B_0(\uparrow^{\otimes n}, \uparrow^{\otimes n})$ as before \cref{basisthm}.  Let $\bB_{\Sym}$ be the basis of $\Sym$ consisting of the power sum symmetric functions $p_\lambda = p_{\lambda_1} \dotsb p_{\lambda_l}$, where $\lambda = (\lambda_1,\dotsc,\lambda_l)$ is a partition.  A standard straightening argument shows that the morphisms
    \[
        \phi \otimes \beta(f \otimes 1),\quad
        \phi \in B_\circ(\uparrow^{\otimes n}, \uparrow^{\otimes n}),\quad
        f \in \bB_{\Sym},
    \]
    span $\Hom_{\Heis^+}(\uparrow^{\otimes n}, \uparrow^{\otimes n})$.  (See for example, the proof of \cite[Th.~10.1]{BSW-qheis}.)  By \cref{basisthm}, the images of these morphisms under $\Psi_k^+$ are linearly independent in $\Heis_k$.
\end{proof}

For $k \in \Z$, define
\[
    \Heis_k^\pm := \Psi_k^\pm(\Heis^\pm).
\]
It follows that elements of $\beta(\Sym \otimes 1)$ are morphisms in $\Heis_k^+$, while elements of $\beta(1 \otimes \Sym)$ are morphisms in $\Heis_k^-$.  It also follows from \cref{yurt} that we have isomorphisms of $\kk$-linear monoidal categories
\begin{equation} \label{eagle}
    \Heis_k^\pm \cong \Heis_l^\pm,\quad k,l \in \Z.
\end{equation}

%===================================
\section{Skein algebra of the torus}
%===================================

In this section we recall the definition of the skein algebra of the torus and identify it with the trace of the quantum Heisenberg category of central charge zero.  Throughout this section we work over an arbitrary commutative ground ring $\kk$ and $z,t \in \kk^\times$.  (Although we introduced $\Heis_k$ in \cref{sec:qheis} under more restrictive assumptions on $\kk$, all the results used in the current section hold more generally, as shown in \cite{BSW-qheis}.)

Consider the annulus
\begin{equation}
    A = [0,1]^2/\sim,
\end{equation}
where $\sim$ is the relation given by $(0,b) \sim (1,b)$ for all $b \in (0,1)$.  We will denote points in $A$ by representatives of the equivalence classes under $\sim$.  In order to make the categories we are about to describe strict, we fix a countable number of points in $A$, which will be the possible endpoints of tangles.  We choose the points
\begin{equation} \label{endpoints}
    P_n = \left( 1 - \tfrac{1}{2^n}, \tfrac{1}{2} \right) \in A,\quad n \in \Z_{>0}.
\end{equation}
Up to isomorphism, our categories will not depend on the particular choice of points.  We will typically draw them as equally spaced, or adjust the spacing to the particular tangle we draw.

We let $\FOT(A)$ be the category of framed oriented tangles over $A$.  Its objects are finite sequences $(\varepsilon_1,\dotsc,\varepsilon_n)$ of elements of $\{\uparrow,\downarrow\}$.  The unit object $\one$ is the empty sequence.  Morphisms in $\FOT(A)$ from $(\varepsilon_1,\dotsc,\varepsilon_m)$ to $(\varepsilon'_1,\dotsc,\varepsilon'_n)$ are framed oriented tangles in $A \times [0,1]$, up to ambient isotopy, with endpoints
\[
    \left( \{P_1,\dotsc,P_m\} \times \{0\} \right) \cup \left( \{P_1,\dotsc,P_n\} \times \{1\} \right)
\]
such that the orientation of the tangle at each $P_i \times \{0\}$ agrees with $\varepsilon_i$, the orientation at each $P_i' \times \{1\}$ agrees with $\varepsilon'_i$, and the framing at the point $P_i \times \{0\}$ (respectively, $P_i \times \{1\}$) points towards $P_{i+1} \times \{0\}$ (respectively, $P_{i+1} \times \{1\}$).  We allow tangles to have closed components.  For example,
\begin{equation} \label{pizza}
    \begin{tikzpicture}[anchorbase]
        \draw[->] (0.5,-0.6) to[out=up,in=225] (1.8,0.2);
        \draw[->] (-1.3,0.2) to[out=45,in=down] (-1,1);
        \draw[wipe] (1,-0.6) -- (1,-0.3) to[out=up,in=down] (-0.5,1);
        \draw[<-] (1,-0.6) -- (1,-0.3) to[out=up,in=down] (-0.5,1);
    	\draw[wipe] (0,1) to (0,0.3);
    	\draw[<-] (0,1) to (0,0.3);
    	\draw (0.3,-0.2) to[out=0,in=-90](.5,0);
    	\draw (0.5,0) to[out=90,in=0](0.3,0.2);
    	\draw (0,-0.3) to (0,-0.6);
    	\draw (0,0.3) to[out=-90,in=180] (0.3,-0.2);
    	\draw[wipe] (0.3,0.2) to[out=180,in=90](0,-0.3);
    	\draw (0.3,0.2) to[out=180,in=90](0,-0.3);
        \draw (1.2,0.6) to[out=up,in=up,looseness=1.5] (1.6,0.6);
        \draw[wipe] (0.5,1) to[out=down,in=down,looseness=2] (1.5,1);
        \draw[->] (0.5,1) to[out=down,in=down,looseness=2] (1.5,1);
        \draw[wipe] (1.2,0.6) to[out=down,in=down,looseness=1.5] (1.6,0.6);
        \draw[->] (1.2,0.6) to[out=down,in=down,looseness=1.5] (1.6,0.6);
        \identify{-1.3}{-0.6}{1.8}{1};
    \end{tikzpicture}
    \in \Hom_{\FOT(A)}(\uparrow \otimes \uparrow \otimes \downarrow, \uparrow \otimes \downarrow \otimes \uparrow \otimes \downarrow \otimes \uparrow),
\end{equation}
where we adopt the convention of blackboard framing (i.e.\ the framing is parallel to the page) and we identify the dashed vertical edges.  We always isotope tangles so that they intersect the cut transversely.  The composite $f \circ g$ is given by placing $f$ above $g$ and rescaling the vertical coordinate.  The category $\FOT(A)$ is a strict monoidal category.  Viewing $A \times [0,1]$ as the cylinder, the tensor product $f \otimes g$ is given by placing the cylinder for $g$ inside the cylinder for $f$, then rescaling and isotoping the endpoints of the tangles so that the endpoints of $g$ are to the right of those of $f$ (preserving the relative order of the endpoints in $f$ and the endpoints in $g$).  In terms of diagrams as in \cref{pizza}, this corresponds to placing the diagram of $g$ to the right of the diagram of $f$, and then passing all strands of $f$ exiting the right side of its diagram over the diagram for $g$ and all strands of $g$ exiting the left side of its diagram under the diagram for $f$.  For example,
\[
    \begin{tikzpicture}[centerzero]
        \identify{-0.7}{-0.5}{0.7}{0.5};
        \draw[->] (-0.3,-0.5) \braidup (0.3,0.5);
        \draw[->] (0.3,-0.5) to[out=up,in=200] (0.7,0);
        \draw[->] (-0.7,0) to[out=20,in=down] (-0.3,0.5);
    \end{tikzpicture}
    \ \otimes\
    \begin{tikzpicture}[anchorbase]
        \identify{-0.7}{-0.5}{0.7}{0.5};
        \draw[->] (0,-0.5) \braidup (-0.3,0.5);
        \draw[<-] (0.3,-0.5) \braidup (0,0.5);
        \draw[->] (-0.3,-0.5) to[out=up,in=-20] (-0.7,0);
        \draw[->] (0.7,0) to[out=160,in=down] (0.3,0.5);
    \end{tikzpicture}
    \ =\
    \begin{tikzpicture}[centerzero]
        \identify{-0.9}{-0.5}{0.9}{0.5};
        \draw[->] (0,-0.5) to[out=up,in=-45] (-0.9,-0.15);
        \draw[->] (0.3,-0.5) \braidup (0,0.5);
        \draw[<-] (0.6,-0.5) \braidup (0.3,0.5);
        \draw[->] (0.9,-0.15) to[out=135,in=down] (0.6,0.5);
        \draw[wipe] (-0.3,-0.5) to[out=up,in=225] (0.9,0.15);
        \draw[->] (-0.3,-0.5) to[out=up,in=225] (0.9,0.15);
        \draw[->] (-0.9,0.15) to[out=45,in=down] (-0.6,0.5);
        \draw[wipe] (-0.6,-0.5) \braidup (-0.3,0.5);
        \draw[->] (-0.6,-0.5) \braidup (-0.3,0.5);
    \end{tikzpicture}
    \ .
\]

Let $\FOT(A)_\kk$ denote the $\kk$-linearization of $\FOT(A)$.  Thus, the morphisms in $\FOT(A)_\kk$ are formal $\kk$-linear combinations of morphisms in $\FOT(A)$, with composition and tensor product extended by linearity.  The \emph{framed HOMFLYPT skein category} $\OS(A;z,t)$ over the annulus is the category obtained from $\FOT(A)_\kk$ by imposing the relations
\begin{equation} \label{metalox}
    \posupcross - \negupcross = z\
    \begin{tikzpicture}[centerzero]
        \draw[->] (-0.15,-0.2) -- (-0.15,0.2);
        \draw[->] (0.15,-0.2) -- (0.15,0.2);
    \end{tikzpicture}
    \ ,\qquad
    \begin{tikzpicture}[centerzero]
        \draw (-0.2,-0.5) -- (-0.2,-0.35) to[out=up,in=west] (0.05,0.2) to[out=right,in=up] (0.2,0);
        \draw[wipe] (0.2,0) to[out=down,in=east] (0.05,-0.2) to[out=left,in=down] (-0.2,0.35) -- (-0.2,0.5);
        \draw[->] (0.2,0) to[out=down,in=east] (0.05,-0.2) to[out=left,in=down] (-0.2,0.35) -- (-0.2,0.5);
    \end{tikzpicture}
    \ = t^{-1}\
    \begin{tikzpicture}[centerzero]
        \draw[->] (0,-0.5) to (0,0.5);
    \end{tikzpicture}
    \ ,\qquad
    \begin{tikzpicture}[centerzero]
        \draw[->] (0.2,0) arc(360:0:0.2);
    \end{tikzpicture}
    = \frac{t - t^{-1}}{z} 1_\one.
\end{equation}
Note that these are precisely the relations \cref{skein,curls1} with $k=0$.  In fact, we have the following result, which states that the framed HOMFLYPT skein category over the annulus is the quantum Heisenberg category at central charge zero.

\begin{prop}[{\cite[Cor.~7.7]{MS20}}] \label{Polly}
    We have an isomorphism of monoidal categories
    \begin{equation}
        \Heis_0 \xrightarrow{\cong} \OS(A;z,t).
    \end{equation}
    This isomorphism sends the generators $\posupcross$, $\negupcross$, $\rightcup$, $\rightcap$, $\leftcup$, and $\leftcap$ to the tangles with the same diagrams, and the image of the dots are
    \[
        \updot \mapsto
        \begin{tikzpicture}[centerzero]
            \identify{-0.3}{-0.3}{0.3}{0.3};
            \draw[->] (0,-0.3) to[out=up,in=200] (0.3,0);
            \draw[->] (-0.3,0) to[out=20,in=south] (0,0.3);
        \end{tikzpicture}
        \ ,\qquad
        \downdot \mapsto
        \begin{tikzpicture}[centerzero]
            \identify{-0.3}{-0.3}{0.3}{0.3};
            \draw[<-] (0,-0.3) to[out=up,in=-20] (-0.3,0);
            \draw[<-] (0.3,0) to[out=160,in=south] (0,0.3);
        \end{tikzpicture}
        \ .
    \]
\end{prop}

Now consider the torus
\begin{equation}
    T^2 = [0,1]^2 / \approx
\end{equation}
where $\approx$ is the relation given by $(0,b) \approx (1,b)$ for all $b \in [0,1]$, and $(a,0) \approx (a,1)$ for all $a \in [0,1]$.  We let $\Sk(T^2;z,t)$ be the skein algebra of the torus.  As a $\kk$-module, this is space of $\kk$-linear combinations of framed oriented links in $T^2 \times [0,1]$, up to isotopy, modulo the relations \cref{metalox}.  The product in $\Sk(T^2;z,t)$ is defined as follows.  Consider the two embeddings
\begin{align*}
    \imath_1 &\colon T^2 \otimes [0,1] \hookrightarrow T^2 \otimes [0,1],&
    (a,b) &\mapsto (a,(b+2)/3),
    \\
    \imath_2 &\colon T^2 \otimes [0,1] \hookrightarrow T^2 \otimes [0,1],&
    (a,b) &\mapsto (a,b/3).
\end{align*}
Then, for $x,y \in \Sk(T^2;z,t)$, we define $xy := \imath_1(x) \sqcup \imath_2(x)$.  Intuitively, the product $xy$ is given by stacking $x$ above $y$.

\begin{prop} \label{Digory}
    We have an isomorphism of algebras
    \[
        \Tr(\OS(A;z,t)) \cong \Sk(T^2;z,t).
    \]
\end{prop}

\begin{proof}
    Consider the natural surjection $A \twoheadrightarrow T^2$ sending the equivalence class of $(a,b)$ under $\sim$ to the equivalence class of $(a,b)$ under $\approx$.  Under this surjection, any endomorphism in $\OS(A;z,t)$ can be viewed as an element of $\Sk(T^2;z,t)$ by identifying the top and bottom of the string diagrams in $\OS(A;z,t)$.  This clearly descends to an algebra homomorphism $f \colon \Tr(\OS(A;z,t)) \to \Sk(T^2;z,t)$.  Conversely, we can isotope framed oriented tangles in $T^2$ so that they intersect the circle $\{(a,0) : a \in [0,1]\} = \{(a,1) : b \in [0,1]\} \subseteq T^2$ transversely.  Then, cutting along this circle gives a map $g \colon \Sk(T^2;z,t) \to \Tr(\OS(A;z,t))$; the trace condition \cref{dizzy} ensures that this map is well-defined.  It is straightforward to verify that $f$ and $g$ are mutually inverse.
\end{proof}

\begin{cor} \label{chipmunk}
    We have an isomorphism of algebras
    \[
        \Sk(T^2;z,t) \xrightarrow{\cong} \Tr(\Heis_0).
    \]
\end{cor}

\begin{proof}
    This follows immediately from \cref{Polly,Digory}.
\end{proof}

%=================================================================
\section{Trace of the quantum Heisenberg category\label{sec:main}}
%=================================================================

In this section, we prove our main result (\cref{mainthm}).  Namely, we describe an algebra isomorphism from $\EH_k$ to the trace of the quantum Heisenberg category.  Throughout this section we continue with the assumptions on $\kk$, $q$, $z$, $t$, and $k$ made at the beginning of \cref{sec:qheis}.

Let $\Heis_k$ be the \emph{quantum Heisenberg category} introduced in \cref{sec:qheis}.   Before stating our main result, we introduce some notation.  For $i,j \in \N$, define
\begin{align}
    \sigma_{i,j} &:=
    \begin{tikzpicture}[centerzero]
        \draw[->] (0.5,-0.6) -- (0,0.6);
        \node at (0.65,0.2) {$\cdots$};
        \draw[->] (1.5,-0.6) -- (1,0.6);
        \draw[wipe] (0,-0.6) to[out=45,in=225] (3,0.6);
        \draw[->] (0,-0.6) to[out=45,in=225] (3,0.6);
        \draw[wipe] (2,-0.6) -- (1.5,0.6);
        \draw[->] (2,-0.6) -- (1.5,0.6);
        \node at (2.35,-0.2) {$\cdots$};
        \draw[wipe] (3,-0.6) -- (2.5,0.6);
        \draw[->] (3,-0.6) -- (2.5,0.6);
    \end{tikzpicture}
    \in \End_{\Heis_k} (\uparrow^{\otimes (i+j+1)}),
    \\
    \sigma_{-i,-j} &:=
    \begin{tikzpicture}[centerzero,rotate=180]
        \draw[->] (0.5,-0.6) -- (0,0.6);
        \node at (0.65,0.2) {$\cdots$};
        \draw[->] (1.5,-0.6) -- (1,0.6);
        \draw[wipe] (0,-0.6) to[out=45,in=225] (3,0.6);
        \draw[->] (0,-0.6) to[out=45,in=225] (3,0.6);
        \draw[wipe] (2,-0.6) -- (1.5,0.6);
        \draw[->] (2,-0.6) -- (1.5,0.6);
        \node at (2.35,-0.2) {$\cdots$};
        \draw[wipe] (3,-0.6) -- (2.5,0.6);
        \draw[->] (3,-0.6) -- (2.5,0.6);
    \end{tikzpicture}
    \in \End_{\Heis_k} (\downarrow^{\otimes (i+j+1)}),
\end{align}
where the strand crossing all the others passes over $i$ strands and under $j$ strands.  Then define
\begin{equation} \label{Freud}
    \sigma_n := \frac{z}{\{n\}} \sum_{i=0}^{n-1} \sigma_{i,n-i-1},\qquad
    \sigma_{-n} := \frac{z}{\{n\}} \sum_{i=0}^{n-1} \sigma_{-i,i-n+1}.
\end{equation}
Note that $\sigma_1 = 1_\uparrow$ and $\sigma_{-1} = 1_\downarrow$.

For $r \in \Z$, $n \in \Z_{>0}$, define
\begin{align}
    \chi_{r,n} &:=
    \frac{\{r\}}{\{rn\}}
    \left(
        \begin{tikzpicture}[centerzero]
            \draw[->] (0,-0.2) -- (0,0.2);
            \draw[->] (0.4,-0.2) -- (0.4,0.2);
            \node at (0.83,0) {$\cdots$};
            \draw[->] (1.2,-0.2) -- (1.2,0.2);
            \multdot{0,0}{east}{r};
        \end{tikzpicture}
       \ +\
        \begin{tikzpicture}[centerzero]
            \draw[->] (0,-0.2) -- (0,0.2);
            \draw[->] (0.4,-0.2) -- (0.4,0.2);
            \node at (0.83,0) {$\cdots$};
            \draw[->] (1.2,-0.2) -- (1.2,0.2);
            \multdot{0.4,0}{east}{r};
        \end{tikzpicture}
        \ + \dotsb +\
        \begin{tikzpicture}[centerzero]
            \draw[->] (0,-0.2) -- (0,0.2);
            \draw[->] (0.4,-0.2) -- (0.4,0.2);
            \node at (0.83,0) {$\cdots$};
            \draw[->] (1.2,-0.2) -- (1.2,0.2);
            \multdot{1.2,0}{west}{r};
        \end{tikzpicture}
    \right)
    \in \End_{\Heis_k}(\uparrow^{\otimes n}),
    \\
    \chi_{r,-n} &:=
    \frac{\{r\}}{\{rn\}}
    \left(
        \begin{tikzpicture}[centerzero]
            \draw[<-] (0,-0.2) -- (0,0.2);
            \draw[<-] (0.4,-0.2) -- (0.4,0.2);
            \node at (0.83,0) {$\cdots$};
            \draw[<-] (1.2,-0.2) -- (1.2,0.2);
            \multdot{0,0}{east}{r};
        \end{tikzpicture}
       \ +\
        \begin{tikzpicture}[centerzero]
            \draw[<-] (0,-0.2) -- (0,0.2);
            \draw[<-] (0.4,-0.2) -- (0.4,0.2);
            \node at (0.83,0) {$\cdots$};
            \draw[<-] (1.2,-0.2) -- (1.2,0.2);
            \multdot{0.4,0}{east}{r};
        \end{tikzpicture}
        \ + \dotsb +\
        \begin{tikzpicture}[centerzero]
            \draw[<-] (0,-0.2) -- (0,0.2);
            \draw[<-] (0.4,-0.2) -- (0.4,0.2);
            \node at (0.83,0) {$\cdots$};
            \draw[<-] (1.2,-0.2) -- (1.2,0.2);
            \multdot{1.2,0}{west}{r};
        \end{tikzpicture}
    \right)
    \in \End_{\Heis_k}(\downarrow^{\otimes n}),
\end{align}
where we adopt the convention that $\frac{\{r\}}{\{rn\}} = \frac{1}{n}$ when $r=0$, so that $\chi_{0,n} = 1_\uparrow^{\otimes n}$ and $\chi_{0,-n} = 1_\downarrow^{\otimes n}$

\begin{theo} \label{mainthm}
    For $k \in \Z$, there is a unique isomorphism of algebras
    \[
        \varphi_k \colon \EH_k \xrightarrow{\cong} \Tr(\Heis_k)
    \]
    such that
    \begin{equation} \label{sky}
        w_{r,1} \mapsto \left[ \multupdot{r} \right],\quad
        w_{r,-1} \mapsto \left[ \multdowndot{r} \right],\quad
        r \in \Z.
    \end{equation}
    Under $\varphi_k$, we also have
    \begin{align} \label{clouds}
        w_{r,0} &\mapsto \left[ \symbub{r} \right],& r \in \Z,\ r \ne 0,
        \\ \label{sun}
        w_{r,n} &\mapsto [ \chi_{r,n} \sigma_n ],& r,n \in \Z,\ n \ne 0.
    \end{align}
\end{theo}

The proof of~\cref{mainthm} is given at the end of this section, after some preparatory results.  Note that the definition of $\EH_k$ and the isomorphism of \cref{mainthm} are independent of $t$, even though the definition of quantum Heisenberg category $\Heis_k$ involves $t$.

Recall the notation $\equiv$ from \cref{equivdef}, which we will use frequently in this section.

\begin{prop} \label{moon}
    The $k=0$ case of \cref{mainthm} holds.
\end{prop}

\begin{proof}
    It follows from \cref{chipmunk} and \cite[Th.~2]{MS17} that we have an isomorphism
    \begin{equation} \label{canal}
        \varphi_0 \colon \EH_0 \xrightarrow{\cong} \Tr(\Heis_0)
    \end{equation}
    satisfying \cref{sky}.  (Recall that our $q$ and $t$ are the $s$ and $v^{-1}$ of \cite{MS17}, respectively.  While \cite{MS17} works over the ground ring $\kk = \C[q^{\pm 1}, t^{\pm}, \{d\}^{-1} : d \ge 1]$, the result we use here holds more generally; see \cref{ground}.)  More precisely, under the isomorphism of \cref{chipmunk}, we have
    \[
        \begin{tikzpicture}[centerzero]
            \draw[red,dashed] (-0.5,-0.5) rectangle (0.5,0.5);
            \draw[->] (0,-0.5) to[out=up,in=194] (0.5,-0.25);
            \draw[->] (-0.5,-0.25) -- (0.5,0);
            \draw[->] (-0.5,0) -- (0.5,0.25);
            \draw[->] (-0.5,0.25) to[out=14,in=down] (0,0.5);
        \end{tikzpicture}
        \ \mapsto
        \left[ \multupdot{r} \right]
        ,\qquad
        \begin{tikzpicture}[centerzero]
            \draw[red,dashed] (-0.5,-0.5) rectangle (0.5,0.5);
            \draw[<-] (0,-0.5) to[out=up,in=-14] (-0.5,-0.25);
            \draw[<-] (0.5,-0.25) -- (-0.5,0);
            \draw[<-] (0.5,0) -- (-0.5,0.25);
            \draw[<-] (0.5,0.25) to[out=166,in=down] (0,0.5);
        \end{tikzpicture}
        \ \mapsto
        \left[ \multdowndot{r} \right].
    \]
    where we identify the vertical dashed edges with each other and the horizontal dashed edges with each other, and the curves wrap $r$ times in the horizontal direction (we have drawn the case $r=3$).  Since $\pm 1$ is coprime to $r$ (see \cite[p.~810]{MS17}), under the isomorphism of \cite[Th.~2]{MS17} we have
    \[
        w_{r,1} \mapsto
        \begin{tikzpicture}[centerzero]
            \draw[red,dashed] (-0.5,-0.5) rectangle (0.5,0.5);
            \draw[->] (0,-0.5) to[out=up,in=194] (0.5,-0.25);
            \draw[->] (-0.5,-0.25) -- (0.5,0);
            \draw[->] (-0.5,0) -- (0.5,0.25);
            \draw[->] (-0.5,0.25) to[out=14,in=down] (0,0.5);
        \end{tikzpicture}
        \ ,\qquad
        w_{r,-1} \mapsto
        \begin{tikzpicture}[centerzero]
            \draw[red,dashed] (-0.5,-0.5) rectangle (0.5,0.5);
            \draw[<-] (0,-0.5) to[out=up,in=-14] (-0.5,-0.25);
            \draw[<-] (0.5,-0.25) -- (-0.5,0);
            \draw[<-] (0.5,0) -- (-0.5,0.25);
            \draw[<-] (0.5,0.25) to[out=166,in=down] (0,0.5);
        \end{tikzpicture}
        \ ,
    \]
    where again the curves wrap $r$ times in the horizontal direction.  (Note that the skein of the torus $\Sk(T^2;z,t)$ is denoted $H(T^2)$ in \cite{MS17}.)   On the other hand, by \cref{slime}, the isomorphism $\varphi_0$ is uniquely determined by where it maps $w_{r,\pm 1}$, $r \in \Z$.

    To show \cref{clouds}, it suffices to prove that, for $r \in \Z$, $r \ne 0$,
    \begin{equation} \label{cirrus}
        -\{r\} \left[ \symbub{r} \right]
        =
        \left[
            \begin{tikzpicture}[centerzero]
                \draw[->] (-0.15,-0.2) -- (-0.15,0.2);
                \draw[<-] (0.15,-0.2) -- (0.15,0.2);
                \multdot{-0.15,0}{east}{r};
            \end{tikzpicture}
        \right]
        -
        \left[
            \begin{tikzpicture}[centerzero]
                \draw[<-] (-0.15,-0.2) -- (-0.15,0.2);
                \draw[->] (0.15,-0.2) -- (0.15,0.2);
                \multdot{0.15,0}{west}{r};
            \end{tikzpicture}
        \right]
        ,
    \end{equation}
    since then the result follows after applying $\varphi_0$ and using \cref{sky,sloth}.  We prove the equality in \cref{cirrus} for $r>0$, since the proof for $r<0$ is analogous.  For $r>0$, we have
    \[
        \begin{tikzpicture}[centerzero]
            \draw[->] (-0.2,-0.5) -- (-0.2,0.5);
            \draw[<-] (0.2,-0.5) -- (0.2,0.5);
            \multdot{-0.2,0}{east}{r};
        \end{tikzpicture}
        \overset{\cref{silence}}{=}
        \begin{tikzpicture}[anchorbase]
            \draw[->] (-0.2,-0.6) -- (-0.2,-0.4) \braidup (0.2,0) \braidup (-0.2,0.4);
            \draw[wipe] (0.2,-0.6) -- (0.2,-0.4) \braidup (-0.2,0) \braidup (0.2,0.4);
            \draw[<-] (0.2,-0.6) -- (0.2,-0.4) \braidup (-0.2,0) \braidup (0.2,0.4);
            \multdot{-0.2,-0.4}{east}{r};
        \end{tikzpicture}
        \overset{\cref{teapos}}{\equiv}
        \begin{tikzpicture}[centerzero]
            \draw[->] (0.2,0) \braidup (-0.2,0.5);
            \draw[<-] (0.2,-0.5) \braidup (-0.2,0);
            \draw[wipe] (-0.2,-0.5) \braidup (0.2,0);
            \draw (-0.2,-0.5) \braidup (0.2,0);
            \draw[wipe] (-0.2,0) \braidup (0.2,0.5);
            \draw (-0.2,0) \braidup (0.2,0.5);
            \multdot{0.2,0}{west}{r};
        \end{tikzpicture}
        - z \sum_{\substack{a+b=r \\ a,b > 0}}
        \begin{tikzpicture}[centerzero]
            \draw[->] (-0.2,-0.3) arc(180:360:0.2) \braidup (-0.2,0.3);;
            \draw[wipe] (-0.2,0.3) arc(180:0:0.2) (0.2,0.3) \braiddown (-0.2,-0.3);
            \draw (-0.2,0.3) arc(180:0:0.2) (0.2,0.3) \braiddown (-0.2,-0.3);
            \multdot{0.2,0.3}{west}{a};
            \multdot{0.2,-0.3}{west}{b};
        \end{tikzpicture}
        \overset{\cref{skein}}{\underset{\cref{silence}}{\equiv}}
        \begin{tikzpicture}[centerzero]
            \draw[<-] (-0.2,-0.5) -- (-0.2,0.5);
            \draw[->] (0.2,-0.5) -- (0.2,0.5);
            \multdot{0.2,0}{west}{r};
        \end{tikzpicture}
        - z \sum_{\substack{a+b=r \\ a \ge 0,\, b > 0}}
        \begin{tikzpicture}[centerzero]
            \draw[->] (-0.2,-0.3) arc(180:360:0.2) \braidup (-0.2,0.3);;
            \draw[wipe] (-0.2,0.3) arc(180:0:0.2) (0.2,0.3) \braiddown (-0.2,-0.3);
            \draw (-0.2,0.3) arc(180:0:0.2) (0.2,0.3) \braiddown (-0.2,-0.3);
            \multdot{0.2,0.3}{west}{a};
            \multdot{0.2,-0.3}{west}{b};
        \end{tikzpicture}
        \ .
    \]
    Hence
    \begin{align*}
        \begin{tikzpicture}[centerzero]
            \draw[->] (-0.2,-0.5) -- (-0.2,0.5);
            \draw[<-] (0.2,-0.5) -- (0.2,0.5);
            \multdot{-0.2,0}{east}{r};
        \end{tikzpicture}
        -
        \begin{tikzpicture}[centerzero]
            \draw[<-] (-0.2,-0.5) -- (-0.2,0.5);
            \draw[->] (0.2,-0.5) -- (0.2,0.5);
            \multdot{0.2,0}{west}{r};
        \end{tikzpicture}
        &\overset{\mathclap{\cref{rightcurl}}}{\equiv}\
        - z^2 \sum_{\substack{a+b=r \\ a \ge 0,\ b>0}}
        \left(
            \sum_{c>0} \leftbub{b-c} \rightminus{a+c}
            - \sum_{c \ge 0} \leftbub{b+c} \rightplus{a-c}
        \right)
        \overset{\cref{clear}}{\underset{\substack{\cref{weeds+} \\ \cref{weeds-}}}{=}}
        z^2 \sum_{\substack{a+b=r \\ a \ge 0,\, b>0}} \sum_{c=0}^a
        \leftplus{b+c} \rightplus{a-c}
        \\
        &= z^2 \sum_{s \in \Z} s\ \leftplus{s} \rightplus{r-s}
        \overset{\cref{pcannon}}{=} \beta(p_r^+) = -\{r\}\ \symbub{r},
    \end{align*}
    as desired.
    \details{
        We give here the details of the omitted case of \cref{cirrus}.  For $r>0$, we have
        \[
            \begin{tikzpicture}[centerzero]
                \draw[->] (-0.2,-0.5) -- (-0.2,0.5);
                \draw[<-] (0.2,-0.5) -- (0.2,0.5);
                \multdot{-0.2,0}{east}{-r};
            \end{tikzpicture}
            \overset{\cref{silence}}{=}
            \begin{tikzpicture}[anchorbase]
                \draw[->] (-0.2,-0.6) -- (-0.2,-0.4) \braidup (0.2,0) \braidup (-0.2,0.4);
                \draw[wipe] (0.2,-0.6) -- (0.2,-0.4) \braidup (-0.2,0) \braidup (0.2,0.4);
                \draw[<-] (0.2,-0.6) -- (0.2,-0.4) \braidup (-0.2,0) \braidup (0.2,0.4);
                \multdot{-0.2,-0.4}{east}{-r};
            \end{tikzpicture}
            \overset{\cref{teapos}}{\equiv}
            \begin{tikzpicture}[centerzero]
                \draw[->] (0.2,0) \braidup (-0.2,0.5);
                \draw[<-] (0.2,-0.5) \braidup (-0.2,0);
                \draw[wipe] (-0.2,-0.5) \braidup (0.2,0);
                \draw (-0.2,-0.5) \braidup (0.2,0);
                \draw[wipe] (-0.2,0) \braidup (0.2,0.5);
                \draw (-0.2,0) \braidup (0.2,0.5);
                \multdot{0.2,0}{west}{-r};
            \end{tikzpicture}
            + z \sum_{\substack{a+b=r \\ a,b \ge 0}}
            \begin{tikzpicture}[centerzero]
                \draw[->] (-0.2,-0.3) arc(180:360:0.2) \braidup (-0.2,0.3);;
                \draw[wipe] (-0.2,0.3) arc(180:0:0.2) (0.2,0.3) \braiddown (-0.2,-0.3);
                \draw (-0.2,0.3) arc(180:0:0.2) (0.2,0.3) \braiddown (-0.2,-0.3);
                \multdot{0.2,0.3}{west}{-a};
                \multdot{0.2,-0.3}{west}{-b};
            \end{tikzpicture}
            \overset{\cref{skein}}{\underset{\cref{silence}}{\equiv}}
            \begin{tikzpicture}[centerzero]
                \draw[<-] (-0.2,-0.5) -- (-0.2,0.5);
                \draw[->] (0.2,-0.5) -- (0.2,0.5);
                \multdot{0.2,0}{west}{r};
            \end{tikzpicture}
            + z \sum_{\substack{a+b=r \\ a > 0,\, b \ge 0}}
            \begin{tikzpicture}[centerzero]
                \draw[->] (-0.2,-0.3) arc(180:360:0.2) \braidup (-0.2,0.3);;
                \draw[wipe] (-0.2,0.3) arc(180:0:0.2) (0.2,0.3) \braiddown (-0.2,-0.3);
                \draw (-0.2,0.3) arc(180:0:0.2) (0.2,0.3) \braiddown (-0.2,-0.3);
                \multdot{0.2,0.3}{west}{-a};
                \multdot{0.2,-0.3}{west}{-b};
            \end{tikzpicture}
            \ .
        \]
        Hence
        \begin{align*}
            \begin{tikzpicture}[centerzero]
                \draw[->] (-0.2,-0.5) -- (-0.2,0.5);
                \draw[<-] (0.2,-0.5) -- (0.2,0.5);
                \multdot{-0.2,0}{east}{-r};
            \end{tikzpicture}
            -
            \begin{tikzpicture}[centerzero]
                \draw[<-] (-0.2,-0.5) -- (-0.2,0.5);
                \draw[->] (0.2,-0.5) -- (0.2,0.5);
                \multdot{0.2,0}{west}{-r};
            \end{tikzpicture}
            &\overset{\mathclap{\cref{rightcurl}}}{\equiv}\
            z^2 \sum_{\substack{a+b=r \\ a > 0,\ b \ge 0}}
            \left(
                \sum_{c>0} \leftbub{-b-c} \rightminus{-a+c}
                - \sum_{c \ge 0} \leftbub{-b+c} \rightplus{-a-c}
            \right)
            \overset{\cref{clear}}{\underset{\substack{\cref{weeds+} \\ \cref{weeds-}}}{=}}
            z^2 \sum_{\substack{a+b=r \\ a > 0,\, b \ge 0}} \sum_{c=1}^a
            \leftminus{-b-c} \rightminus{-a+c}
            \\
            &= z^2 \sum_{s \in \Z} s\ \leftminus{-s} \rightminus{s-r}
            \overset{\cref{pcannon}}{=} \beta(p_r^-) = \{r\} \symbub{-r},
        \end{align*}
        as desired.
    }

    It remains to prove \cref{sun}.  By \cite[Def.~2.5]{MS17}, we have
    \[
        \varphi_0(w_{0,n}) = [\sigma_n], \quad n \ne 0.
    \]
    Thus, for $r,n \in \Z \setminus \{0\}$, we have
    \[
        \varphi_0(w_{r,n})
        = \{rn\}^{-1} \varphi_0( [w_{r,0}, w_{n,0}] )
        \overset{\cref{clouds}}{=} \{rn\}^{-1} \left[ \symbub{r} \otimes \sigma_n - \sigma_n \otimes \symbub{r} \right]
        \overset{\cref{goat}}{=} [ \chi_{r,n} \sigma_n ].
        \qedhere
    \]
\end{proof}

\begin{rem} \label{MSrem}
    As we see from the proof of \cref{moon}, we use the results of \cite{MS17} to prove the $k=0$ case of \cref{mainthm}.  The proof of \cite[Th.~1]{MS17} involves induction on $\det \begin{pmatrix} \bx & \by \end{pmatrix}$, starting with the base cases
    \begin{equation} \label{MSbase}
        \left[ \left[ \symbub{r} \right], \left[ \upstrand \right] \right] = \{r\} \left[ \multupdot{r} \right],\quad
        \left[ \left[ \downstrand \right], \multupdot{r} \right] = \{r\} \left[ \symbub{r} \right],\qquad
        r \in \Z \setminus \{0\}.
    \end{equation}
    In \cite{MS17}, the proof of the first equation in \cref{MSbase} relies on \cite[Th.~4.2]{Mor02}, while the proof of the second involves direct skein manipulation.  Note that the first equation in \cref{MSbase} is precisely the bubble slide relation \cref{goat}, while the second equation in \cref{MSbase} is \cref{cirrus}.  Thus, in order to make the arguments of the current paper independent of the results of \cite{MS17}, one would only need to include the inductive argument of \cite[\S 3.2]{MS17}; this is a purely algebraic argument, involving no skein theory.
\end{rem}

\begin{rem} \label{ground}
    \Cref{MSrem} allows us to see that \cite[Th.~1 \& 2]{MS17} hold over the more arbitrary ground ring $\kk$ considered in this section.  The assumption that $\kk$ contains $\Q$ is needed in the proof of \cite[Lem.~3.1]{MS17}, since this proof uses \cite[Th.~1]{Tur88}, which involves the $\sigma_{n,0}$ (denoted $A_{n,0}$ in \cite{MS17}), whereas \cite[Lem.~3.1]{MS17} involves the $\sigma_n$ (denoted $P_n$ in \cite{MS17}); solving for the $\sigma_{n,0}$ in terms of the $\sigma_n$ requires division by $n$ (see \cite[Rem.~2.4]{MS17}).  Essentially, at issue is the fact that the power sums generate $\Sym$ over $\Q$, but not over $\Z$; see \cref{fort}.  The inductive argument of \cite[\S3.2]{MS17} only requires division by $\{d\}$, $d \ge 1$.
\end{rem}

\begin{rem} \label{decorated}
    In the case $k=0$, our explicit description \cref{clouds,sun} of the image of the isomorphism $\varphi_k$ differs from that given in \cite[Def.~2.5]{MS17}, which involves decorated framed oriented curves.  There is no contradiction here, as the presence of the skein relations \cref{metalox} means that different linear combinations of classes of framed oriented tangles can be equal in the framed HOMFLYPT skein algebra of the torus.
\end{rem}

\begin{prop} \label{rocinante}
    If \cref{mainthm} holds for central charge $k$, then it holds for central charge $-k$.
\end{prop}

\begin{proof}
    Suppose \cref{mainthm} holds for some $k \in \Z$.  Recall the isomorphisms $\Omega_k$ and $\omega_k$ from \cref{earth} and \cref{omegak}, respectively.  Consider the following diagram:
    \[
        \begin{tikzcd}
            \EH_k \arrow[r, "\varphi_k"] & \Tr(\Heis_k(z,t^{-1})) \arrow[d, "\Tr(\Omega_k)"] \\
            \EH_{-k} \arrow[u, "\omega_{-k}"] \arrow[r, "\varphi_{-k}"] & \Tr(\Heis_{-k}(z,t))
        \end{tikzcd}
    \]
    By assumption, $\varphi_k$ is an isomorphism of algebras.  Since $\Tr(\Omega_k)$ and $\omega_{-k}$ are also isomorphisms, there is an algebra isomorphism $\varphi_{-k} \colon \EH_{-k} \xrightarrow{\cong} \Tr(\Heis_{-k}(z,t))$ making the above diagram commute.  For $r \in \Z$, we have
    \begin{gather*}
        \varphi_{-k} \left( w_{r,1} \right)
        = \Tr(\Omega_k) \circ \varphi_k \circ \omega_{-k} ( w_{r,1} )
        = \left[ \multupdot{r} \right]
        \quad \text{and}
        \\
        \varphi_{-k} \left(  w_{r,-1} \right)
        = \Tr(\Omega_k) \circ \varphi_k \circ \omega_{-k} ( w_{r,-1} )
        = \left[ \multdowndot{r} \right].
    \end{gather*}
    Thus the first statement in \cref{mainthm} holds for central charge $-k$.  To prove that \cref{clouds} also holds for central charge $-k$, we compute, for $r \ne 0$,
    \[
        \varphi_{-k} ( w_{r,0} )
        = \Tr(\Omega_k) \circ \varphi_k \circ \omega_{-k} ( w_{r,0} )
        = \left[ \symbub{r} \right].
        \qedhere
    \]
\end{proof}

In light of \cref{moon,rocinante}, it suffices to prove \cref{mainthm} for central charge $k < 0$.  Thus,
\begin{center}
    \emph{for the remainder of this section we assume $k < 0$}.
\end{center}

The proof of following proposition is inspired by that of \cite[Prop.~6.2]{RS20}.

\begin{prop} \label{squirrel}
    The tensor product
    \[
        \Heis_k^+ \times \Heis_k^- \xrightarrow{\otimes} \Heis_k
    \]
    induces a linear isomorphism
    \[
        \Tr(\Heis_k^+) \otimes \Tr(\Heis_k^-)
        \xrightarrow{\cong} \Tr(\Heis_k).
    \]
\end{prop}

\begin{proof}
    Since $k<0$, it follows from \cref{invrel} that every object of $\Add(\Heis_k)$ is isomorphic to a direct sum of objects of the form $\uparrow^{\otimes m} \otimes \downarrow^{\otimes n}$, $m,n \in \N$.  Let $\cC$ be the full subcategory of $\Heis_k$ whose objects are $\uparrow^{\otimes m} \otimes \downarrow^{\otimes n}$, $m,n \in \N$.  Then, by \cref{WSB}, the inclusion functor $\cC \to \Heis_k$ induces a linear isomorphism $\Tr(\cC) \cong \Tr(\Heis_k)$.

    For $m \in \Z$, let $\cC^{(m)}$ be the full subcategory of $\Heis_k$ whose objects are $\uparrow^{\oplus (m+n)} \otimes \downarrow^{\otimes n}$ for $n \ge \max(0,-m)$.  Then there are no morphisms between objects of $\cC^{(m)}$ and $\cC^{(n)}$ for $m \ne n$.  Thus $\cC = \bigsqcup_{m \in \Z} \cC^{(m)}$, and so $\Tr(\cC) = \bigoplus_{m \in \Z} \Tr(\cC^{(m)})$.

    Fix $m \in \Z$ for the remainder of the proof.  For $n \in \Z$ with $n \ge \max(0,-m)$, let
    \[
        X_n := \uparrow^{\otimes (m+n)} \otimes \downarrow^{\otimes n}.
    \]
    Recall the definition of a \emph{reduced lift} given above \cref{basisthm}, and let $\bB_{\Sym}$ be a basis of $\Sym$ (e.g.\ we can take $\bB_{\Sym}$ to be the set of Schur functions).  Fix a set $D(n_2,n_1)$ consisting of a choice of reduced lift each $(X_{n_1},X_{n_2})$-matching.  Then, for $n,n_1,n_2 \ge \max(0,-m)$, define the following:
    \begin{itemize}
        \item If $n_1 > n_2$, let $D_{n_2,n_1}$ denote the set of all morphisms obtained from elements of $D(n_2,n_1)$ \emph{containing no cups} by adding to each string \emph{involved in a cap} an integer number of dots near the terminus of the string.
        \item If $n_1 < n_2$, let $D_{n_2,n_1}$ denote the set of all morphisms obtained from elements of $D(n_2,n_1)$ \emph{containing no caps} by adding to each string \emph{involved in a cup} an integer number of dots near the terminus of the string.
        \item Let $D_{n,n} = \{1_{X_n}\}$.
        \item Let $D_n$ denote the set of all morphisms that can be obtained from the elements of $D(n,n)$ \emph{containing no cups or caps} by adding to each string an integer number of dots near the terminus of the string, and then placing an element of $\beta(\bB_{\Sym} \otimes 1)$ in between the downward and upward strings (i.e.\ to the right of all downward strings and to the left of all upward strings), and then placing an element of $\beta(1 \otimes \bB_{\Sym})$ to the right of all strings.
    \end{itemize}
    We claim that the sets $D_{n_2,n_1}$, $D_n$ satisfy the conditions \eqref{B1} and \eqref{B2} of \cref{sec:trace}, where
    \[
        R_n = \End_{\Heis_k^+}(\uparrow^{\otimes (m+n)}) \otimes \End_{\Heis_k^-}(\downarrow^{\otimes n}).
    \]
    Given the claim, the current proposition then follows from \cref{saber}.

    It remains to prove the claim.  Condition \eqref{B2} is clear.  To see that \eqref{B1} is satisfied, we need to verify that, for each $m,n \in \N$, the set $\bB_{n_2,n_1} = \bigsqcup_{l=0}^{\min(n_1,n_2)} D_{n_2,l} D_l D_{l,n_1}$ is a basis of $\Hom_\cC(X_{n_1},X_{n_2})$.  The difference between the elements of $\bB_{n_2,n_1}$ and the elements described in \cref{basisthm} is that
    \begin{itemize}
        \item for strings connecting the top and bottom of the diagram, the basis elements in \cref{basisthm} have dots near the termini of the strands, whereas the dots on such strands in the elements of $\bB_{n_2,n_1}$ are in the middle of the diagram;
        \item the basis elements in \cref{basisthm} have all bubbles on the right side of the diagram, whereas the $(+)$-bubbles appearing in elements of $\bB_{n_2,n_1}$ are in the middle of the diagram.
    \end{itemize}
    Using \cref{dotslide,skein}, dots slide through crossings modulo diagrams with fewer total crossings.  Similarly, by \cref{bs}, bubbles slide through strands modulo diagrams with fewer total dots.  For example, in \cref{anthill}, the left-hand diagram is a typical element of $D_{4,2} D_2 D_{2,3}$, while the right-hand diagram is a typical element of the basis from \cref{basisthm}.
    \begin{figure}
      \centering
      \begin{tikzpicture}[centerzero]
        \draw[<-] (-1.5,1.2) -- (-1.5,1) to[out=down,in=down] (1,1) -- (1,1.2);
        \draw[wipe] (1,-1.2) \braidup (0.5,1.2);
        \draw[<-] (1,-1.2) \braidup (0.5,1.2);
        \draw[wipe] (0,-1.2) \braidup (1.5,1.2);
        \draw[<-] (0,-1.2) \braidup (1.5,1.2);
        \draw[wipe] (-1,-1.2) -- (-1,-1) to[out=up,in=up] (0.5,-1) -- (0.5,-1.2);
        \draw[->] (-1,-1.2) -- (-1,-1) to[out=up,in=up] (0.5,-1) -- (0.5,-1.2);
        \draw[wipe] (-0.5,-1.2) \braidup (-1,1.2);
        \draw[->] (-0.5,-1.2) \braidup (-1,1.2);
        \draw[<-] (-0.5,1.2) -- (-0.5,1) arc(180:360:0.25) -- (0,1.2);
        \plusrightside{0.1,-0.1}{2};
        \minusrightside{1.6,-0.1}{3};
        \multdot{-0.75,0}{east}{-5};
        \singdot{0.67,0.23};
        \singdot{-0.5,1};
        \multdot{-1.5,1}{east}{-4};
        \multdot{0.5,-1}{west}{2};
      \end{tikzpicture}
      \hspace{2cm}
      \begin{tikzpicture}[centerzero]
        \draw[<-] (-1.5,1.2) -- (-1.5,1) to[out=down,in=down] (1,1) -- (1,1.2);
        \draw[wipe] (1,-1.2) \braidup (0.5,1.2);
        \draw[<-] (1,-1.2) \braidup (0.5,1.2);
        \draw[wipe] (0,-1.2) \braidup (1.5,1.2);
        \draw[<-] (0,-1.2) \braidup (1.5,1.2);
        \draw[wipe] (-1,-1.2) -- (-1,-1) to[out=up,in=up] (0.5,-1) -- (0.5,-1.2);
        \draw[->] (-1,-1.2) -- (-1,-1) to[out=up,in=up] (0.5,-1) -- (0.5,-1.2);
        \draw[wipe] (-0.5,-1.2) \braidup (-1,1.2);
        \draw[->] (-0.5,-1.2) \braidup (-1,1.2);
        \draw[<-] (-0.5,1.2) -- (-0.5,1) arc(180:360:0.25) -- (0,1.2);
        \plusrightside{1.8,-0.6}{2};
        \minusrightside{1.8,0.2}{3};
        \multdot{-0.94,0.63}{west}{-5};
        \singdot{0.97,-0.8};
        \singdot{-0.5,1};
        \multdot{-1.5,1}{east}{-4};
        \multdot{0.5,-1}{west}{2};
      \end{tikzpicture}
      \caption{Typical basis elements for $m=-1$, $n_1=3$, $n_2=4$.}\label{anthill}
    \end{figure}
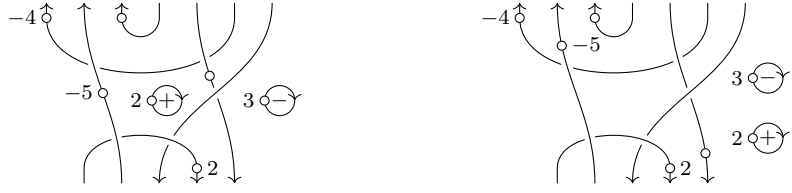
    The left-hand diagram is equal to the right-hand diagram modulo diagrams with fewer total crossings.  Thus the claim follows by a standard triangularity argument.
\end{proof}

\begin{prop} \label{grapefruit}
    We have a linear isomorphism
    \[
        \varphi_k \colon \EH_k \xrightarrow{\cong} \Tr(\Heis_k)
    \]
    satisfying \cref{sky,clouds}, and the composites
    \[
        \EH_k^\pm \hookrightarrow \EH_k
        \xrightarrow{\varphi_k} \Tr(\Heis_k)
        \twoheadrightarrow \Tr(\Heis_k^\pm)
    \]
    are isomorphisms of algebras.
\end{prop}

\begin{proof}
    The map $\varphi_k$ is the composite
    \begin{align*}
        \EH_k &\to \EH^+ \otimes \EH^- & \text{(see \cref{lantern})} \\
        &\to \EH_0 & \text{(see \cref{lantern})} \\
        &\to \Tr(\Heis_0) & \text{(\cref{moon})} \\
        &\to \Tr(\Heis_0^+) \otimes \Tr(\Heis_0^-) & \text{(\cref{squirrel})} \\
        &\to \Tr(\Heis_k^+) \otimes \Tr(\Heis_k^-) & \text{(see \cref{eagle})}\\
        &\to \Tr(\Heis_k).
    \end{align*}
    The proposition then follows immediately from the properties of these individual maps.
\end{proof}

We are now ready to prove \cref{mainthm}.

\begin{proof}[Proof of \cref{mainthm}]
    By \cref{moon,rocinante}, it suffices to give the proof for the case $k<0$, which we assume for the remainder of the proof.  By \cref{grapefruit}, we have a homomorphism of algebras
    \[
        \EH_k^+ \star \EH_k^- \to \Tr(\Heis_k)
    \]
    (recall that $\star$ denotes the free product of algebras) satisfying
    \[
        w_{r,1} \mapsto \left[ \multupdot{r} \right],\quad
        w_{r,-1} \mapsto \left[ \multdowndot{r} \right],\quad
        w_{s,0} \mapsto \left[ \symbub{s} \right],\quad
        r,s \in \Z,\ s \ne 0.
    \]
    By \cref{crossbow}, if we verify that the images of the relations \cref{bolt1,bolt2,bolt3} hold in $\Tr(\Heis_k)$, it follows that we have an induced algebra homomorphism $\EH_k \to \Tr(\Heis_k)$, which is equal to the linear isomorphism $\varphi_k$ of \cref{grapefruit}.

    It is clear the image of \cref{bolt3} holds in $\Tr(\Heis_k)$ since $\Tr(\beta(\Sym \otimes \Sym))$ is a commutative subalgebra of $\Tr(\Heis_k)$.  To verify that the image of \cref{bolt2} holds, we compose all morphisms in \cref{goat} with $\multupdot{s}$ to see that
    \[
        \symbub{r}\
        \begin{tikzpicture}[centerzero]
            \draw[->] (0,-0.5) -- (0,0.5);
            \multdot{0,0}{west}{s};
        \end{tikzpicture}
        =
        \begin{tikzpicture}[centerzero]
            \draw[->] (0,-0.5) to (0,0.5);
            \multdot{0,0}{east}{s};
        \end{tikzpicture}
        \ \symbub{r}
        + \{r\}\
        \begin{tikzpicture}[centerzero]
            \draw[->] (0,-0.5) to (0,0.5);
            \multdot{0,0}{west}{s + r};
        \end{tikzpicture}
        .
    \]
    Passing to $\Tr(\Heis_k)$ gives the desired relations.

    It remains to verify that the image of \cref{bolt1} is satisfied.  More precisely, we must show that
    \begin{equation} \label{tree}
        \left[ \updot\ \multdowndot{s} \right]
        =
        \left[ \multdowndot[east]{s}\ \updot \right]
        - \{s+1\}
        \left[
            \begin{tikzpicture}[centerzero]
                \draw (0,0) circle(0.3);
                \node at (0,0) {\dotlabel{s+1}};
            \end{tikzpicture}
        \right]
        + \delta_{s,-1} k.
    \end{equation}
    (Recall our convention that $\symbub{0}=0$; see \cref{fort}.)

    We first prove \cref{tree} for $s \ge 0$.  In this case we have
    \begin{align*}
        \begin{tikzpicture}[centerzero]
            \draw[->] (-0.2,-0.5) -- (-0.2,0.5);
            \draw[<-] (0.2,-0.5) -- (0.2,0.5);
            \singdot{-0.2,0};
            \multdot{0.2,0}{west}{s};
        \end{tikzpicture}
        \ &\overset{\mathclap{\cref{silence}}}{=}\ \
        \begin{tikzpicture}[anchorbase]
            \draw (-0.2,-0.6) -- (-0.2,-0.4) \braidup (0.2,0);
            \draw[wipe] (0.2,-0.6) -- (0.2,-0.4) \braidup (-0.2,0) \braidup (0.2,0.4);
            \draw[<-] (0.2,-0.6) -- (0.2,-0.4) \braidup (-0.2,0) \braidup (0.2,0.4);
            \draw[wipe] (0.2,0) \braidup (-0.2,0.4);
            \draw[->] (0.2,0) \braidup (-0.2,0.4);
            \singdot{-0.2,-0.4};
            \multdot{0.2,-0.4}{west}{s};
        \end{tikzpicture}
        \overset{\cref{dotslide}}{=}
        \begin{tikzpicture}[anchorbase]
            \draw[<-] (0.2,-0.6) -- (0.2,-0.4) \braidup (-0.2,0) \braidup (0.2,0.4);
            \draw[wipe] (-0.2,-0.6) -- (-0.2,-0.4) \braidup (0.2,0);
            \draw (-0.2,-0.6) -- (-0.2,-0.4) \braidup (0.2,0);
            \draw[wipe] (0.2,0) \braidup (-0.2,0.4);
            \draw[->] (0.2,0) \braidup (-0.2,0.4);
            \singdot{0.2,0};
            \multdot{0.2,-0.4}{west}{s};
        \end{tikzpicture}
        \overset{\cref{teapos}}{\equiv}
        \begin{tikzpicture}[centerzero]
            \draw (-0.2,0) \braidup (0.2,0.5);
            \draw[wipe] (-0.2,-0.5) \braidup (0.2,0) \braidup (-0.2,0.5);
            \draw[->] (-0.2,-0.5) \braidup (0.2,0) \braidup (-0.2,0.5);
            \draw[wipe] (0.2,-0.5) \braidup (-0.2,0);
            \draw[<-] (0.2,-0.5) \braidup (-0.2,0);
            \singdot{0.2,0};
            \multdot{-0.2,0}{east}{s};
        \end{tikzpicture}
        - z \sum_{\substack{a+b=s \\ a,b \ge 0}}
        \begin{tikzpicture}[centerzero]
            \draw (-0.2,0.3) arc(180:0:0.2) (0.2,0.3) \braiddown (-0.2,-0.3);
            \draw[wipe] (-0.2,-0.3) arc(180:360:0.2) \braidup (-0.2,0.3);;
            \draw[->] (-0.2,-0.3) arc(180:360:0.2) \braidup (-0.2,0.3);;
            \multdot{0.2,0.3}{west}{a};
            \multdot{0.2,-0.3}{west}{b+1};
        \end{tikzpicture}
        \equiv
        \begin{tikzpicture}[anchorbase]
            \draw[<-] (-0.2,-0.6) -- (-0.2,-0.4) \braidup (0.2,0);
            \draw[wipe] (0.2,-0.6) -- (0.2,-0.4) \braidup (-0.2,0) \braidup (0.2,0.4);
            \draw[->] (0.2,-0.6) -- (0.2,-0.4) \braidup (-0.2,0) \braidup (0.2,0.4);
            \draw[wipe] (0.2,0) \braidup (-0.2,0.4);
            \draw (0.2,0) \braidup (-0.2,0.4);
            \multdot{-0.2,-0.4}{east}{s};
            \singdot{0.2,-0.4};
        \end{tikzpicture}
        - z \sum_{\substack{a+b=s \\ a,b \ge 0}}
        \begin{tikzpicture}[centerzero]
            \draw (-0.2,0.3) arc(180:0:0.2) (0.2,0.3) \braiddown (-0.2,-0.3);
            \draw[wipe] (-0.2,-0.3) arc(180:360:0.2) \braidup (-0.2,0.3);;
            \draw[->] (-0.2,-0.3) arc(180:360:0.2) \braidup (-0.2,0.3);;
            \multdot{0.2,0.3}{west}{a};
            \multdot{0.2,-0.3}{west}{b+1};
        \end{tikzpicture}
        \\
        &\overset{\mathclap{\cref{neg}}}{\underset{\mathclap{\cref{rightcurl}}}{=}}\
        \begin{tikzpicture}[centerzero]
            \draw[<-] (-0.2,-0.5) -- (-0.2,0.5);
            \draw[->] (0.2,-0.5) -- (0.2,0.5);
            \singdot{0.2,0};
            \multdot{-0.2,0}{east}{s};
        \end{tikzpicture}
        + tz \leftbub{s+1}
        + z^2 \sum_{a>0} a \leftbub{s+a+1} \rightplus{-a}
        - z^2 \sum_{\substack{a+b=s \\ a,b \ge 0}}
        \left(
            \sum_{c \ge 0} \leftbub{b-c+1} \rightminus{a+c}
            - \sum_{c>0} \leftbub{b+c+1} \rightplus{a-c}
        \right)
        \\
        &\overset{\mathclap{\substack{\cref{poppy} \\ \cref{weeds+}}}}{\underset{\mathclap{\cref{weeds-}}}{=}}\
        \begin{tikzpicture}[centerzero]
            \draw[<-] (-0.2,-0.5) -- (-0.2,0.5);
            \draw[->] (0.2,-0.5) -- (0.2,0.5);
            \singdot{0.2,0};
            \multdot{-0.2,0}{east}{s};
        \end{tikzpicture}
        + z^2 \sum_{a>0} a \leftplus{s+a+1} \rightplus{-a}
        + z^2 \sum_{\substack{a+b=s \\ a,b \ge 0}} \sum_{c>0} \leftplus{b+c+1} \rightplus{a-c}
        \\
        &=
        \begin{tikzpicture}[centerzero]
            \draw[<-] (-0.2,-0.5) -- (-0.2,0.5);
            \draw[->] (0.2,-0.5) -- (0.2,0.5);
            \singdot{0.2,0};
            \multdot{-0.2,0}{east}{s};
        \end{tikzpicture}
        + z^2 \sum_{r>s} (r-s) \leftplus{r+1} \rightplus{s-r}
        + z^2 \sum_{r=1}^s r \leftplus{r+1} \rightplus{s-r}
        + z^2 \sum_{r>s} s \leftplus{r+1} \rightplus{s-r}
        \\
        &=
        \begin{tikzpicture}[centerzero]
            \draw[<-] (-0.2,-0.5) -- (-0.2,0.5);
            \draw[->] (0.2,-0.5) -- (0.2,0.5);
            \singdot{0.2,0};
            \multdot{-0.2,0}{east}{s};
        \end{tikzpicture}
        + z^2 \sum_{r \ge 1} r\ \leftplus{r+1} \rightplus{s-r}
        \ \overset{\mathclap{\cref{infgrass}}}{\underset{\mathclap{\cref{weeds+}}}{=}}\
        \begin{tikzpicture}[centerzero]
            \draw[<-] (-0.2,-0.5) -- (-0.2,0.5);
            \draw[->] (0.2,-0.5) -- (0.2,0.5);
            \singdot{0.2,0};
            \multdot{-0.2,0}{east}{s};
        \end{tikzpicture}
        + z^2 \sum_{r \in \Z} r\ \leftplus{r} \rightplus{s+1-r}
        \overset{\cref{pcannon}}{=}
        \begin{tikzpicture}[centerzero]
            \draw[<-] (-0.2,-0.5) -- (-0.2,0.5);
            \draw[->] (0.2,-0.5) -- (0.2,0.5);
            \singdot{0.2,0};
            \multdot{-0.2,0}{east}{s};
        \end{tikzpicture}
        - \{s+1\}
        \begin{tikzpicture}[centerzero]
            \draw (0,0) circle(0.3);
            \node at (0,0) {\dotlabel{s+1}};
        \end{tikzpicture}
        \ .
    \end{align*}
    Thus \cref{tree} holds.

    Finally, we prove \cref{tree} for $s<0$.  In this case we have
    \begin{align*}
        \begin{tikzpicture}[centerzero]
            \draw[->] (-0.2,-0.5) -- (-0.2,0.5);
            \draw[<-] (0.2,-0.5) -- (0.2,0.5);
            \singdot{-0.2,0};
            \multdot{0.2,0}{west}{s};
        \end{tikzpicture}
        \ &\overset{\mathclap{\cref{silence}}}{=}\ \
        \begin{tikzpicture}[anchorbase]
            \draw (-0.2,-0.6) -- (-0.2,-0.4) \braidup (0.2,0);
            \draw[wipe] (0.2,-0.6) -- (0.2,-0.4) \braidup (-0.2,0) \braidup (0.2,0.4);
            \draw[<-] (0.2,-0.6) -- (0.2,-0.4) \braidup (-0.2,0) \braidup (0.2,0.4);
            \draw[wipe] (0.2,0) \braidup (-0.2,0.4);
            \draw[->] (0.2,0) \braidup (-0.2,0.4);
            \singdot{-0.2,-0.4};
            \multdot{0.2,-0.4}{west}{s};
        \end{tikzpicture}
        \overset{\cref{dotslide}}{=}
        \begin{tikzpicture}[anchorbase]
            \draw[<-] (0.2,-0.6) -- (0.2,-0.4) \braidup (-0.2,0) \braidup (0.2,0.4);
            \draw[wipe] (-0.2,-0.6) -- (-0.2,-0.4) \braidup (0.2,0);
            \draw (-0.2,-0.6) -- (-0.2,-0.4) \braidup (0.2,0);
            \draw[wipe] (0.2,0) \braidup (-0.2,0.4);
            \draw[->] (0.2,0) \braidup (-0.2,0.4);
            \singdot{0.2,0};
            \multdot{0.2,-0.4}{west}{s};
        \end{tikzpicture}
        \overset{\cref{teapos}}{\equiv}
        \begin{tikzpicture}[centerzero]
            \draw (-0.2,0) \braidup (0.2,0.5);
            \draw[wipe] (-0.2,-0.5) \braidup (0.2,0) \braidup (-0.2,0.5);
            \draw[->] (-0.2,-0.5) \braidup (0.2,0) \braidup (-0.2,0.5);
            \draw[wipe] (0.2,-0.5) \braidup (-0.2,0);
            \draw[<-] (0.2,-0.5) \braidup (-0.2,0);
            \singdot{0.2,0};
            \multdot{-0.2,0}{east}{s};
        \end{tikzpicture}
        + z \sum_{\substack{a+b=s \\ a,b < 0}}
        \begin{tikzpicture}[centerzero]
            \draw (-0.2,0.3) arc(180:0:0.2) (0.2,0.3) \braiddown (-0.2,-0.3);
            \draw[wipe] (-0.2,-0.3) arc(180:360:0.2) \braidup (-0.2,0.3);;
            \draw[->] (-0.2,-0.3) arc(180:360:0.2) \braidup (-0.2,0.3);;
            \multdot{0.2,0.3}{west}{a};
            \multdot{0.2,-0.3}{west}{b+1};
        \end{tikzpicture}
        \equiv
        \begin{tikzpicture}[anchorbase]
            \draw[<-] (-0.2,-0.6) -- (-0.2,-0.4) \braidup (0.2,0);
            \draw[wipe] (0.2,-0.6) -- (0.2,-0.4) \braidup (-0.2,0) \braidup (0.2,0.4);
            \draw[->] (0.2,-0.6) -- (0.2,-0.4) \braidup (-0.2,0) \braidup (0.2,0.4);
            \draw[wipe] (0.2,0) \braidup (-0.2,0.4);
            \draw (0.2,0) \braidup (-0.2,0.4);
            \multdot{-0.2,-0.4}{east}{s};
            \singdot{0.2,-0.4};
        \end{tikzpicture}
        + z \sum_{\substack{a+b=s \\ a,b < 0}}
        \begin{tikzpicture}[centerzero]
            \draw (-0.2,0.3) arc(180:0:0.2) (0.2,0.3) \braiddown (-0.2,-0.3);
            \draw[wipe] (-0.2,-0.3) arc(180:360:0.2) \braidup (-0.2,0.3);;
            \draw[->] (-0.2,-0.3) arc(180:360:0.2) \braidup (-0.2,0.3);;
            \multdot{0.2,0.3}{west}{a};
            \multdot{0.2,-0.3}{west}{b+1};
        \end{tikzpicture}
        \\
        &\overset{\mathclap{\cref{neg}}}{\underset{\mathclap{\cref{rightcurl}}}{=}}\
        \begin{tikzpicture}[centerzero]
            \draw[<-] (-0.2,-0.5) -- (-0.2,0.5);
            \draw[->] (0.2,-0.5) -- (0.2,0.5);
            \singdot{0.2,0};
            \multdot{-0.2,0}{east}{s};
        \end{tikzpicture}
        + tz \leftbub{s+1}
        + z^2 \sum_{a,b>0} \leftbub{s+a+b+1} \rightplus{-a-b}
        + z^2 \sum_{\substack{a+b=s \\ a,b < 0}}
        \left(
            \sum_{c \ge 0} \leftbub{b-c+1} \rightminus{a+c}
            - \sum_{c>0} \leftbub{b+c+1} \rightplus{a-c}
        \right)
        \\
        &\overset{\mathclap{\cref{poppy}}}{\underset{\mathclap{\cref{weeds+}}}{=}}\
        \begin{tikzpicture}[centerzero]
            \draw[<-] (-0.2,-0.5) -- (-0.2,0.5);
            \draw[->] (0.2,-0.5) -- (0.2,0.5);
            \singdot{0.2,0};
            \multdot{-0.2,0}{east}{s};
        \end{tikzpicture}
        + tz \leftbub{s+1}
        + z^2 \sum_{r>0} (r-1) \leftbub{s+r+1} \rightplus{-r}
        + z^2 \sum_{\substack{a+b=s \\ a,b < 0}}
        \left(
            \sum_{c \ge 0} \leftminus{b-c+1} \rightminus{a+c}
            - \sum_{c>0} \leftminus{b+c+1} \rightplus{a-c}
        \right),
    \end{align*}
    where, in the final sum above, we used the fact that $\leftplus{m} \rightplus{n} = 0$ whenever $m+n<0$, by \cref{weeds+}.  Now,
    \begin{align*}
        \sum_{r>0} (r-1) \leftbub{s+r+1} \rightplus{-r}
        &\overset{\mathclap{\cref{poppy}}}{=}\
        \sum_{r>0} (r-1) \leftplus{s+r+1} \rightplus{-r}
        + \sum_{r>0} (r-1) \leftminus{s+r+1} \rightplus{-r}
        \\
        &\overset{\mathclap{\cref{infgrass}}}{\underset{\mathclap{\cref{weeds+}}}{=}}\ \
        \delta_{s,-1} (k+1) z^{-2} 1_\one + \sum_{r>0} (r-1) \leftminus{s+r+1} \rightplus{-r}
    \end{align*}
    and
    \[
        \sum_{\substack{a+b=s \\ a,b<0}} \sum_{c>0} \leftminus{b+c+1} \rightplus{a-c}
        = \sum_{a=1}^{-s-1} \sum_{c>0} \leftminus{s+a+c+1} \rightplus{-a-c}
        = \sum_{a,c>0} \leftminus{s+a+c+1} \rightplus{-a-c}
        = \sum_{r>0} (r-1) \leftminus{s+r+1} \rightplus{-r},
    \]
    where the second equality follows from the fact that, when $a \ge -s$ and $c>0$, we have $s+a+c+1 > 0$.   Thus
    \begin{align*}
        \begin{tikzpicture}[centerzero]
            \draw[->] (-0.2,-0.5) -- (-0.2,0.5);
            \draw[<-] (0.2,-0.5) -- (0.2,0.5);
            \singdot{-0.2,0};
            \multdot{0.2,0}{west}{s};
        \end{tikzpicture}
        &-
        \begin{tikzpicture}[centerzero]
            \draw[<-] (-0.2,-0.5) -- (-0.2,0.5);
            \draw[->] (0.2,-0.5) -- (0.2,0.5);
            \singdot{0.2,0};
            \multdot{-0.2,0}{east}{s};
        \end{tikzpicture}
        =
        \delta_{s,-1} (k+1) 1_\one + tz \leftbub{s+1}
        + z^2 \sum_{\substack{a+b=s \\ a,b < 0}} \sum_{c \ge 0} \leftminus{b-c+1} \rightminus{a+c}
        \\
        &\overset{\mathclap{\cref{poppy}}}{\underset{\mathclap{\substack{\cref{weeds+} \\ \cref{weeds-}}}}{=}}\ \
        \delta_{s,-1} (k+1) 1_\one + z^2 \leftminus{s+1} \rightminus{0}
        + z^2 \sum_{b=s+2}^{0} \sum_{c \le 0} \leftminus{b+c} \rightminus{s+1-b-c}
        \\
        &\overset{\mathclap{\cref{weeds-}}}{=}\ \
        \delta_{s,-1} (k+1) 1_\one
        + z^2 \sum_{b=s+1}^{0} \sum_{c \le 0} \leftminus{b+c} \rightminus{s+1-b-c}
        \overset{\cref{weeds-}}{=}
        \delta_{s,-1} (k+1) 1_\one
        + z^2 \sum_{r \in \Z} (r+1) \leftminus{-r} \rightminus{r+s+1}
        \\
        &\overset{\mathclap{\cref{infgrass}}}{=}\
        \delta_{s,-1} k 1_\one
        + z^2 \sum_{r \in \Z} r\ \leftminus{-r} \rightminus{r+s+1}
        \overset{\cref{pcannon}}{=}
        \delta_{s,-1} k
        -\{s+1\}
        \begin{tikzpicture}[centerzero]
            \draw (0,0) circle(0.3);
            \node at (0,0) {\dotlabel{s+1}};
        \end{tikzpicture}
        \ .
    \end{align*}
    This completes the proof of \cref{tree}.
\end{proof}

\begin{rem} \label{chess}
    When $k=-1$, \cref{mainthm} can be seen as an extension of \cite[Th.~6.3]{CLLSS18}, which gives an isomorphism between ``half'' of $\EH_{-1}$ and the trace of the $q$-deformed Heisenberg category of \cite{LS13}, which is isomorphic to the monoidal subcategory of $\Heis_{-1}(z,-z^{-1})$ consisting of all objects and all morphisms \emph{not involving negative dots} (which thus can be viewed as ``half'' of $\Heis_{-1}$).  Note that, even in the case $k=-1$, the approach of the current paper has significant advantages.  In particular, the extension of the isomorphism to the full $\EH_k$ allows one to work with the simpler presentation given in \cref{crossbow}.  This allows one, for example, to avoid many of the lengthy and technical arguments of \cite[\S4]{CLLSS18}, such as \cite[Prop.~4.10]{CLLSS18}.
\end{rem}

%===========================================================
\section{Action on symmetric functions\label{sec:centeract}}
%===========================================================

There is a natural action of the trace $\Tr(\Heis_k)$ on the center $\End_{\Heis_k}(\one)$, which we now explore.  Throughout this section we continue with the assumptions on $\kk$, $q$, $z$, $t$, and $k$ made at the beginning of \cref{sec:qheis}.  Let us depict an endomorphism $f \in \End_{\Heis_k}(X)$ by
\[
    \begin{tikzpicture}[centerzero]
        \draw (-0.3,-0.2) rectangle (0.3,0.2);
        \node at (0,0) {\dotlabel{f}};
        \draw[very thick] (0,0.2) -- (0,0.4);
        \draw[very thick] (0,-0.2) -- (0,-0.4);
    \end{tikzpicture}
    \ ,
\]
where the thick vertical strand is a horizontal juxtaposition of upward and downward strands corresponding to $1_X$.  Then we define the action of $\Tr(\Heis_k)$ on $\End_{\Heis_k}(\one)$ by
\begin{equation} \label{hazy}
    \left[
        \begin{tikzpicture}[centerzero]
            \draw (-0.3,-0.2) rectangle (0.3,0.2);
            \node at (0,0) {\dotlabel{f}};
            \draw[very thick] (0,0.2) -- (0,0.4);
            \draw[very thick] (0,-0.2) -- (0,-0.4);
        \end{tikzpicture}
    \right]
    \cdot g
    =
    \begin{tikzpicture}[centerzero]
        \draw (-0.3,-0.2) rectangle (0.3,0.2);
        \node at (0,0) {\dotlabel{f}};
        \draw[very thick] (0,0.2) to[out=up,in=up] (1,0.2) -- (1,-0.2) to[out=down,in=down] (0,-0.2);
        \node at (0.6,0) {$g$};
    \end{tikzpicture}
    \ ,\quad
    f \in \End_{\Heis_k}(X),\
    X \in \Heis_k,\
    g \in \End_{\Heis_k}(\one),
\end{equation}
and extend by linearity.

There is a unique map $\rho \colon \EH_k \otimes \Sym^{\otimes 2} \to \Sym^{\otimes 2}$ making the diagram
\begin{equation} \label{encircle}
    \begin{tikzcd}
        \EH_k \otimes \Sym^{\otimes 2} \arrow[d, "\varphi_k \otimes \beta", "\cong"'] \arrow[r, "\rho"] & \Sym^{\otimes 2} \arrow[d, "\beta", "\cong"'] \\
        \Tr(\Heis_k) \otimes \End_{\Heis_k}(\one) \arrow[r] & \End_{\Heis_k}(\one)
    \end{tikzcd}
\end{equation}
commute, where the bottom horizontal map is given by the action \cref{hazy}.  The map $\rho$ gives an action of $\EH_k$ on $\Sym^{\otimes 2}$ and our goal is to give an explicit description of this action.  We will use the notation $a \cdot \theta$ for $\rho(a \otimes \theta)$, $a \in \EH_k$, $\theta \in \Sym^{\otimes 2}$.

Recall that, for $r \ge 1$, $h_r$, $e_r$, and $p_r$ denote the degree $r$ complete homogeneous symmetric function, elementary symmetric function, and power sum, respectively.  We also adopt the conventions
\[
    h_0=e_0=1,\ p_0=0
    \quad \text{and} \quad
    h_r = e_r = p_r = 0 \text{ for } r < 0.
\]
Recall also the notation $f^\pm$ for $f \in \Sym$ given in \cref{electric}.  Then \cref{hcannon,ecannon,pcannon} are valid for all $r \in \Z$.

\begin{lem}
    We have
    \begin{align} \label{circ1}
        w_{\pm r,0} \cdot \theta
        &=
        \mp \{r\}^{-1} p_r^\pm \theta,&
        r \ge 1,\ \theta \in \Sym^{\otimes 2},
        \\ \label{circ2}
        w_{r,1} \cdot 1
        &= -t^{-1} z^{-1} h_{r-k}^+ + tz^{-1} h_{-r}^-,&
        r \in \Z,
        \\ \label{circ3}
        w_{r,-1} \cdot 1
        &= (-1)^{r+k} t z^{-1} e_{r+k}^+ + (-1)^{r-1} t^{-1} z^{-1} e_{-r}^-,&
        r \in \Z.
    \end{align}
\end{lem}

\begin{proof}
    Equation \cref{circ1} follows from \cref{clouds,fort}.  To see \cref{circ2,circ3}, we compute
    \begin{gather*}
        w_{r,1} \cdot 1
        = \beta^{-1} \left( \rightbubside{r} \right)
        \overset{\cref{poppy}}{=} \beta^{-1} \left( \rightplusside{r} + \rightminusside{r} \right)
        \overset{\cref{hcannon}}{=} -t^{-1}z^{-1} h_{r-k}^+ + tz^{-1} h_{-r}^-,
        \\
        w_{r,-1} \cdot 1
        = \beta^{-1} \left( \leftbubside{r} \right)
        \overset{\cref{poppy}}{=} \beta^{-1} \left( \leftplusside{r} + \leftminusside{r} \right)
        \overset{\cref{ecannon}}{=} (-1)^{r+k} t z^{-1} e_{r+k}^+ + (-1)^{r-1} t^{-1} z^{-1} e_{-r}^-.
    \end{gather*}
\end{proof}

Let $\cP$ be the set of all $(\lambda_{-\ell_-}, \dotsc, \lambda_{-2}, \lambda_{-1},\lambda_1, \lambda_2, \dotsc, \lambda_{\ell_+}) \in \Z^{\ell_-+\ell_+}$, $\ell_-,\ell_+ \in \N$, satisfying
\begin{equation}
    \lambda_{-\ell_-} \le \dotsb \le \lambda_{-1} < 0 < \lambda_1 \le \dotsb \le \lambda_{\ell_+}.
\end{equation}
For such an element $\lambda \in \cP$, we define
\begin{gather}
    \ell_\pm(\lambda) := \ell_\pm,\quad
    \ell(\lambda) := \ell_+ + \ell_-,\quad
    I_\lambda := \{-\ell_-,\dotsc,-1,1,\dotsc,\ell_+\},\quad
    \\
    |\lambda| := \sum_{i \in I_\lambda} \lambda_i,\quad
    \{ \lambda \} := \prod_{i \in I_\lambda} \{\lambda_i\},\quad
    \symbub{\lambda} := \prod_{i \in I_\lambda} \symbub{\lambda_i},\quad
    P_\lambda :=  \frac{1}{\{\lambda\}^2} \left( \prod_{i=1}^{\ell_+} p_{\lambda_i}^+ \right) \left( \prod_{i=1}^{\ell_-} p_{-\lambda_i}^- \right).
\end{gather}
We think of elements of $\cP$ as partitions whose parts can be either positive or negative.  We allow $\ell_-$ or $\ell_+$ (or both) to be zero.  In particular, the empty partition $\varnothing$ is an element of $\cP$.  It follows from \cref{cannon,fort} that
\[
    \{P_\lambda : \lambda \in \cP\}
    \text{ is a basis for } \Sym \otimes \Sym
    \quad \text{and} \quad
    \{\symbub{\lambda} : \lambda \in \cP\}
    \text{ is a basis for } \End_{\Heis_k}(\one).
\]
Furthermore, it follows from \cref{fort} that
\begin{equation}
    \beta(P_\lambda) = \tfrac{(-1)^{\ell(\lambda)}}{\{\lambda\}}\, \symbub{\lambda},\quad
    \lambda \in \cP.
\end{equation}

For $\lambda,\mu \in \cP$ we define
\begin{multline}
    \mu \trianglelefteq \lambda
    \iff
    \mu = (\lambda_{i_a},\dotsc,\lambda_{i_1},\lambda_{j_1},\dotsc,\lambda_{j_b})
    \\
    \text{ for some }
    -\ell_-(\lambda) \le i_a < \dotsb < i_1 < 0 < j_1 < \dotsc < j_b \le \ell_+(\lambda).
\end{multline}
In other words, $\mu \trianglelefteq \lambda$ if $\mu$ is obtained from $\lambda$ by deleting some of its parts (allowing also $\mu = \lambda$).

\begin{lem}
    For $\lambda \in \cP$, we have
    \begin{align} \label{yak}
        \symbub{\lambda}\
        \begin{tikzpicture}[centerzero]
            \draw[->] (0,-0.5) -- (0,0.5);
        \end{tikzpicture}
        &=
        \{\lambda\} \sum_{\mu \trianglelefteq \lambda} \frac{1}{\{\mu\}}
        \begin{tikzpicture}[centerzero]
            \draw[->] (0,-0.5) -- (0,0.5);
            \multdot{0,0}{east}{|\lambda|-|\mu|};
        \end{tikzpicture}
        \ \symbub{\mu}
        \ ,
        &
        \symbub{\lambda}\
        \begin{tikzpicture}[centerzero]
            \draw[<-] (0,-0.5) -- (0,0.5);
        \end{tikzpicture}
        &=
        \{\lambda\} \sum_{\mu \trianglelefteq \lambda} \frac{(-1)^{\ell(\lambda) - \ell(\mu)}}{\{\mu\}}
        \begin{tikzpicture}[centerzero]
            \draw[<-] (0,-0.5) -- (0,0.5);
            \multdot{0,0}{east}{|\lambda|-|\mu|};
        \end{tikzpicture}
        \ \symbub{\mu}\ .
    \end{align}
\end{lem}

\begin{proof}
    This follows from \cref{goatee} by induction on $\ell(\lambda)$.
\end{proof}

\begin{theo} \label{Trantor}
    The action $\rho$ in \cref{encircle} is the unique action of $\EH_k$ on $\Sym \otimes \Sym$ given by
    \begin{align*}
        w_{r,1} \cdot P_\lambda
        &= z^{-1} \sum_{\mu \trianglelefteq \lambda}
        \left( -t^{-1} h_{r+|\lambda|-|\mu|-k}^+ + t h_{-r-|\lambda|+|\mu|}^- \right)
        P_\mu,
        \\
        w_{r,-1} \cdot P_\lambda
        &= z^{-1} \sum_{\mu \trianglelefteq \lambda}
        (-1)^{r + \ell(\lambda) - \ell(\mu) + |\lambda| + |\mu|}
        \left( (-1)^k t e_{r+|\lambda|-|\mu|+k}^+ - t^{-1} e_{-r-|\lambda|+|\mu|}^- \right)
        P_\mu,
    \end{align*}
    for $r \in \Z$, $\lambda \in \cP$.  Furthermore, $\Sym \otimes \Sym$ is a cyclic $\EH_k$-module generated by $1 \in \Sym \otimes \Sym$.
\end{theo}

\begin{proof}
    The given expressions for $w_{r,\pm 1} \cdot P_\lambda$ follow from \cref{yak,circ2,circ3}.  The uniqueness follows from the fact that $\EH_k$ is generated by $w_{r,\pm 1}$, $r \in \Z$; see \cref{slime}.  The fact that the module is generated by $1$ follows from \cref{circ1}.
\end{proof}

\begin{rem} \label{MSaction}
    When $k=0$, the action described in \cref{Trantor} is essentially the action described in \cite[Th.~4.1]{MS17}.  The precise connection is as follows.  Let
    \[
        \gamma = \begin{pmatrix} 0 & -1 \\ 1 & 0 \end{pmatrix} \in \GL_2(\Z),\qquad
        \xi \colon \Sym^{\otimes 2} \to \Sym^{\otimes 2},\quad
        \xi(P_\lambda) = (-1)^{\ell(\lambda)} \{\lambda\} P_\lambda,\ \lambda \in \cP.
    \]
    Then we have a commutative diagram
    \[
        \begin{tikzcd}
            \EH_0 \otimes \Sym^{\otimes 2} \arrow[r, "\rho"] \arrow[d, "\gamma \otimes \xi"] & \Sym^{\otimes 2} \arrow[d, "\xi"] \\
            \EH_0 \otimes \Sym^{\otimes 2} \arrow[r] & \Sym^{\otimes 2}
        \end{tikzcd}
        \ ,
    \]
    where $\gamma \colon \EH_0 \to \EH_0$ is the isomorphism \cref{icy} and the bottom horizontal arrow is the action of \cite[Th.~4.1]{MS17}.
    \details{
        Let $\rho'$ denote the action of \cite[Th.~4.1]{MS17}.  Note that, in the notation of \cite{MS17}, $Q_{\lambda,\varnothing} = s_\lambda \otimes 1$, $Q_{\varnothing,\mu} = 1 \otimes s_\mu$; see \cite[p.~834]{MS17}.  Thus, comparing \cite[(4.2)]{MS17} to the Murnaghan--Nakayama rule, we see that, for $n \ge 1$, $\theta \in \Sym$,
        \[
            \rho'(w_{0,n} \otimes \theta \otimes 1)
            = p_n \theta \otimes 1,\qquad
            \rho'(w_{0,-n} \otimes 1 \otimes \theta)
            = 1 \otimes p_n \theta.
        \]
        Then, for $r \ge 1$, and $\lambda$ a usual partition (i.e.\ with only positive parts) we have
        \begin{gather*}
            \xi \circ \rho(w_{r,0} \otimes p_\lambda \otimes 1)
            \overset{\cref{circ1}}{=} -\{r\}^{-1} \xi( p_r p_\lambda \otimes 1 )
            = (-1)^{\ell(\lambda)} \{\lambda\} p_r p_\lambda \otimes 1 ,
            \\
            \rho' \circ (\gamma \otimes \xi) (w_{r,0} \otimes p_\lambda \otimes 1)
            = (-1)^{\ell(\lambda)} \{\lambda\} \rho'(w_{0,r} \otimes p_\lambda \otimes 1)
            = (-1)^{\ell(\lambda)} \{\lambda\} p_r p_\lambda \otimes 1,
        \end{gather*}
        and
        \begin{gather*}
            \xi \circ \rho(w_{-r,0} \otimes 1 \otimes p_\lambda)
            \overset{\cref{circ1}}{=} \{r\}^{-1} \xi(1 \otimes p_r p_\lambda)
            = \{\lambda\} \otimes p_r p_\lambda,
            \\
            \rho' \circ (\gamma \otimes \xi) (w_{-r,0} \otimes 1 \otimes p_\lambda)
            = \{\lambda\} \rho'(w_{0,-r} \otimes 1 \otimes p_\lambda)
            = \{\lambda\} \otimes p_r p_\lambda.
        \end{gather*}
        In addition, recalling that our $q$ and $t$ are the $s$ and $v^{-1}$ of \cite{MS17}, we have
        \begin{gather*}
            \xi \circ \rho(w_{0,{\pm n}} \otimes 1 \otimes 1)
            \overset{\cref{mouse}}{=} \frac{t^n-t^{-n}}{\{n\}} \xi (1 \otimes 1)
            = \frac{t^n-t^{-n}}{\{n\}} (1 \otimes 1),
            \\
            \rho' \circ (\gamma \otimes \xi) (w_{0,\pm n} \otimes 1 \otimes 1)
            = \rho'(w_{\mp n,0} \otimes 1 \otimes 1)
            = \frac{t^n-t^{-n}}{\{n\}} (1 \otimes 1),
        \end{gather*}
        where, in the last equality, we use the equation at the top of \cite[p.~826]{MS17}.
    }
\end{rem}

Note that, even though the definition of $\EH_k$ does not depend on $t$, its action on $\Sym \otimes \Sym$ does.  So we obtain a family of modules depending on the parameter $t$.  In addition, the following result shows that the cyclic vector $1 \in \Sym \otimes \Sym$ is an eigenvector for many of the $w_{0,n}$, with the eigenvalues depending on $t$.

\begin{prop}
    We have
    \begin{align} \label{mouse}
        w_{0,n} \cdot 1
        &= \frac{t^n - t^{-n}}{\{n\}}
        \quad \text{if } k=0,\ n \ne 0,
        &
        w_{0,n} \cdot 1
        &= \frac{t^n}{\{n\}}
        \quad \text{if } k,n > 0 \text{ or } k,n < 0.
    \end{align}
\end{prop}

\begin{proof}
    First suppose $k=0$.  Then we have
    \begin{equation} \label{heyna}
        \begin{tikzpicture}[centerzero]
            \draw[->] (0.2,0) to[out=down,in=east] (0.05,-0.2) to[out=left,in=down] (-0.2,0.35) -- (-0.2,0.5);
            \draw[wipe] (-0.2,-0.5) -- (-0.2,-0.35) to[out=up,in=west] (0.05,0.2) to[out=right,in=up] (0.2,0);
            \draw (-0.2,-0.5) -- (-0.2,-0.35) to[out=up,in=west] (0.05,0.2) to[out=right,in=up] (0.2,0);
        \end{tikzpicture}
        \ = t\
        \begin{tikzpicture}[centerzero]
            \draw[->] (0,-0.5) to (0,0.5);
        \end{tikzpicture}
        \ ,\qquad
        \begin{tikzpicture}[centerzero]
            \draw (-0.2,-0.5) -- (-0.2,-0.35) to[out=up,in=west] (0.05,0.2) to[out=right,in=up] (0.2,0);
            \draw[wipe] (0.2,0) to[out=down,in=east] (0.05,-0.2) to[out=left,in=down] (-0.2,0.35) -- (-0.2,0.5);
            \draw[->] (0.2,0) to[out=down,in=east] (0.05,-0.2) to[out=left,in=down] (-0.2,0.35) -- (-0.2,0.5);
        \end{tikzpicture}
        \ = t^{-1}\
        \begin{tikzpicture}[centerzero]
            \draw[->] (0,-0.5) to (0,0.5);
        \end{tikzpicture}
        \ ,\qquad
        \begin{tikzpicture}[centerzero]
            \draw (0.2,0) to[out=down,in=east] (0.05,-0.2) to[out=left,in=down] (-0.2,0.35) -- (-0.2,0.5);
            \draw[wipe] (-0.2,-0.5) -- (-0.2,-0.35) to[out=up,in=west] (0.05,0.2) to[out=right,in=up] (0.2,0);
            \draw[<-] (-0.2,-0.5) -- (-0.2,-0.35) to[out=up,in=west] (0.05,0.2) to[out=right,in=up] (0.2,0);
        \end{tikzpicture}
        \ = t\
        \begin{tikzpicture}[centerzero]
            \draw[,-] (0,-0.5) to (0,0.5);
        \end{tikzpicture}
        \ ,\qquad
        \begin{tikzpicture}[centerzero]
            \draw[<-] (-0.2,-0.5) -- (-0.2,-0.35) to[out=up,in=west] (0.05,0.2) to[out=right,in=up] (0.2,0);
            \draw[wipe] (0.2,0) to[out=down,in=east] (0.05,-0.2) to[out=left,in=down] (-0.2,0.35) -- (-0.2,0.5);
            \draw (0.2,0) to[out=down,in=east] (0.05,-0.2) to[out=left,in=down] (-0.2,0.35) -- (-0.2,0.5);
        \end{tikzpicture}
        \ = t^{-1}\
        \begin{tikzpicture}[centerzero]
            \draw[<-] (0,-0.5) to (0,0.5);
        \end{tikzpicture}
        \ .
    \end{equation}
    This can be seen by using the fact that $\Heis_0$ is the framed HOMFLYPT skein category (\cref{Polly}) or by using \cref{weeds+,weeds-,leftcurl,rightcurl}.

    If $n > 0$, we have
    \[
        w_{0,n} \cdot 1
        \overset{\cref{sun}}{=} \frac{z}{\{n\}} \sum_{i=0}^{n-1} [\sigma_{i,n-i-1}] \cdot 1
        \overset{\cref{heyna}}{\underset{\cref{curls1}}{=}}\ \frac{z}{\{n\}} \sum_{i=0}^{n-1} t^{n-2i-1} \frac{t-t^{-1}}{z}
        = \frac{t^n-t^{-n}}{\{n\}}.
    \]
    The case $n < 0$ is analogous
    \details{
        For $n = -m < 0$, we have
        \[
            w_{0,-m} \cdot 1
            \overset{\cref{sun}}{=} \frac{z}{\{m\}} \sum_{i=0}^{m-1} [\sigma_{-i,i-m-i+1}] \cdot 1
            \overset{\cref{heyna}}{\underset{\cref{curls2}}{=}}\ \frac{z}{\{m\}} \sum_{i=0}^{m-1} t^{m-2i-1} \frac{t-t^{-1}}{z}
            = \frac{t^n-t^{-m}}{\{m\}}.
        \]
    }

    Now suppose $n,k>0$.  Then, by \cref{weeds+,weeds-,leftcurl,rightcurl}, we have
    \begin{equation} \label{heyna2}
        \begin{tikzpicture}[centerzero]
            \draw[->] (0.2,0) to[out=down,in=east] (0.05,-0.2) to[out=left,in=down] (-0.2,0.35) -- (-0.2,0.5);
            \draw[wipe] (-0.2,-0.5) -- (-0.2,-0.35) to[out=up,in=west] (0.05,0.2) to[out=right,in=up] (0.2,0);
            \draw (-0.2,-0.5) -- (-0.2,-0.35) to[out=up,in=west] (0.05,0.2) to[out=right,in=up] (0.2,0);
        \end{tikzpicture}
        \ = t\
        \begin{tikzpicture}[centerzero]
            \draw[->] (0,-0.5) to (0,0.5);
        \end{tikzpicture}
        \ ,\qquad
        \begin{tikzpicture}[centerzero]
            \draw (-0.2,-0.5) -- (-0.2,-0.35) to[out=up,in=west] (0.05,0.2) to[out=right,in=up] (0.2,0);
            \draw[wipe] (0.2,0) to[out=down,in=east] (0.05,-0.2) to[out=left,in=down] (-0.2,0.35) -- (-0.2,0.5);
            \draw[->] (0.2,0) to[out=down,in=east] (0.05,-0.2) to[out=left,in=down] (-0.2,0.35) -- (-0.2,0.5);
        \end{tikzpicture}
        \ = 0.
    \end{equation}
    Thus
    \[
        w_{0,n} \cdot 1
        \overset{\cref{sun}}{=} \frac{z}{\{n\}} \sum_{i=0}^{n-1} [\sigma_{i,n-i-1}] \cdot 1
        \overset{\cref{heyna2}}{\underset{\cref{curls1}}{=}}\ \frac{z}{\{n\}} t^{n-1} \frac{t}{z}
        = \frac{t^n}{\{n\}}.
    \]
    The case $n,k < 0$ can be proved directly in an analogous manner, or obtained from the $n,k>0$ case by applying the isomorphism $\Omega_k$ from \cref{earth}.
\end{proof}

%=========================================================
\section{Action on cocenters of cyclotomic Hecke algebras}
%=========================================================

In this final section, we describe the action of $\Tr(\Heis_k)$ on traces of cyclotomic quotients of $\Heis_k$ or, equivalently, on cocenters of cyclotomic Hecke algebras.  Throughout this section we assume that $\kk$ is a field of characteristic zero and $q,t \in \kk^\times$, with $q$ not a root of unity.

If $\cC$ and $\cD$ are $\kk$-linear categories, we let $\cC \boxtimes \cD$ be the $\kk$-linear category with objects that are pairs $(X,Y)$ of objects $X \in \cC$ and $Y \in \cD$, and morphisms given by
\[
    \Hom_{\cC \boxtimes \cD}((X_1,Y_1), (X_2,Y_2))
    := \Hom_\cC(X_1,X_2) \otimes_\kk \Hom_\cD(Y_1,Y_2).
\]
Composition of morphisms in $\cC \boxtimes \cD$ is given by $(e \otimes f) \circ (g \otimes h) := (e \circ g) \otimes (f \circ h)$, extended by linearity.  It is straightforward to verify that we have a linear isomorphism
\[
    \Tr(\cC) \otimes \Tr(\cD)
    \xrightarrow{\cong} \Tr(\cC \boxtimes \cD),\quad
    [f] \otimes [g] \mapsto [f \otimes g].
\]

A \emph{module category} over $\cC$ is a $\kk$-linear category $\cM$, together with a $\kk$-linear functor $\cC \boxtimes \cM \to \cM$ satisfying the usual associativity and unity axioms.  Equivalently, it is a $\kk$-linear category $\cM$ together with a strict $\kk$-linear monoidal functor $\cC \to \cEnd_\kk(\cM)$, where $\cEnd_\kk(\cM$) denotes the strict $\kk$-linear monoidal category whose objects are $\kk$-linear endofunctors of $\cM$ and whose morphisms are natural transformations.

If $\cM$ is a module category over $\cC$ with action functor $F \colon \cC \boxtimes \cM \to \cM$, then we have an induced action of $\Tr(\cC)$ on $\Tr(\cM)$ given by
\begin{equation}
    \Tr(\cC) \otimes \Tr(\cM)
    \xrightarrow{\cong} \Tr(\cC \boxtimes \cM)
    \xrightarrow{\Tr[F]} \Tr(\cM).
\end{equation}
The goal of this section is to use this fact to construct $\EH_k$-modules from certain module categories over the quantum Heisenberg category.

For $n \in \Z_{\ge 1}$, let $\rH_n$ denote the Iwahori--Hecke algebra of rank $n$.  This is the associative $\kk$-algebra with generators $\tau_1,\dotsc,\tau_{n-1}$ and relations
\begin{align*}
    \tau_i \tau_j &= \tau_j \tau_i,& 1 \le i,j \le n-1,\ |i-j| > 1,\\
    \tau_i \tau_{i+1} \tau_i &= \tau_{i+1} \tau_i \tau_{i+1},& 1 \le i \le n-2,\\
    \tau_i^2 &= z\tau_i + 1,& 1 \le i \le n-1.
\end{align*}
Let $\AH_n$ denote the affine Hecke algebra of rank $n$.  Thus $\AH_n = \rH_n \otimes_\kk \kk[x_1^{\pm 1},\dotsc,x_n^{\pm n}]$ as $\kk$-modules, with the two factors being subalgebras, and
\[
    \tau_i x_i = x_{i+1} \tau_i^{-1},\quad 1 \le i \le n-1.
\]
We adopt the convention that $\rH_0 = \AH_0 = \kk$.

It follows from \cref{braid,skein,dotslide} that we have an algebra homomorphism
\begin{equation} \label{rabbit}
    \AH_n \to \End_{\Heis_k}(\uparrow^{\otimes n}),
\end{equation}
sending $x_i$ to a dot on the $i$-th string and $\tau_i$ to a positive crossing $\posupcross$ of the $i$-th and $(i+1)$-st strings, where we number strings $1,2,\dotsc,n$ from right to left.  In fact, it follows from \cref{basisthm} that this map is injective.

Fix a nonnegative integer $l$ and a polynomial
\begin{equation}
    f(u) = f_0 u^l + f_1 u^{l-1} + \dotsb + f_l \in \kk[u],\quad
    f_0 = 1,\ f_l = t^2.
\end{equation}
(Note that this forces $t = \pm 1$ if $l=0$.)  The \emph{cyclotomic Hecke algebra} $\rH_n^f$ of level $l$ associated to the polynomial $f(u)$ is the quotient of $\AH_n$ by the two-sided ideal generated by $f(x_1)$.  By convention $\rH_0^f = \kk$.  The basis theorem proved in \cite[Th.~3.10]{AK94} states that
\[
    \{ x_1^{r_1} \dotsm x_n^{r_n} \tau_g : 0 \le r_1,\dotsc,r_n < l,\ g \in \fS_n \}
\]
is a basis for $\rH_n^f$ as a free $\kk$-module, where $\tau_g$ denotes the element of the finite Hecke algebra defined from a reduced expression for the element $g$ of the symmetric group $\fS_n$.

Define the \emph{tower of cyclotomic Hecke algebras} associated to $f$ to be the $\kk$-linear category $\cH^f$ with objects $X_n$, $n \in \N$, and
\[
    \Hom_{\cH^f}(\uparrow^{\otimes n}, \uparrow^{\otimes m})
    =
    \begin{cases}
        \rH_n^f & \text{if } m=n, \\
        0 & \text{if } m \ne n.
    \end{cases}
\]
Note that $\cH^f$ is \emph{not} naturally a monoidal category.  As we will now explain the tower $\cH^f$ can also be realized as a cyclotomic quotient of the quantum Heisenberg category.

Let $\cI(f)$ be the left tensor ideal generated by the morphism $f(\updot)$.  The \emph{cyclotomic quantum Heisenberg category} associated to the polynomial $f(u)$ is the quotient category
\begin{equation}
    \cH(f) := \Heis_{-l}/\cI(f).
\end{equation}
Note that $\cH(f)$ is a $\kk$-linear category, but it does not inherit the monoidal structure from $\Heis_{-l}$.  However, it is a left module category over $\Heis_{-l}$.

\begin{prop} \label{cyclotomic}
    The map \cref{rabbit} induces algebra isomorphisms
    \begin{equation} \label{lettuce}
        \rH_n^f \xrightarrow{\cong} \End_{\cH(f)}(\uparrow^{\otimes n}),\quad n \in \N.
    \end{equation}
    Furthermore, the functor $\cH^f \to \cH(f)$ given on objects by $X_n \mapsto\ \uparrow^{\otimes n}$ and on morphisms by \cref{lettuce} is an equivalence of $\kk$-linear categories.
\end{prop}

\begin{proof}
    This is shown in \cite[Lem.~5.13]{BSW-HKM} and \cite[Th.~9.5]{BSW-qheis}.
\end{proof}

\begin{rem}
    As shown in \cite[Th.~9.5]{BSW-qheis}, the equivalence of \cref{cyclotomic} induces an equivalence of $\kk$-linear categories
    \[
        \Kar(\cH(f)) \to \bigoplus_{n \ge 0} \rH_n^f\pmd,
    \]
    where $\Kar(\cH(f))$ denotes the additive Karoubi envelope of $\cH(f)$ and $\rH_n^f\pmd$ denotes the category of finitely-generated projective left $\rH_n^f$-modules.  Under this isomorphism, the natural action of $\Heis_k$ on $\cH(f)$ corresponds to an action of $\Heis_k$ on $\bigoplus_{n \ge 0} \rH_n^f\pmd$, with the objects $\uparrow$ and $\downarrow$ acting by induction and restriction, respectively.  We refer the reader to \cite[\S 6]{BSW-qheis} for details.
\end{rem}

For an associative $\kk$-algebra $A$, its \emph{cocenter} is
\[
    C(A) := A/\Span_\kk\{ab - ba : ab \in A\}.
\]
Note that this is the same as the trace of $A$ considered as a monoidal category with one object.  By \cref{cyclotomic,mainthm}, we have an action of $\EH_{-l} \cong \Tr(\Heis_{-l})$ on
\begin{equation}
    V_f := \Tr(\cH(f)) \cong \bigoplus_{n \ge 0} C(\rH_n^f).
\end{equation}
Denote this action by $\cdot$, and let $v_f$ denote unit element in $C(\rH_0^f) \cong \kk$.

\begin{prop} \label{hoops}
    The $\EH_{-l}$-module $V_f$ is cyclic, generated by $v_f$.  Furthermore, we have
    \begin{align} \label{hoop1}
        w_{r,n} \cdot v_f
        &= 0,& r \in \Z,\ n < 0,
        \\ \label{hoop2}
        w_{r,1} \cdot v_f
        &= - \sum_{i=1}^l f_i w_{r-i,1} \cdot v_f, & r \in \Z,
        \\ \label{hoop3}
        \sum_{r \ge 1} u^{1-r} \{r\} w_{r,0} \cdot v_f
        &= \left( u^2 f'(u)f(u)^{-1} - lu \right) v_f,
        \\ \label{hoop4}
        \sum_{r \ge 1} u^{r-1} \{r\} w_{-r,0} \cdot v_f
        &= \left( f'(u)f(u)^{-1} \right) v_f,
    \end{align}
    where \cref{hoop3} and \cref{hoop4} are equalities of Laurent series in $\kk \Laurent{u^{-1}}$ and $\kk \Laurent{u}$, respectively.
\end{prop}

\begin{proof}
    The fact that $V_f$ is cyclic, generated by $v_f$, follows from the fact that this module is a quotient of $\Tr(\Heis_{-l})$.  The equalities \cref{hoop1,hoop2} also follow immediately from the definition of $\cI(f)$.

    It is shown in \cite[Lem.~9.2]{BSW-qheis} that the ideal $\cI(f)$ contains the morphism $\downstrand$ and the coefficients of the series
    \begin{equation} \label{mtb}
        \leftplusgen - f(u)^{-1},\quad
        \leftminusgen - t^2 f(u)^{-1},\quad
        \rightplusgen - f(u),\quad
        \rightminusgen - t^{-2} f(u),
    \end{equation}
    where the first and second occurrences of $f(u)^{-1}$ are interpreted as Laurent series in $u^{-1}$ and $u$, respectively; cf.\ \cref{bg1,bg3}.  Thus, recalling the series defined in \cref{mars}, we have
    \begin{multline*}
        \sum_{r \ge 1} u^{1-r} \{r\} w_{r,0} \cdot v_f
        \overset{\cref{clouds}}{\underset{\cref{fort}}{=}}
        - \beta \left( P_+(u) \right) \cdot v_f
        \overset{\cref{blob}}{=}\ u^2 \beta(H_+'(u)/H_+(u)) \cdot v_f
        \\
        \overset{\cref{road}}{\underset{\cref{mtb}}{=}}
        \left( u^2 f'(u)f(u)^{-1} - lu \right) \cdot v_f.
    \end{multline*}
    This proves \cref{hoop3}; the proof of \cref{hoop4} is similar.
    \details{
        Using the notation in the details environment in the proof of \cref{goatee}, we have
        \[
            \sum_{r \ge 1} u^{r-1} \{r\} w_{-r,0} \cdot v_f
            \overset{\cref{clouds}}{\underset{\cref{fort}}{=}}
            \beta \left( P_-(u) \right) \cdot v_f
            =\ \beta(H_-'(u)/H_-(u)) \cdot v_f
            \overset{\cref{mtb}}{=} \left( f'(u)f(u)^{-1} \right) v_f.
        \]
    }
\end{proof}

It follows from \cref{hoops} that the cyclic vector $v_f$ generates a one-dimensional subspace under the action of the commutative subalgebra of $\EH_{-l}$ generated by the $w_{r,0}$, $r \in \Z \setminus \{0\}$, and that $v_f$ is annihilated by the elements $w_{r,n}$, $r \in \Z$, $n <0$.  In this way, $V_f$ is somewhat like a lowest weight module.

As for the action on the center described in \cref{sec:centeract}, the action of $\EH_{-l}$ on $V_f$ depends on $t$ (since $f$ can involve $t$) even though $\EH_{-l}$ does not.  \Cref{hoops} is not a complete algebraic description of the action of $\EH_{-l}$ on $V_f$ since it only describes the action of certain elements on the cyclic vector $v_f$.  To give a complete algebraic description of the action, one would need to give an explicit description of the images of the elements $x_i^r \in \AH_n$ in the cyclotomic quotients $\rH_n^f$.

\begin{rem} \label{glow}
    When $l=1$, we have $\rH_n^f \cong \rH_n$, with the $x_i$ being sent to the Jucys--Murphy elements.  In this case, explicit formulas are known, and the action on $V_f \cong \Sym$ (for $t=-z^{-1}$ and $f(u)=u+z^{-2}$) was computed in \cite[\S7]{CLLSS18} for ``half'' of $\EH_{-1}$; see \cref{chess}.  The action computed in \cite[\S7]{CLLSS18} is a twist of the polynomial representation defined in \cite[\S1]{SV13}, where it is also realized in terms of the $K$-theory of the Hilbert scheme of $\mathbb{A}^2$.  It is natural to expect that, for higher level $l$, the modules $V_f$ are related to the $K$-theory of the moduli space of framed torsion-free sheaves on $\mathbb{P}^2$, which can be viewed as higher rank analogues of the Hilbert scheme; see \cite[\S8]{SV13}.
\end{rem}

For the remainder of this subsection, we assume that $\kk$ is an algebraically closed field of characteristic zero.  Let $I$ be the union of the orbits of the roots of $f(u)$ under the maps $i \mapsto q^{\pm 2} i$, $i \in \kk$.  It follows from our assumption \cref{trevor} that the map $i \mapsto q^2 i$ defines oriented edges making the set $I$ into a quiver with connected components of type $A_\infty$.  Let $\fg$ denote the Kac--Moody Lie algebra associated to this quiver, and let $\cU(\fg)$ be the corresponding Kac--Moody 2-category, as defined in \cite{KL10,Rou08}.  (See also \cite{Bru16}, which unified the two approaches.)

For $i \in I$, let $\mu_i$ be the multiplicity of $i$ as a root of $f(u)$.  Then let $\mu := \sum_{i \in I} \mu_i \Lambda_i$ be the corresponding dominant integral weight of $\fg$, where $\Lambda_i$ denotes the fundamental weight corresponding to $i \in I$.  By \cite[Th.~B]{BSW-HKM}, $\cH(f)$ is isomorphic (as a locally unital algebra) to the cyclotomic quotient $\cH(\mu)$ of $\cU(\fg)$ corresponding to $\mu$.  As for the case of the quantum Heisenberg category discussed above, this implies that we have an action of $\Tr(\cU(\fg))$ on $\Tr(\cH(\mu)) \cong \Tr(\cH(f)) \cong V_f$.

To any symmetrizable Kac--Moody algebra $\fg$, one can associate a Lie algebra $C\fg$ (denoted $L\fg$ in \cite[Def.~3.24]{SVV17} and, when $\fg$ is simply laced, $\mathcal{C}\fg$ in \cite[\S3.2]{BHLW17}).  The Lie algebra $C\fg$ is isomorphic to the current algebra $\fg \otimes \kk[t]$ when $\fg$ is of finite type $ADE$, but is larger in general; see \cite[Rem.~3.26]{SVV17}.  If $\fg$ is of \emph{finite} type $A$, then it follows from \cite[Th.~1]{SVV17} and \cite[Th.~A, B]{BHLW17}, that $\Tr(\cU(\fg))$ is isomorphic to an idempotented form of the universal enveloping algebra $U(C\fg)$ and that the induced action on cyclotomic quotients realizes its Weyl modules.  While \cite{SVV17,BHLW17} do not treat type $A_\infty$, we expect that one should be able to take an appropriate limit to handle this case.  In this way one would identify $V_f$ as both a Weyl module for the current algebra and a module for the elliptic Hall algebra.

\begin{rem}
    In general, one associates a generalized cyclotomic quotient $\cH(f|g)$ to a pair $(f,g)$ of monic polynomials.  If $f$ and $g$ are of degrees $l$ and $m$, respectively, then $\cH(f|g)$ is a module category over $\Heis_{m-l}$.  We have restricted our attention here to the case where $g=1$, and hence to central charge $k = -l <0$.  As we saw in \cref{hoops}, this gives rise to a negative-central-charge module $V_f$ generated by an eigenvector $v_f$ for $\EH^-$.  If we instead considered the case where $f=1$, we would obtain positive-central-charge modules generated by an eigenvector for $\EH^+$.  Alternatively these positive-central-charge modules can be obtained from the $V_f$ by twisting by the automorphism $\Omega_k$ from \cref{earth}.  The general case, of arbitrary $f$ and $g$, would yield a tensor product of these positive- and negative-central-charge modules; see \cref{tensor}.  We refer the reader to \cite[\S9]{BSW-qheis} for further details on the more general $\cH(f|g)$.
\end{rem}

%========
\appendix
%===================================================
\section{Universal central extension\label{sec:uce}}
%===================================================

In this section we prove \cref{uce} using an argument inspired by the one in \cite{LT05}.  We make the same assumptions on the ground field $\kk$ as given at the beginning of \cref{sec:EHA}, and we assume that $q$ is not integral over the canonical image of $\Z$ in $\kk$.

Let
\begin{equation} \label{archer}
    0 \to \mathfrak{Z} \to \widehat{\fEH} \xrightarrow{\hat{\pi}} \fEH \to 0
\end{equation}
be an arbitrary central extension of $\fEH$.  We must show that there exists a unique homomorphism of Lie algebras $\zeta \colon \fEHc \to \widehat{\fEH}$ such that $\hat{\pi} \zeta = \tilde{\pi}$, where
\[
    \tilde{\pi} \colon \fEHc \to \fEH,\qquad
    w_\bx \mapsto w_\bx,\quad z \mapsto 0,\quad \bx \in \bZ^*,\ z \in \bZ_\kk.
\]

Fix a linear map (not necessarily a homomorphism of Lie algebras) $\zeta_1 \colon \fEH \to \widehat{\fEH}$ such that $\hat{\pi} \zeta_1 = \id_{\fEH}$.  Then define
\begin{equation} \label{magpie}
    \vartheta \colon \fEH \times \fEH \to \mathfrak{Z},\quad
    \vartheta(x,y) = [\zeta_1(x), \zeta_1(y)] - \zeta_1([x,y]),\quad x,y \in \fEH.
\end{equation}
\details{
    Note that the image of $\vartheta$ in contained in $\mathfrak{Z}$ since
    \[
        \hat{\pi}(\varphi(x,y))
        = \hat{\pi} \left( [\zeta_1(x), \zeta_1(y)] \right) - \hat{\pi}\zeta_1([x,y])
        = [x,y] - [x,y]
        = 0,\quad
        x,y \in \fEH.
    \]
}
It follows immediately from the fact that the Lie bracket on $\fEH$ is alternating and satisfies the Jacobi identity that
\begin{gather} \label{alternate}
    \vartheta(x,y) = - \vartheta(y,x),\qquad
    \vartheta(x,x) = 0,
    \\ \label{bicycle}
    \vartheta([x,y],z) + \vartheta([y,z],x) + \vartheta([z,x],y) = 0,
\end{gather}
for all $x,y,z \in \fEH$.  (The two equations in \cref{alternate} are equivalent when $\operatorname{char}(\kk) \ne 2$.)

\begin{lem} \label{milk}
    The function
    \[
        \bZ \to \mathfrak{Z},\qquad
        \bx \mapsto
        \begin{cases}
            \vartheta(w_\bx,w_{-\bx}) & \text{if } \bx \ne 0,\\
            0 & \text{if } \bx = 0,
        \end{cases}
    \]
    is a homomorphism of additive groups.
\end{lem}

\begin{proof}
    Suppose $\bx,\by \in \bZ^*$ satisfy $\det \begin{pmatrix} \bx & \by \end{pmatrix} \ne 0$, that is, $\bx$ and $\by$ are not colinear.  Taking $x=w_\bx$, $y=w_\by$, $z=w_{-\bx-\by}$ in \cref{bicycle}, and letting $\bz = -\bx-\by$, we have
    \begin{equation} \label{dets}
        \{\det \begin{pmatrix} \bx & \by \end{pmatrix}\} \vartheta(w_{\bx+\by},w_{-\bx-\by})
        + \{\det \begin{pmatrix} \by & \bz \end{pmatrix}\} \vartheta(w_{-\bx}, w_\bx)
        + \{\det \begin{pmatrix} \bz & \bx \end{pmatrix}\} \vartheta(w_{-\by}, w_\by).
    \end{equation}
    The determinants appearing in \cref{dets} are all equal, and so, also using \cref{alternate}, we have
    \[
        \vartheta(w_{\bx+\by}, w_{-\bx-\by}) = \vartheta(w_\bx,w_{-\bx}) + \vartheta(w_\by,w_{-\by}).
    \]
    On the other hand, if $\bx$ and $\by$ are colinear, choose $\bz$ that is not colinear with $\bx$ (equivalently, not colinear with $\by$).  Then, using the non-colinear case proved above, we have
    \begin{align*}
        \vartheta(w_{\bx+\by},w_{-\bx-\by})
        &= \vartheta(w_{\bx+\bz}, w_{-\bx-\bz}) + \vartheta(w_{\by-\bz},w_{-\by+\bz}) \\
        &= \vartheta(w_\bx,w_{-\bx}) + \vartheta(w_\bz,w_{-\bz}) + \vartheta(w_\by,w_{-\by}) + \vartheta(w_{-\bz}, w_\bz) \\
        &\overset{\mathclap{\cref{alternate}}}{=}\ \vartheta(w_\bx,w_{-\bx}) + \vartheta(w_\by,w_{-\by}),
    \end{align*}
    completing the proof.
\end{proof}

\begin{lem} \label{vanish}
    We have
    \[
        \vartheta(w_{r,0}, w_{s,0})
        = 0
        = \vartheta(w_{0,m}, w_{0,n}),\quad
        r,s,m,n \in \Z \setminus \{0\},\ r \ne -s,\ m \ne -n.
    \]
\end{lem}

\begin{proof}
    We prove the first equality, since the proof of the second is similar.  Since the case $r=s$ follows immediately from \cref{alternate}, suppose $r \ne s$.  Taking $x=w_{0,1}$, $y=w_{0,-1}$, $z \in w_{r+s,0}$ in \cref{bicycle}, and then dividing by $\{r+s\}$, we have
    \begin{equation} \label{can1}
        \vartheta(w_{r+s,-1}, w_{0,1}) + \vartheta(w_{r+s,1}, w_{0,-1}) = 0.
    \end{equation}
    Next, take $x=w_{s,-1}$, $y=w_{0,1}$, $z \in w_{r,0}$ in \cref{bicycle} to get
    \begin{equation} \label{can2}
        \{s\} \vartheta(w_{s,0}, w_{r,0}) - \{r\} \vartheta(w_{r,1}, w_{s,-1}) -\{r\} \vartheta(w_{r+s,-1}, w_{0,1}) = 0.
    \end{equation}
    Then take $x=w_{r,1}$, $y=w_{0,-1}$, $z=w_{s,0}$ in \cref{bicycle} to get
    \begin{equation} \label{can3}
        -\{r\} \vartheta(w_{r,0}, w_{s,0}) + \{s\} \vartheta(w_{s,-1}, w_{r,1}) + \{s\} \vartheta(w_{r+s,1}, w_{0,-1}) = 0.
    \end{equation}
    Subtracting $\{s\}$ times \cref{can2} from $\{r\}$ times \cref{can3}, then using \cref{can1,alternate}, we have
    \[
        (\{s\}^2 - \{r\}^2) \vartheta(w_{r,0}, w_{s,0}) = 0.
    \]
    Since $r \ne \pm s$ and $q$ is not integral over the canonical image of $\Z$ in $\kk$, we have $\{s\}^2 \ne \{r\}^2$, and so the result follows.
\end{proof}

Define
\begin{equation} \label{dino}
    \zeta_2 \colon \fEH \to \mathfrak{Z},\quad
    \zeta_2(w_{r,n}) =
    \begin{cases}
        \frac{1}{\{rn\}} \vartheta(w_{r,0}, w_{0,n}) & \text{if } r,n \ne 0, \\
        \frac{1}{\{n\}} \vartheta(w_{1,n}, w_{-1,0}) & \text{if } r=0,\ n \ne 0, \\
        \frac{1}{\{r\}} \vartheta(w_{0,-1}, w_{r,1}) & \text{if } r \ne 0, n=0.
    \end{cases}
\end{equation}

\begin{lem} \label{fairy}
    We have
    \begin{equation} \label{tooth}
        \vartheta(w_\bx, w_\by) = \zeta_2([w_\bx, w_\by]),\quad \bx,\by \in \bZ^*,\ \bx + \by \ne 0.
    \end{equation}
\end{lem}

\begin{proof}
    Let $\bx = (r,n)$, $\by = (s,m)$.  If $n=m=0$ or $r=s=0$, then the result holds by \cref{vanish}.  If $s=n=0$, then we have
    \[
        \vartheta(w_{r,0}, w_{0,m})
        \overset{\cref{dino}}{=} \{rm\} \zeta_2(w_{r,m})
        \overset{\cref{sloth}}{=} \zeta_2([w_{r,0}, w_{0,m}]).
    \]
    The case $r=m=0$ then follows by using the fact that both sides of \cref{tooth} are antisymmetric in the arguments $w_\bx$ and $w_\by$.  We have now proved that \cref{tooth} holds when at least two of $r,s,m,n$ are zero.  Therefore, for the remainder of the proof, we assume that at most one of these is zero.

    Suppose that $m=0$ and $r+s \ne 0$.  Taking $x=w_{-s,n}$, $y=w_{s,0}$, $z=w_{r+s,0}$ in \cref{bicycle}, we have
    \[
        -\{sn\} \vartheta(w_{0,n}, w_{r+s,0}) + \{(r+s)n\} \vartheta(w_{r,n}, w_{s,0}) = 0.
    \]
    Thus, we have
    \[
        \vartheta(w_{r,n}, w_{s,0})
        = \frac{\{sn\}}{\{(r+s)n\}} \vartheta(w_{0,n}, w_{r+s,0})
        \overset{\cref{dino}}{=} -\{sn\} \zeta_2(w_{r+s,n})
        \overset{\cref{sloth}}{=} \zeta_2(w_{r,n}, w_{s,0}).
    \]
    On the other hand, if $m=0=r+s$, then, taking $x=w_{-1,0}$, $y=w_{r+1,n}$, $z=w_{-r,0}$ in \cref{bicycle}, we have
    \[
        -\{n\} \vartheta(w_{r,n},w_{-r,0}) +\{rn\} \vartheta(w_{1,n}, w_{-1,0}) = 0.
    \]
    Therefore
    \[
        \vartheta(w_{r,n}, w_{-r,0})
        = \frac{\{rn\}}{\{n\}} \vartheta(w_{1,n}, w_{-1,0})
        \overset{\cref{dino}}{=} \{rn\} \zeta_2(w_{0,n})
        \overset{\cref{sloth}}{=} \zeta_2([w_{r,n}, w_{-r,0}]).
    \]
    This completes the proof of \cref{tooth} when $m=0$.  The case $n=0$ then follows by using the fact that both sides of \cref{tooth} are antisymmetric in the arguments $w_\bx$ and $w_\by$.  The cases $r=0$ and $s=0$ are similar.
    \details{
        Suppose that $s=0$ and $m+n \ne 0$.  Taking $x=w_{r,-m}$, $y=w_{0,m}$, $z=w_{0,m+n}$ in \cref{bicycle}, we have
        \[
            \{rm\} \vartheta(w_{r,0}, w_{0,m+n}) - \{r(m+n)\} \vartheta(w_{r,n}, w_{0,m}) = 0.
        \]
        Thus,
        \[
            \vartheta(w_{r,n}, w_{0,m})
            = \frac{\{rm\}}{\{r(m+n)\}} \vartheta(w_{r,0}, w_{0,m+n})
            \overset{\cref{dino}}{=} \{rm\} \zeta_2(w_{r,n})
            \overset{\cref{sloth}}{=} \zeta_2([w_{r,n}, w_{0,m}]).
        \]
        On the other hand, if $s=0=m+n$, then, taking $x=w_{0,-1}$, $y=w_{r,n+1}$, $z=w_{0,-n}$ in \cref{bicycle}, we have
        \[
            \{r\} \vartheta(w_{r,n}, w_{0,-n}) -\{rn\} \vartheta(w_{r,1}, w_{0,-1}) = 0.
        \]
        Thus,
        \[
            \vartheta(w_{r,n}, w_{0,-n})
            = \frac{\{rn\}}{\{r\}} \vartheta(w_{r,1}, w_{0,-1})
            \overset{\cref{dino}}{=} - \{rn\} \zeta_2(w_{r,n})
            \overset{\cref{sloth}}{\underset{\cref{alternate}}{=}} \zeta_2([w_{r,n}, w_{0,-n}]).
        \]
        The case $r=0$ now follows by using the fact that both sides of \cref{tooth} are antisymmetric in the arguments $w_\bx$ and $w_\by$.
    }

    It remains to consider the case where $r,s,m,n$ are all nonzero.  In this case, taking $x=w_{s,m}$, $y=w_{r,0}$, $z=w_{0,n}$ in \cref{bicycle} gives
    \[
        -\{rm\} \vartheta(w_{r+s,m}, w_{0,n}) + \{rn\} \vartheta(w_{r,n}, w_{s,m}) - \{sn\} \vartheta(w_{s,m+n}, w_{r,0}) = 0.
    \]
    Thus
    \begin{align*}
        \vartheta(w_{r,n}, w_{s,m})
        &= \frac{\{rm\}}{\{rn\}} \vartheta(w_{r+s,m}, w_{0,n}) + \frac{\{sn\}}{\{rn\}} \vartheta(w_{s,m+n}, w_{r,0})
        \\
        &= \frac{\{rm\}}{\{rn\}} \zeta_2([w_{r+s,m}, w_{0,n}]) + \frac{\{sn\}}{\{rn\}} \zeta_2([w_{s,m+n}, w_{r,0}]) \\
        &= \left( \frac{\{rm\}}{\{rn\}} \{(r+s)n\} - \frac{\{sn\}}{\{rn\}} \{r(m+n)\} \right) \zeta_2(w_{r+s,m+n}) \\
        &= \{rm-sn\} \zeta_2(w_{r+s,m+n}) \\
        &\overset{\mathclap{\cref{sloth}} }{=}\ \zeta_2([w_{r,n}, w_{s,m}]),
    \end{align*}
    where, in the second equality, we used the previous cases.
\end{proof}

Define the linear map $\zeta \colon \fEHc \to \widehat{\fEH}$ by
\begin{equation} \label{horse}
    \begin{aligned}
        \zeta(x) &= \zeta_1(x) + \zeta_2(x),& x \in \fEH,\\
        \zeta((a,b)) &= a \vartheta(w_{1,0}, w_{-1,0}) + b \vartheta(w_{0,1}, w_{0,-1}),& (a,b) \in \bZ_\kk.
    \end{aligned}
\end{equation}

\begin{lem} \label{reef}
    The map $\zeta$ is a homomorphism of Lie algebras.
\end{lem}

\begin{proof}
    To avoid potential confusion with the Lie bracket on $\fEH$, we denote the Lie bracket on $\fEHc$ by $[\cdot,\cdot]'$ in this proof.  If $x,y \in \fEHc$ with $x \in \bZ_\kk$ or $y \in \bZ_\kk$, then
    \[
        \zeta([x,y]') = 0 = [\zeta(x), \zeta(y)].
    \]

    It remains to show that $\zeta([w_\bx,w_\by]) = [\zeta(w_\bx), \zeta(w_\by)]$ for $\bx,\by \in \bZ^*$.  We have
    \[
        \zeta([w_\bx,w_\by]')
        \overset{\cref{phoenix}}{=} \zeta([w_\bx,w_\by] + \delta_{\bx,-\by} \bx)
        = \zeta_1([w_\bx,w_\by]) + \zeta_2([w_\bx,w_\by]) + \delta_{\bx,-\by} \zeta(\bx).
    \]
    If $\bx + \by \ne 0$, then, by \cref{fairy}, we have
    \[
        \zeta([w_\bx,w_\by]')
        = \zeta_1([w_\bx,w_\by]) + \vartheta(w_\bx,w_\by)
        \overset{\cref{magpie}}{=} [\zeta_1(w_\bx), \zeta_1(w_\by)]
        = [\zeta(w_\bx), \zeta(w_\by)],
    \]
    where the last equality follows from the fact that the image of $\zeta_2$ is contained in $\mathfrak{Z}$, and hence in the center of $\widehat{\fEH}$.  On the other hand, if $\bx + \by = 0$, with $\bx =(r,n)$, then
    \begin{align*}
        \zeta([w_\bx, w_{-\bx}]')
        &= \zeta(\bx) \\
        &= r \vartheta(w_{1,0}, w_{-1,0}) + n \vartheta(w_{0,1}, w_{0,-1}) \\
        &= \vartheta(w_\bx, w_{-\bx}) & \text{(by \cref{milk})} \\
        &\overset{\mathclap{\cref{magpie}}}{=}\ [\zeta_1(w_\bx), \zeta_1(w_{-\bx})] \\
        &= [\zeta(w_\bx), \zeta(w_{-\bx})],
    \end{align*}
    where the last equality again follows from the fact that the image of $\zeta_2$ is contained in $\mathfrak{Z}$.
\end{proof}

We are now ready to prove \cref{uce}.

\begin{proof}[Proof of \cref{uce}]
    Suppose we have a central extension as in \cref{archer}.  By \cref{reef}, the map $\zeta$ defined by \cref{horse} is a homomorphism of Lie algebras.  For $\bx \in \bZ^*$ and $z \in \bZ_\kk$, we have
    \[
        (\hat{\pi}\zeta)(w_\bx)
        = \hat{\pi}(\zeta_1(w_\bx) + \zeta_2(w_\bx))
        = (\hat{\pi} \zeta_1)(w_\bx)
        = w_\bx
        = \tilde{\pi}(w_\bx),
        \qquad
        (\hat{\pi}\zeta)(z)
        = 0
        = \tilde{\pi}(z).
    \]
    Hence $\hat{\pi} \zeta = \tilde{\pi}$, and so $\zeta$ is a morphism of extensions, as desired.

    It remains to show uniqueness of $\zeta$.  Suppose we have another homomorphism of Lie algebra $\zeta' \colon \fEHc \to \widehat{\fEH}$ such that $\hat{\pi} \zeta' = \tilde{\pi}$.  Let $z \in \fEHc$.  Since $\fEHc$ is easily seen to be perfect (that is, $[\fEHc, \fEHc] = \fEHc$), there exist $x,y \in \fEHc$ such that $z=[x,y]$.  Then we have
    \[
        \zeta'(x) - \zeta(x),\ \zeta'(y) - \zeta(y) \in \ker(\hat{\pi}) = \mathfrak{Z}.
    \]
    Since $\mathfrak{Z}$ is contained in the center of $\widehat{\fEH}$, this implies that
    \[
        \zeta'(z)
        = [\zeta'(x), \zeta'(y)]
        = [\zeta(x), \zeta(y)]
        = \zeta(z).
    \]
    Thus $\zeta' = \zeta$, as desired.
\end{proof}

%=====================================================================================
\section{Relation to the elliptic Hall algebra of Burban and Schiffmann\label{sec:BS}}
%=====================================================================================

In this section we show (in \cref{hyperion}) how the central reductions defined in \cref{subsec:reduct} are specializations of a central extension of the elliptic Hall algebra of Burban and Schiffmann \cite{BS12}.  This relationship is not used elsewhere in the paper.  The fundamental ingredient here is the work of Morton and Samuelson \cite{MS17}, who described an isomorphism between the HOMFLYPT skein algebra of the torus and the elliptic Hall algebra.  This corresponds to the case of central charge $k=0$.  The case $k=-1$ was treated in \cite{CLLSS18}; see \cref{nova}.

We work here over the field $\C(v,q)$ of rational functions in two indeterminates.  We first recall from \cite[Def.~6.4]{BS12} the definition of the central extension $\BSc$ of the elliptic Hall algebra.  This central extension is denoted $\widetilde{\mathcal{E}}_\mathbf{K}$ in \cite{BS12}, where $\mathbf{K} = \C(v,q)$.  Our $v$ and $q$ are denoted $\sigma^{1/2}$ and $\bar{\sigma}^{1/2}$ in \cite{BS12}.  For $\bx,\by \in \bZ^*$, we define
\begin{align*}
    \epsilon_\bx =
    \begin{cases}
        1 & \text{if } \bx \in \bZ^+, \\
        -1 & \text{if } \bx \in \bZ^-,
    \end{cases}
    \quad \text{and} \quad
    \epsilon_{\bx,\by} = \operatorname{sign}(\det \begin{pmatrix} \bx & \by \end{pmatrix}) \in \{\pm 1\},\quad \text{if $\bx,\by$ are not colinear}.
\end{align*}
For $a \in \C(v,q)^\times$ and $d \in \Z$, define
\[
    \{d\}_a := a^d - a^{-d}
    \quad \text{and} \quad
    [d]_a := \frac{\{d\}_a}{\{1\}_a} = \frac{a^d-a^{-d}}{a-a^{-1}}.
\]
These are both elements of $\Z[a^{\pm 1}] \subseteq \C(v,q)$.  Note also that $[d]_1 = d$.  For $d \ge 1$, we define
\[
    \alpha_d := \frac{1}{d} (1-v^{2d}) (1-q^{2d}) (1 - (vq)^{-2d})
    = \frac{1}{d} \{d\}_v \{d\}_q \{d\}_{vq}.
\]

Define $\BSc$ to be the $\C(v,q)$-algebra with generators
\[
    \kappa_\bx,\ \bx \in \bZ,\quad
    u_\bx,\ \bx \in \bZ^*,
\]
modulo the following relations:
\begin{enumerate}
    \item The $\kappa_\bx$ are central, and we have
        \[
            \kappa_{(0,0)} = 1,\quad
            \kappa_\bx \kappa_\by = \kappa_{\bx + \by}.
        \]

    \item If $\bx, \by \in \bZ^*$ are collinear, then
        \[
            [u_\by, u_\bx]
            = \delta_{\bx, -\by} \frac{\kappa_\bx - \kappa_{\bx}^{-1}}{\alpha_{\gcd(\bx)}},
        \]
        where $\gcd(\bx)$ denotes the greatest common denominator of the components of $\bx$.

    \item If $\bx, \by \in \bZ^*$ are not colinear, $\gcd(\bx)=1$, and the triangle in $\bZ$ with vertices $\{(0,0), \bx, \bx + \by\}$ has no element of $\bZ$ in its interior, then
        \[
            [u_\by, u_\bx]
            = \epsilon_{\bx,\by} \kappa_{\alpha(\bx,\by)} \frac{\theta_{\bx+\by}}{\alpha_1},
        \]
        where
        \[
            \alpha(\bx,\by) =
            \begin{cases}
                \epsilon_\bx (\epsilon_\bx \bx + \epsilon_\by \by - \epsilon_{\bx + \by} (\bx + \by))/2 & \text{if } \epsilon_{\bx,\by} = 1,
                \\
                \epsilon_\by (\epsilon_\bx \bx + \epsilon_\by \by - \epsilon_{\bx + \by} (\bx + \by))/2 & \text{if } \epsilon_{\bx,\by} = -1,
            \end{cases}
        \]
        and where the elements $\theta_\bz$, $\bz \in \bZ^*$, are determined by
        \[
            \sum_{i \ge 1} \theta_{i \bx_0} w^i
            = \exp \left( \sum_{r \ge 1} \alpha_r u_{r \bx_0} w^r \right)
        \]
        for any $\bx_0 \in \bZ^*$ such that $\gcd(\bx_0)=1$.  Here $w$ is a formal variable.
\end{enumerate}
The relations imply that the $\C(v,q)$-subalgebra $K$ generated by the $\kappa_\bx$, $\bx \in \bZ$, is isomorphic to the group algebra, over $\C(v,q)$ of the abelian group $\bZ$, and that $\BSc$ is naturally a $K$-algebra.

Fix a $\Z$-linear map $\lambda \colon \bZ \to \Z$ and define $\BSc_\lambda$ to be the $\C(v,q)$-algebra obtained from $\BSc$ by imposing the additional relations
\[
    \kappa_\bx = (vq)^{\lambda(\bx)},\quad
    \bx \in \bZ.
\]
Now define the following elements of $\BSc_\lambda$:
\[
    w_\bx := \{\gcd(\bx)\}_v u_\bx,\quad \bx \in \bZ^*.
\]
Let $\BSc'_\lambda$ be the $\C[v,q,\{d\}_v^{-1},\{d\}_q^{-1} : d \ge 1]$-subalgebra of $\BSc_\lambda$ generated by the $w_\bx$, $\bx \in \bZ^*$.  Thus $\BSc'_\lambda$ is the $\C[v,q,\{d\}_v^{-1},\{d\}_q^{-1} : d \ge 1]$-algebra generated by $w_\bx$, $\bx \in \bZ^*$, subject to the following relations:
\begin{enumerate}
    \item If $\bx, \by \in \bZ^*$ are collinear, then
        \[
            [w_\by, w_\bx]
            = \delta_{\bx,-\by} d \frac{\{d\}_v}{\{d\}_q} \left[ \frac{\lambda(\bx)}{d} \right]_{(vq)^d},\quad
            \text{where } d = \gcd(\bx).
        \]

    \item If $\bx, \by \in \bZ^*$ are such that $\gcd(\bx)=1$ and the triangle in $\bZ$ with vertices $\{(0,0), \bx, \bx + \by\}$ has no interior lattice point, then
        \[
            [w_\by, w_\bx]
            = \epsilon_{\bx,\by} \kappa_{\alpha(\bx,\by)} \{1\}_v \{\gcd(\by)\}_v \frac{\theta_{\bx+\by}}{\alpha_1}.
        \]
\end{enumerate}

Now, by \cite[Lem.~5.4]{MS17}, we have, for $\bx \in \bZ^*$,
\[
    \frac{\theta_\bx}{\alpha_1}
    = \left( [\gcd(\bx)]_v \right)^2 u_\bx
    = -\frac{[\gcd(\bx)]_q}{\{1\}_q} w_\bx
    \quad \text{when } q=v^{-1}.
\]
(Note that our $v$ and $q$ are the $q^{1/2}$ and $t^{-1/2}$ of \cite{MS17}, respectively.)  Let $\BS_\lambda := \BSc_\lambda / (vq-1)$.  Thus, setting $\kk = \C[q^{\pm 1},\{d\}^{-1} : d \ge 1]$, we see that $\BS_\lambda$ is the $\kk$-algebra generated by $w_\bx$, $\bx \in \bZ^*$, subject to the following relations:
\begin{enumerate}
    \item If $\bx, \by \in \bZ^*$ are collinear, then
        \begin{equation} \label{BS1}
            [w_\bx, w_\by]
            = \delta_{\bx,-\by} \lambda(\bx).
        \end{equation}

    \item If $\bx, \by \in \bZ^*$ are such that $\gcd(\bx)=1$ and the triangle in $\bZ$ with vertices $\{(0,0), \bx, \bx + \by\}$ has no element of $\bZ$ in its interior, then
        \begin{equation} \label{BS2}
            [w_\bx, w_\by]
            = \epsilon_{\bx,\by} \{\gcd(\by)\}_q [\gcd(\bx+\by)]_q w_{\bx + \by}.
        \end{equation}
\end{enumerate}

\begin{prop} \label{hyperion}
    We have an isomorphism of $\kk$-algebras
    \[
        \EH_\lambda \xrightarrow{\cong} \BS_\lambda,\quad w_\bx \mapsto w_\bx,\ \bx \in \bZ^*.
    \]
\end{prop}

\begin{proof}
    When $\lambda=0$, this is precisely \cite[Th.~5.6]{MS17} after recalling that the $s,v,q,t$ of \cite{MS17} are $q,t^{-1},v^2,q^{-2}$ in our notation.  To prove the result for general $\lambda$, we make the dependence on $\lambda$ explicit by letting $[\cdot,\cdot]_\lambda$ denote the bracket on $\BS_\lambda$ given by \cref{BS1,BS2}.  Then we have
    \begin{equation} \label{kitty}
        [w_\bx, w_\by]_\lambda = [w_\bx, w_\by]_0 + \delta_{\bx,-\by} \lambda(\bx).
    \end{equation}
    Comparing to \cref{phoenix,pendulum}, we see that this is precisely the relationship between the bracket in $\EH_0$ and the one in $\EH_\lambda$.
\end{proof}

\begin{rem} \label{nova}
    When $\lambda=0$, $\EH_0$ is the elliptic Hall algebra (no central extension) denoted $\mathcal{E}_{\sigma,\bar{\sigma}}$ in \cite{BS12}, specialized at $\sigma^{-1/2} = q = \bar{\sigma}^{1/2}$.  When $\lambda = \lambda_{-1}$, $\EH_{-1}$ is the algebra denoted $\mathbb{E}$ in \cite[Def.~4.4]{CLLSS18}.
\end{rem}

%=============
% Bibliography
%=============

\bibliographystyle{alphaurl}
\bibliography{EHAcat}

\newcommand{\etalchar}[1]{$^{#1}$}
\begin{thebibliography}{BHLW17}

\bibitem[AK94]{AK94}
S.~Ariki and K.~Koike.
\newblock A {H}ecke algebra of {$({\Z}/r{\Z})\wr{\mathfrak{S}}_n$} and
  construction of its irreducible representations.
\newblock {\em Adv. Math.}, 106(2):216--243, 1994.
\newblock \href {https://doi.org/10.1006/aima.1994.1057}
  {\path{doi:10.1006/aima.1994.1057}}.

\bibitem[BGHL14]{BGHL14}
A.~Beliakova, Z.~Guliyev, K.~Habiro, and A.~D. Lauda.
\newblock Trace as an alternative decategorification functor.
\newblock {\em Acta Math. Vietnam.}, 39(4):425--480, 2014.
\newblock \href {http://arxiv.org/abs/1409.1198} {\path{arXiv:1409.1198}},
  \href {https://doi.org/10.1007/s40306-014-0092-x}
  {\path{doi:10.1007/s40306-014-0092-x}}.

\bibitem[BHLW17]{BHLW17}
A.~Beliakova, K.~Habiro, A.~D. Lauda, and B.~Webster.
\newblock Current algebras and categorified quantum groups.
\newblock {\em J. Lond. Math. Soc. (2)}, 95(1):248--276, 2017.
\newblock \href {http://arxiv.org/abs/1412.1417} {\path{arXiv:1412.1417}},
  \href {https://doi.org/10.1112/jlms.12001} {\path{doi:10.1112/jlms.12001}}.

\bibitem[Bru16]{Bru16}
J.~Brundan.
\newblock On the definition of {K}ac-{M}oody 2-category.
\newblock {\em Math. Ann.}, 364(1-2):353--372, 2016.
\newblock \href {http://arxiv.org/abs/1501.00350} {\path{arXiv:1501.00350}},
  \href {https://doi.org/10.1007/s00208-015-1207-y}
  {\path{doi:10.1007/s00208-015-1207-y}}.

\bibitem[BS12]{BS12}
I.~Burban and O.~Schiffmann.
\newblock On the {H}all algebra of an elliptic curve, {I}.
\newblock {\em Duke Math. J.}, 161(7):1171--1231, 2012.
\newblock \href {http://arxiv.org/abs/math/0505148}
  {\path{arXiv:math/0505148}}, \href {https://doi.org/10.1215/00127094-1593263}
  {\path{doi:10.1215/00127094-1593263}}.

\bibitem[BSW18]{BSW-K0}
J.~Brundan, A.~Savage, and B.~Webster.
\newblock The degenerate {H}eisenberg category and its {G}rothendieck ring.
\newblock 2018.
\newblock \href {http://arxiv.org/abs/1812.03255} {\path{arXiv:1812.03255}}.

\bibitem[BSW20a]{BSW-HKM}
J.~Brundan, A.~Savage, and B.~Webster.
\newblock Heisenberg and {K}ac-{M}oody categorification.
\newblock {\em Selecta Math. (N.S.)}, 26(5):Paper No. 74, 62, 2020.
\newblock \href {http://arxiv.org/abs/1907.11988} {\path{arXiv:1907.11988}},
  \href {https://doi.org/10.1007/s00029-020-00602-5}
  {\path{doi:10.1007/s00029-020-00602-5}}.

\bibitem[BSW20b]{BSW-qheis}
J.~Brundan, A.~Savage, and B.~Webster.
\newblock On the definition of quantum {H}eisenberg category.
\newblock {\em Algebra Number Theory}, 14(2):275--321, 2020.
\newblock \href {http://arxiv.org/abs/1812.04779} {\path{arXiv:1812.04779}},
  \href {https://doi.org/10.2140/ant.2020.14.275}
  {\path{doi:10.2140/ant.2020.14.275}}.

\bibitem[BSW20c]{BSW-qFrobHeis}
J.~Brundan, A.~Savage, and B.~Webster.
\newblock Quantum {F}robenius {H}eisenberg categorification.
\newblock 2020.
\newblock \href {http://arxiv.org/abs/2009.06690} {\path{arXiv:2009.06690}}.

\bibitem[CLL{\etalchar{+}}18]{CLLSS18}
S.~Cautis, A.~D. Lauda, A.~M. Licata, P.~Samuelson, and J.~Sussan.
\newblock The elliptic {H}all algebra and the deformed {K}hovanov {H}eisenberg
  category.
\newblock {\em Selecta Math. (N.S.)}, 24(5):4041--4103, 2018.
\newblock \href {http://arxiv.org/abs/1609.03506} {\path{arXiv:1609.03506}},
  \href {https://doi.org/10.1007/s00029-018-0429-8}
  {\path{doi:10.1007/s00029-018-0429-8}}.

\bibitem[DI97]{DI97}
J.~Ding and K.~Iohara.
\newblock Generalization of {D}rinfeld quantum affine algebras.
\newblock {\em Lett. Math. Phys.}, 41(2):181--193, 1997.
\newblock \href {https://doi.org/10.1023/A:1007341410987}
  {\path{doi:10.1023/A:1007341410987}}.

\bibitem[FFJ{\etalchar{+}}11]{FFJMM11}
B.~Feigin, E.~Feigin, M.~Jimbo, T.~Miwa, and E.~Mukhin.
\newblock Quantum continuous {$\mathfrak{gl}_\infty$}: semiinfinite
  construction of representations.
\newblock {\em Kyoto J. Math.}, 51(2):337--364, 2011.
\newblock \href {http://arxiv.org/abs/1002.3100} {\path{arXiv:1002.3100}},
  \href {https://doi.org/10.1215/21562261-1214375}
  {\path{doi:10.1215/21562261-1214375}}.

\bibitem[FT11]{FT11}
B.~L. Feigin and A.~I. Tsymbaliuk.
\newblock Equivariant {$K$}-theory of {H}ilbert schemes via shuffle algebra.
\newblock {\em Kyoto J. Math.}, 51(4):831--854, 2011.
\newblock \href {http://arxiv.org/abs/0904.1679} {\path{arXiv:0904.1679}},
  \href {https://doi.org/10.1215/21562261-1424875}
  {\path{doi:10.1215/21562261-1424875}}.

\bibitem[KL10]{KL10}
M.~Khovanov and A.~D. Lauda.
\newblock A categorification of quantum {${\mathrm{sl}}(n)$}.
\newblock {\em Quantum Topol.}, 1(1):1--92, 2010.
\newblock \href {http://arxiv.org/abs/0807.3250} {\path{arXiv:0807.3250}},
  \href {https://doi.org/10.4171/QT/1} {\path{doi:10.4171/QT/1}}.

\bibitem[LS13]{LS13}
A.~Licata and A.~Savage.
\newblock Hecke algebras, finite general linear groups, and {H}eisenberg
  categorification.
\newblock {\em Quantum Topol.}, 4(2):125--185, 2013.
\newblock \href {http://arxiv.org/abs/1101.0420} {\path{arXiv:1101.0420}},
  \href {https://doi.org/10.4171/QT/37} {\path{doi:10.4171/QT/37}}.

\bibitem[LT05]{LT05}
W.~Lin and S.~Tan.
\newblock Central extensions and derivations of the {L}ie algebras of skew
  derivations for the quantum torus.
\newblock {\em Comm. Algebra}, 33(11):3919--3938, 2005.
\newblock \href {https://doi.org/10.1080/00927870500261066}
  {\path{doi:10.1080/00927870500261066}}.

\bibitem[Mac95]{Mac95}
I.~G. Macdonald.
\newblock {\em Symmetric functions and {H}all polynomials}.
\newblock Oxford Mathematical Monographs. The Clarendon Press, Oxford
  University Press, New York, second edition, 1995.
\newblock With contributions by A. Zelevinsky, Oxford Science Publications.

\bibitem[Mik07]{Mik07}
K.~Miki.
\newblock A {$(q,\gamma)$} analog of the {$W_{1+\infty}$} algebra.
\newblock {\em J. Math. Phys.}, 48(12):123520, 35, 2007.
\newblock \href {https://doi.org/10.1063/1.2823979}
  {\path{doi:10.1063/1.2823979}}.

\bibitem[Mor02]{Mor02}
H.~R. Morton.
\newblock Skein theory and the {M}urphy operators.
\newblock volume~11, pages 475--492. 2002.
\newblock Knots 2000 Korea, Vol. 2 (Yongpyong).
\newblock \href {http://arxiv.org/abs/math/0102098}
  {\path{arXiv:math/0102098}}, \href
  {https://doi.org/10.1142/S0218216502001767}
  {\path{doi:10.1142/S0218216502001767}}.

\bibitem[MS17]{MS17}
H.~Morton and P.~Samuelson.
\newblock The {HOMFLYPT} skein algebra of the torus and the elliptic {H}all
  algebra.
\newblock {\em Duke Math. J.}, 166(5):801--854, 2017.
\newblock \href {http://arxiv.org/abs/1410.0859} {\path{arXiv:1410.0859}},
  \href {https://doi.org/10.1215/00127094-3718881}
  {\path{doi:10.1215/00127094-3718881}}.

\bibitem[MS21]{MS20}
Y.~Mousaaid and A.~Savage.
\newblock Affinization of monoidal categories.
\newblock {\em J. \'{E}c. polytech. Math.}, 8:791--829, 2021.
\newblock \href {http://arxiv.org/abs/2010.13598} {\path{arXiv:2010.13598}},
  \href {https://doi.org/10.5802/jep.158} {\path{doi:10.5802/jep.158}}.

\bibitem[Neg14]{Neg14}
A.~Negut.
\newblock The shuffle algebra revisited.
\newblock {\em Int. Math. Res. Not. IMRN}, (22):6242--6275, 2014.
\newblock \href {http://arxiv.org/abs/1209.3349} {\path{arXiv:1209.3349}},
  \href {https://doi.org/10.1093/imrn/rnt156} {\path{doi:10.1093/imrn/rnt156}}.

\bibitem[Neg15]{Neg15}
A.~Negu\c{t}.
\newblock Moduli of flags of sheaves and their {$K$}-theory.
\newblock {\em Algebr. Geom.}, 2(1):19--43, 2015.
\newblock \href {http://arxiv.org/abs/1209.4242} {\path{arXiv:1209.4242}},
  \href {https://doi.org/10.14231/AG-2015-002}
  {\path{doi:10.14231/AG-2015-002}}.

\bibitem[Rou08]{Rou08}
R.~Rouquier.
\newblock 2-{K}ac-{M}oody algebras.
\newblock 2008.
\newblock \href {http://arxiv.org/abs/0812.5023} {\path{arXiv:0812.5023}}.

\bibitem[RS]{RS20}
M.~Reeks and A.~Savage.
\newblock Frobenius {W}-algebras and traces of {F}robenius {H}eisenberg
  categories.
\newblock {\em J. Pure Appl. Algebra}.
\newblock To appear.
\newblock \href {http://arxiv.org/abs/2007.02732} {\path{arXiv:2007.02732}},
  \href {https://doi.org/10.1142/S0219498822500359}
  {\path{doi:10.1142/S0219498822500359}}.

\bibitem[SV11]{SV11}
O.~Schiffmann and E.~Vasserot.
\newblock The elliptic {H}all algebra, {C}herednik {H}ecke algebras and
  {M}acdonald polynomials.
\newblock {\em Compos. Math.}, 147(1):188--234, 2011.
\newblock \href {http://arxiv.org/abs/0802.4001} {\path{arXiv:0802.4001}},
  \href {https://doi.org/10.1112/S0010437X10004872}
  {\path{doi:10.1112/S0010437X10004872}}.

\bibitem[SV13]{SV13}
O.~Schiffmann and E.~Vasserot.
\newblock The elliptic {H}all algebra and the {$K$}-theory of the {H}ilbert
  scheme of {$\Bbb A^2$}.
\newblock {\em Duke Math. J.}, 162(2):279--366, 2013.
\newblock \href {http://arxiv.org/abs/0905.2555} {\path{arXiv:0905.2555}},
  \href {https://doi.org/10.1215/00127094-1961849}
  {\path{doi:10.1215/00127094-1961849}}.

\bibitem[SVV17]{SVV17}
P.~Shan, M.~Varagnolo, and E.~Vasserot.
\newblock On the center of quiver {H}ecke algebras.
\newblock {\em Duke Math. J.}, 166(6):1005--1101, 2017.
\newblock \href {http://arxiv.org/abs/1411.4392} {\path{arXiv:1411.4392}},
  \href {https://doi.org/10.1215/00127094-3792705}
  {\path{doi:10.1215/00127094-3792705}}.

\bibitem[Tur88]{Tur88}
V.~G. Turaev.
\newblock The {C}onway and {K}auffman modules of a solid torus.
\newblock {\em Zap. Nauchn. Sem. Leningrad. Otdel. Mat. Inst. Steklov. (LOMI)},
  167(Issled. Topol. 6):79--89, 190, 1988.
\newblock \href {https://doi.org/10.1007/BF01099241}
  {\path{doi:10.1007/BF01099241}}.

\end{thebibliography}

\end{document}